\documentclass[10pt,reqno]{amsart}

\usepackage[margin=1in]{geometry}
\usepackage{amsfonts,amsmath,amssymb,enumerate, amscd}    
\usepackage{amsthm}    
\usepackage{mathrsfs}  
\usepackage{setspace}
\usepackage{stmaryrd}
\usepackage{tikz}
\usepackage{genyoungtabtikz}
\usepackage{mathtools}

\onehalfspace
\numberwithin{equation}{section}
\theoremstyle{plain}
\newtheorem{thm}{Theorem}[section]
\newtheorem{lem}[thm]{Lemma}
\newtheorem{prop}[thm]{Proposition}
\newtheorem{cor}[thm]{Corollary}
\theoremstyle{definition}
\newtheorem{defn}[thm]{Definition}
\newtheorem{defprop}[thm]{Definition-Proposition}
\newtheorem{conj}[thm]{Conjecture}
\newtheorem{exmp}[thm]{Example}
\theoremstyle{remark}
\newtheorem{rem}[thm]{Remark}

\newtheorem*{rem*}{Remark}

\newcommand{\be}{\begin{equation}}    
\newcommand{\ee}{\end{equation}}    
\newcommand{\beu}{\begin{equation*}}    
\newcommand{\eeu}{\end{equation*}}    
\newcommand{\bea}{\begin{eqnarray}}    
\newcommand{\eea}{\end{eqnarray}}    
\newcommand{\beaa}{\begin{eqnarray*}}    
\newcommand{\eeaa}{\end{eqnarray*}}    
\newcommand{\bmx}{\begin{pmatrix}}    
\newcommand{\emx}{\end{pmatrix}}    
\newcommand{\btz}{\begin{tikzpicture}}
\newcommand{\etz}{\end{tikzpicture}}
    
\newcommand{\Aut}[1]{\mathrm{Aut}(#1)}

\newcommand{\g}{{\mathfrak g}}    
    
\newcommand{\m}{{\mathfrak m}}

\newcommand{\nn}{\nonumber}    
\newcommand{\sign}{{\rm sign}}    

\newcommand{\Z}{{\mathbb Z}}
\newcommand{\N}{{\mathbb N}}
\newcommand{\C}{{\mathbb C}}

\newcommand{\Q}{{\mathbb Q}}

\newcommand{\pair}[2]{\left \langle #1, #2\right\rangle}

\newcommand{\gh}{\dot{\g}}
    
\newcommand{\uq}{{\mathrm{U}_q}}
\newcommand{\uqd}{\dot{\mathrm{U}}_q}
\newcommand{\uqdd}{\ddot{\mathrm{U}}_q}
\newcommand{\qaff}{\dot{\mathrm{U}}_q}
\newcommand{\qdaff}{\ddot{\mathrm{U}}_q}
\newcommand{\dqaslt}{\qdaff(\mathfrak{sl}_2)}

\newcommand{\btp}{\begin{tikzpicture}[baseline=0pt,scale=0.9,line width=0.25pt]}    
\newcommand{\etp}{\end{tikzpicture}}

\newcommand{\range}[2]{\llbracket #1,#2 \rrbracket}

\newcommand{\atp}[1]{}

\newcommand{\ie}{i.e. }
\newcommand{\eg}{e.g. }

\newcommand{\uqslthh}{\dot{\mathrm{U}}_q  ( \dot{\mathfrak{a}}_1 ) }

\newcommand{\uqhhzeroslt}{\ddot{\mathrm{U}}_q^0  ({\mathfrak{a}}_1 ) }
\newcommand{\uqmslthh}{\dot{\mathrm{U}}_q^-  ( \dot{\mathfrak{a}}_1 ) }
\newcommand{\uqpslthh}{\dot{\mathrm{U}}_q^+  ( \dot{\mathfrak{a}}_1 ) }
\newcommand{\uqpmslthh}{\dot{\mathrm{U}}_q^\pm  ( \dot{\mathfrak{a}}_1 ) }
\newcommand{\uqzeroslthh}{\dot{\mathrm{U}}_q^0  ( \dot{\mathfrak{a}}_1 ) }
\newcommand{\uqsltc}{\widehat{\uqslthh}}
\newcommand{\uqgslthh}{\dot{\mathrm{U}}_q^\geq  ( \dot{\mathfrak{a}}_1 ) }
\newcommand{\uqlslthh}{\dot{\mathrm{U}}_q^\leq  ( \dot{\mathfrak{a}}_1) }

\newcommand{\br}{\ddot{\mathfrak B}}
\newcommand{\xp}[1]{{\bf x}_{#1}^+}
\newcommand{\xm}[1]{{\bf x}_{#1}^-}
\newcommand{\kp}[1]{{\bf k}_{#1}^+}
\newcommand{\km}[1]{{\bf k}_{#1}^-}

\newcommand{\x}[2]{{\bf x}_{#2}^{#1}}
\newcommand{\kk}[2]{{\bf k}_{#2}^{#1}}
\newcommand{\Xp}[1]{{\bf X}_{#1}^+}
\newcommand{\Xm}[1]{{\bf X}_{#1}^-}
\newcommand{\KK}[2]{{\bf K}_{#2}^{#1}}

\newcommand{\X}[2]{{\bf X}_{#2}^{#1}}
\newcommand{\Xsf}[2]{{\mathsf X}_{#2}^{#1}}
\newcommand{\Xbsf}[2]{{\textbf{\textsf{X}}}_{#2}^{#1}}
\newcommand{\tsf}[2]{{\mathsf t}^{#1}_{#2}}
\newcommand{\psf}[2]{{\mathsf p}^{#1}_{#2}}

\newcommand{\Ksf}[2]{{\mathsf K}^{#1}_{#2}}
\newcommand{\csf}[2]{{\mathsf c}^{#1}_{#2}}
\newcommand{\tbsf}[2]{{\textbf{\textsf{t}}}^{#1}_{#2}}
\newcommand{\pbsf}[1]{{\textbf{\textsf{p}}}^{#1}}

\newcommand{\Kbsf}[2]{{\textbf{\textsf{K}}}^{#1}_{#2}}
\newcommand{\cbsf}[1]{{\textbf{\textsf{c}}}^{#1}}
\newcommand{\eb}[1]{{\textbf{{e}}}^{#1}}
\newcommand{\psib}[1]{{\boldsymbol \psi}^{#1}}
\newcommand{\Csf}{{{\textsf{C}}}}
\newcommand{\Dsf}{{{\textsf{D}}}}
\newcommand{\Gam}[1]{{\bf \Gamma}_0^{#1}}
\newcommand{\bwp}[1]{{\boldsymbol \wp}^{#1}}
\newcommand{\bpsi}[2]{{\boldsymbol \psi}^{#1}_{#2}}
\newcommand{\bz}[2]{{\boldsymbol \Gamma}^{#1}_{#2}}
\newcommand{\K}{\mathbb K}
\newcommand{\F}{\mathbb F}
\newcommand{\isom}{ \stackrel{\sim}{\longrightarrow} }
\newcommand{\noi}{\noindent}
\DeclareMathOperator*{\res}{res}
\DeclareMathOperator*{\card}{card}
\newcommand{\supp}{\mathrm{supp}}
\newcommand{\E}[3]{\mathcal E_{#1,#2,#3}}
\newcommand{\bA}{{\bf{A}}}
\newcommand{\bn}{{\bf{n}}}
\newcommand{\rran}[1]{{\llbracket #1 \rrbracket}}
\newcommand{\Prt}[2]{{\textsf{P}}^{#1}_{#2}}
\newcommand{\Trt}[2]{{\textsf{T}}^{#1}_{#2}}

\makeatletter
\newbox\xrat@below
\newbox\xrat@above
\newcommand{\xrightarrowtail}[2][]{%
  \setbox\xrat@below=\hbox{\ensuremath{\scriptstyle #1}}%
  \setbox\xrat@above=\hbox{\ensuremath{\scriptstyle #2}}%
  \pgfmathsetlengthmacro{\xrat@len}{max(\wd\xrat@below,\wd\xrat@above)+.6em}%
  \mathrel{\tikz [>->,baseline=-.75ex]
                 \draw (0,0) -- node[below=-2pt] {\box\xrat@below}
                                node[above=-2pt] {\box\xrat@above}
                       (\xrat@len,0) ;}}
\makeatother

\author{E. Mounzer}

\author{R. Zegers}
\address{\vspace{-.30cm} \hspace{-.20cm}
Laboratoire de Physique Th\'eorique (UMR8627)}
\address{\vspace{-.3cm}CNRS, Univ. Paris-Sud, Universit\'e Paris-Saclay;}
\address{\vspace{-.08cm} \hspace{-.20cm}
91405 Orsay, France \vspace{.3cm}}
\email{robin.zegers@th.u-psud.fr}

\begin{document} 
{\flushright{\small{LPT-Orsay-19-09
}}\\}

\title[On double quantum affinization]{On double quantum affinization:\\
1. Type $\mathfrak a_1$}

\begin{abstract}
We define the double quantum affinization $\qdaff(\mathfrak a_1)$ of type $\mathfrak{a}_1$ as a topological Hopf algebra. We prove that it admits a subalgebra $\qdaff'(\mathfrak a_1)$ whose completion is (bicontinuously) isomorphic to the completion of the quantum toroidal algebra $\uqslthh$, defined as the (simple) quantum affinization of the untwisted affine Ka\v c-Moody Lie algebra $\dot{\mathfrak{sl}}_2$ of type $\dot{\mathfrak a}_1$, equipped with a certain topology inherited from its natural $\Z$-grading. The isomorphism is constructed by means of a bicontinuous action by automorphisms of an affinized version $\br$ -- technically a split extension $\br \cong \dot{\mathfrak B} \ltimes P^\vee$ by the coweight lattice $P^\vee$ -- of the affine braid group $\dot{\mathfrak B}$ of type $\dot{\mathfrak a}_1$ on that completion of $\uqslthh$. It can be regarded as an affinized version of the Damiani-Beck isomorphism, familiar from the quantum affine setting. We eventually prove the corresponding triangular decomposition of $\qdaff(\mathfrak a_1)$ and briefly discuss the consequences regarding the representation theory of quantum toroidal algebras.\end{abstract}

\maketitle

\section{Introduction}
Let $\g$ be a simple Lie algebra and denote by $\gh$ the corresponding untwisted affine Ka\v c-Moody algebra. Starting from $\g$ and $\gh$ or from their respective root systems, one can construct two a priori different algebras: on one hand, the quantum affine algebra $\uq(\gh)$ is the standard Drinfel'd-Jimbo algebra associated with $\gh$; whereas on the other hand, the quantum affinization $\uqd(\g)$ of $\g$, which we define as $\uq(\gh)$ in its Drinfel'd current presentation, is associated with the simple finite root system of $\g$. Now $\uqd(\g)$ and $\uq(\gh)$ are isomorphic by virtue of a theorem established by Damiani and Beck, \cite{Damiani, Beck}, which can be regarded as a quantum version of the classic result that each affine Lie algebra is isomorphic to the corresponding untwisted affine Ka\v c-Moody Lie algebra. The situation can be summarized by the following diagram
\be\begin{CD}
\g @>\text{\footnotesize Classical Affinization}>> \gh\\
@V\text{\footnotesize Quantum Affinization}VV @VV\text{\footnotesize Quantization}V\\
\qaff(\g) @>\sim>\text{\footnotesize Damiani-Beck isom.}> \uq(\gh)
\end{CD}\nn
\ee
It turns out that quantum affinization still makes sense for the already affine Lie algebra $\gh$, thus yielding a doubly affine quantum algebra known as the \emph{quantum toroidal algebra} $\uqd(\gh)$. These originally appeared in type $\mathfrak{a}_n$ in the work of Ginzburg, Kapranov and Vasserot, \cite{GKV}. Quantum toroidal algebras have received a lot of attention in the past -- see \cite{Hernandez} for a review -- and are presently the subject of a renewed interest due to their relevance for integrable systems -- see e.g. \cite{feigin2011,feigin2012,feigin2015} -- and for 5 dimensional supersymmetric Yang-Mills theory and related AGT like correspondence -- see \cite{Morozov}. From a more mathematical perspective, it is well known -- see \cite{Varagnolo} -- that they are the Frobenius-Schur duals of Cherednik's Doubly Affine Hecke Algebras (DAHA) -- see \cite{Cherednik, Macdonald} for classic references on the latter. 

The purpose of the present work is to reconsider quantum toroidal algebras as topological Hopf algebras. On the one hand, this is only natural since the existence of an algebraic comultiplication for quantum toroidal algebras is still essentially open to this date -- although see \cite{Guay} for recent results on algebraic comultiplications for affine Yangians that may suggest the existence of similar results for quantum toroidal algebras -- and only a topological coalgebra structure is provided by the so-called Drinfel'd current coproduct. On the other hand, the existence of a braid group action by bicontinuous algebra automorphisms, generalizing those in \cite{DingKhoroshkin}, provides us with a topological version of the Lusztig symmetries that prove pivotal in both Damiani's and Beck's proofs of Drinfel'd's current presentation. We may therefore expect, in that context, the existence of an alternative presentation for quantum toroidal algebras, in terms of double current generators. In the same spirit as Drinfel'd's current presentation, such a presentation could be regarded as defining the \emph{double quantum affinization} $\uqdd(\g)$ of $\g$ and (a subalgebra $\qdaff'(\mathfrak a_1)$ of) $\uqdd(\g)$ should be isomorphic to (the completion of) $\uqd(\gh)$ -- see section \ref{sec:dqaff}. We therefore expect a diagram of the form
\be\begin{CD}
\g @>\text{\footnotesize Classical Affinization}>> \gh\\
@V\text{\footnotesize Double Quantum Affinization}VV @VV\text{\footnotesize Quantum Affinization}V\\
\qdaff'(\g)  && \qaff(\gh)\\
@V\text{\footnotesize Completion}VV @VV\text{\footnotesize Completion}V\\
\widehat{\qdaff'(\g)} @>\sim>\text{\footnotesize Affine Damiani-Beck isom.}> \widehat{\qaff(\gh)}
\end{CD}\nn
\ee
In the present paper we prove such results in the particular case where $\mathfrak g$ is of type $\mathfrak a_1$. It is fairly natural to conjecture that similar results hold for higher rank root systems, thus yielding
\begin{conj}
Every simple Lie algebra $\g$ admits a (unique up to isomorphisms) double quantum affinization $\uqdd(\g)$.
\end{conj}
\noi and\begin{conj}
Every untwisted affine Ka\v c-Moody Lie algebra $\dot\g$ admits a (unique up to isomorphisms) double quantum affinization $\uqdd(\dot\g)$.
\end{conj}
\noi Note that the latter would naturally provide a definition for the so far elusive triply affine quantum algebras. The latter are believed to play an important role in mathematical physics, as the conformal block side of an AGT type correspondence with 6-dimensional super Yang-Mills theories, \cite{Morozov}.

In any case, $\uqdd(\mathfrak a_1)$ -- and presumably other double quantum affinizations if any -- admits a triangular decomposition $(\uqdd^-(\mathfrak a_1),\uqdd^0(\mathfrak a_1), \uqdd^+(\mathfrak a_1))$. The latter naturally leads to an alternative notion of weight and highest weight modules that we shall refer to as $t$-weight and highest $t$-weight modules. Natural analogues of the finite dimensional modules over quantum affine algebras also appear, that we refer to as \emph{quasi-finite} modules -- see section \ref{sec:dqaff} for definitions. We actually expect that it will be possible to classify simple quasi-finite modules over $\qdaff(\mathfrak a_1)$, by essentially classifying those simple $\qdaff^0(\mathfrak a_1)$-modules that appear as their highest $t$-weight spaces -- see section \ref{sec:dqaff} for the corresponding discussion. This is the subject of ongoing work. 

Quite remarkably, there exists an algebra homomorphism  $f:\E{q^{-4}}{q^2}{q^2} \to \qdaff^{0^+}(\mathfrak a_1)$, where $\qdaff^{0^+}(\mathfrak a_1)$ is a closed subalgebra of $\uqdd^0(\mathfrak a_1)$ and, for every $q_1, q_2, q_3$ such that $q_1q_2q_3=1$, $\E{q_1}{q_2}{q_3}$ is the corresponding elliptic Hall algebra  -- see section \ref{sec:dqaff}. The latter was first defined by Miki in \cite{Miki} as a $(q,\gamma)$-analogue of the $W_{1+\infty}$ algebra. It reappeared later in \cite{feigin2011}, as the quantum continuous  $\mathfrak{gl}_\infty$ algebra. Schiffmann then identified it with the Hall algebra of the category of coherent sheaves on some elliptic curve whose Weil numbers are related to $q_1, q_2, q_3$, \cite{Schiffmann}. More recently, it also appeared in \cite{feigin2012} and in subsequent works by Feigin et al. as the quantum toroidal algebra associated with $\mathfrak{gl}_1$. As we shall see, it appears natural to make the following
\begin{conj}
\label{conj:Hall}
$\uqdd^{0^+}(\mathfrak a_1)$ is isomorphic to the completion of $\E{q^{-4}}{q^2}{q^2}$.
\end{conj}
\noi If it held true, the above conjecture would have many interesting implications. On one hand, in view of Schiffmann's results, it seems reasonable to expect that the double quantum affinization $\qdaff(\mathfrak a_1)$ admits a $K$-theoretic realization, in the spirit of Nakajima's quiver varieties realization of quantum affine algebras \cite{NakajimaQaff}, wherein the generators outside of the elliptic Hall algebras would be realized as correspondences. At the level of representation theory on the other hand, conjecture \ref{conj:Hall} would imply that the classification of the simple $\qdaff^0(\mathfrak a_1)$-modules that appear as highest $t$-weight spaces of simple quasifinite $\qdaff(\mathfrak a_1)$-modules would almost entirely reduce to a classification of the corresponding subclass of simple modules over the elliptic Hall algebra. Again, we leave these questions for future work.

The paper is organized as follows. In section \ref{sec:qtoralg}, we briefly review some well known facts about quantum toroidal algebras, including their definition and natural gradings. We endow them with a topology and construct the corresponding completion. On the latter, we construct a set of automorphisms, including affinized versions of Lusztig's symmetry. Analogues of these for simply laced untwisted affine $\dot{\mathfrak{a}}_{n\geq 2}$-types appeared in the work of Ding and Khoroshkin \cite{DingKhoroshkin}. The $\dot{\mathfrak a}_1$ version we give here plays a crucial role in section \ref{sec:DamianiBeck} where we prove the main result of this paper. In section \ref{sec:dqaff}, we define the double quantum affinization of type $\mathfrak{a}_1$, $\qdaff(\mathfrak{a}_1)$. We prove that there exists an algebra homomorphism from the elliptic Hall algebra $\E{q_1}{q_2}{q_3}$ to its subalgebra $\qdaff^0(\mathfrak{a}_1)$. We also ellaborate on the consequences at the level of representation theory and introduce the notions of (highest) $t$-weights and of quasi-finiteness. Finally, in section \ref{sec:DamianiBeck}, we prove the affinized version of the Damiani-Beck isomorphism. The appendix contains a short review of formal distributions as relevant to the present work. This is already covered in the literature -- see e.g. \cite{Kac} --, however, since our conventions slightly differ from the standard ones, we included it for the sake of clarity.

\subsection*{Notations and conventions}
We let $\N = \{0, 1, \dots\}$ be the set of natural integers including $0$. We denote by $\N^\times$ the set $\N-\{0\}$. For every $m\leq n \in\N$, we denote by $\range{m}{n} = \{m, m+1, \dots, n\}$. We also let $\rran{n}=\range{1}{n}$ for every $n\in\N$. For every finite subset $\Sigma\subset \N$ with $\card \Sigma=N$, any $n\leq N$ and any  $m_1,\dots, m_n\in\N$ such that $m_1+\dots +m_n=N$, we let $\Prt{(m_1,\dots, m_n)}{\Sigma}$ denote the set of ordered $(m_1,\dots, m_n)$ set $n$-partitions, i.e. any $\bA=(A^{(1)}, \dots ,A^{(n)}) \in \Prt{(m_1,\dots, m_n)}{\Sigma}$ is such that
\begin{enumerate}
\item[i.] for every $p\in\rran{n}$, $\card A^{(p)} = m_p$;
\item[ii.] for every $p\in\rran{n}$, $A^{(p)} = \{A^{(p)}_1, \dots, A^{(p)}_{m_p}\}\subset \Sigma$, with $A^{(p)}_1< \dots <A^{(p)}_{m_p}$;
\item[iii.] $A^{(1)} \sqcup \dots \sqcup A^{(n)} = \Sigma$.
\end{enumerate}
We let $\sign:\Z\to\{-1, 0, 1\}$ be defined by setting, for any $n\in\Z$,
$$\sign(n) = \begin{cases}
-1 & \mbox{if $n<0$;}\\
0 & \mbox{if $n=0$;}\\
1 & \mbox{if $n>0$.}
\end{cases}$$

We assume throughout that $\K$ is a field of characteristic $0$ and we let $\F:=\K(q)$ denote the field of rational functions over $\K$ in the formal variable $q$. As usual, we let $\K^\times=\K-\{0\}$ and $\F^\times=\F-\{0\}$. We make $\F$ a topological field by endowing it with the discrete topology.

For every $m, n \in \N$, we define the following elements of $\F$
\be [n]_q := \frac{q^n-q^{-n}}{q-q^{-1}}\,, \qquad [n]_q^! := \begin{cases} [n]_q [n-1]_q \cdots [1]_q & \mbox{if $n\in\N^\times$;}\\
1&\mbox{if $n=0$;}
\end{cases}  \qquad {n \choose m}_q := \frac{[n]_q^!}{[m]_q^! [n-m]_q^!}\, .\ee

We shall let
$${}_a\left[A, B \right ]_b = a AB - b BA\,,$$
for any symbols $a$, $b$, $A$ and $B$ provided the r.h.s of the above equations makes sense. At some point we may need the following obvious identities
\be\label{eq:identity1} [[A,B]_a, C]_b = [[A,C]_b, B]_a + [A,[B,C]]_{ab}\,, \ee
\be\label{eq:identity2} [{}_a[A,B], C]_b = {}_a[[A,C]_b, B] + {}_a[A,[B,C]]_{b} \,.\ee
We refer to the Appendix for conventions and more details on formal distributions. 

The Dynkin diagrams and correponding Cartan matrices in type $\mathfrak a_1$ and $\dot{\mathfrak a}_1$ are reminded in the following table.
\begin{center}
\begin{tabular}{|c|c|c|c|}
\hline
Type & Dynkin diagram & Simple roots & Cartan matrix\\
\hline
		$\mathfrak a_1$&
		\btz
		\tikzstyle{vert}=[circle,thick,fill=black]
		\draw (0,.5) node {$1$};		
		\draw (0,0) node[vert]{};
		\etz
		& $\Phi=\{\alpha_1\}$ & $(2)$\\
		&&&\\
		\hline
		$\dot{\mathfrak a}_1$&
		\btz
\tikzset{doublearrow/.style={draw, color=black, draw=black, double distance=2pt, ->}}
\tikzset{doubleline/.style={draw, color=black, draw=black, double distance=2pt}}
\tikzstyle{vert}=[circle,thick,fill=black]
\draw[doubleline] (0:2) node[vert] {}-- ++(0:1.1) node[vert] {}; 
\draw[doublearrow] (0:2.9)-- ++(180:.77); 
\draw[doublearrow] (0:2.3)-- ++(0:.65); 
\draw (2,.5) node {$0$};
\draw (3.1,.5) node {$1$};
\etz
		& $\dot\Phi=\{\alpha_0, \alpha_1\}$ & $\begin{pmatrix}
		2&-2\\-2&2
		\end{pmatrix}$\\
		&&&\\
		\hline
		\end{tabular}
\end{center}

\section{The quantum toroidal algebra of type $\mathfrak{a}_1$ and its completion}
\label{sec:qtoralg}
\subsection{Definition}
\label{sec:defnqtor}
Let $\dot{I}=\{0,1\}$ be the above labeling of the nodes of the Dynkin diagram of type $\dot{\mathfrak a}_1$ and let $\dot\Phi = \left\{\alpha_0, \alpha_1\right \}$ be a choice of simple roots for the corresponding root system. We denote by $(c_{ij})_{i,j=0,1}$ the entries of the associated Cartan matrix. Let $\dot Q^\pm = \Z^\pm \alpha_0 \oplus \Z^\pm  \alpha_1$ and let $\dot Q = \Z\alpha_0 \oplus \Z\alpha_1$ be the type $\dot{\mathfrak a}_1$ root lattice.
\begin{defn}
\label{def:defqaffdota1}
The \emph{quantum toroidal algebra} $\uqslthh$ is the associative $\F$-algebra generated by the generators 
$$\left \{D, D^{-1}, C^{1/2}, C^{-1/2}, k_{i, n}^+, k_{i, -n}^-,  x_{i,m}^+, x_{i,m}^- : i \in \dot{I}, m \in Z, n \in \N\right\}$$ 
subject to the following relations
\be\label{eq:ccentral} \mbox{$C^{\pm1/2}$ is central} \qquad C^{\pm 1/2} C^{\mp 1/2} = 1 \qquad D^{\pm 1} D^{\mp 1} =1\ee
\be D\kk\pm i(z)D^{-1} = \kk\pm i(zq^{-1}) \qquad D\x\pm i(z)D^{-1} = \x\pm i(zq^{-1}) \ee
\be \res_{z_1,z_2} \frac{1}{z_1z_2}\kk\pm i(z_1)\kk\mp i(z_2) = 1\ee
\be \kk \pm i(z_1) \kk \pm j(z_2) =  \kk \pm j(z_2) \kk \pm i(z_1)\ee
\be \kk - i(z_1) \kk + j(z_2) = G^-_{ij}(C^{-1}z_1/z_2) G^+_{ij}(Cz_1/z_2)  \kk + j(z_2)  \kk - i(z_1) 
\label{eq:kpkm}\ee
\be\label{eq:kpxpm} G_{ij}^\mp(C^{\mp 1/2} z_2/z_1)  \kk +i(z_1) \x \pm j(z_2) = \x \pm j(z_2)  \kk +i(z_1)\ee
\be\label{eq:kmxpm} \kk -i(z_1) \x \pm j(z_2) = G_{ij}^\mp(C^{\mp 1/2} z_1/z_2) \x \pm j(z_2)  \kk -i(z_1)\ee
\be\label{eq:xpmxpm} (z_1-q^{\pm c_{ij}} z_2) \x \pm i(z_1) \x \pm j(z_2) = (z_1q^{\pm c_{ij}}-z_2) \x \pm j(z_2) \x \pm i(z_1)\ee
\be\label{eq:relx+x-} [\x + i(z_1), \x - j(z_2)] = \frac{\delta_{ij}}{q-q^{-1}} \left [ \delta \left ( \frac{z_1}{Cz_2} \right ) \kk + i(z_1C^{-1/2}) - \delta \left( \frac{z_1C}{z_2}\right ) \kk - i(z_2C^{-1/2}) \right ]  \ee
\be\label{eq:qaffserre} \sum_{\sigma \in S_{1-c_{ij}}} \sum_{k=0}^{1-c_{ij}} (-1)^k {{1-c_{ij}}\choose{k}}_q \x \pm i(z_{\sigma(1)})  \cdots \x \pm i(z_{\sigma(k)}) \x \pm j(z)  \x \pm i(z_{\sigma(k+1)}) \cdots \x\pm i(z_{\sigma(1-c_{ij})}) =0  \ee
where, for every $i \in \dot I$, we define the following $\uqslthh$-valued formal distributions
\be \x \pm i(z) := \sum_{m \in \Z} x^\pm_{i, m} z^{-m} \in \uqslthh[[z, z^{-1}]]\,;\ee
\be \kk \pm i(z) := \sum_{n \in \N} k_{i, \pm n}^\pm z^{\mp n} \in \uqslthh[[z^{\mp 1}]]\, ,\ee
for every $i,j\in\dot I$, we define the following $\F$-valued formal power series
\be G_{ij}^\pm(z):= q^{\pm c_{ij}} + (q-q^{-1}) [\pm c_{ij} ]_q \sum_{m \in \N^\times} q^{\pm m c_{ij}} z^m \in \F[[z]] \ee
and
\be \delta(z) := \sum_{m \in \Z} z^{m} \in \F[[z, z^{-1}]]\ee
is an $\F$-valued formal distribution.
\end{defn}
Note that $G_{ij}^\pm(z)$ is invertible in $\F[[z]]$ with inverse $G^\mp_{ij}(z)$, \ie
\be G_{ij}^\pm(z) G_{ij}^{\mp}(z) = 1\, ,\ee
and that it can be viewed as the power series expansion of a rational function of $(z_1, z_2) \in \C^2$ as $|z_2| \gg |z_1|$, which we shall denote as follows
\be G_{ij}^\pm(z_1/z_2) = \left ( \frac{z_1q^{\mp c_{ij}} - z_2}{z_1-q^{\mp c_{ij}} z_2} \right )_{|z_2| \gg |z_1|} \, .\ee
Observe furthermore that we have the following useful identity in $\F[[z, z^{-1}]]$
\be\label{eq:G+G-} 
\frac{G_{ij}^\pm(z)- G_{ij}^\mp(z^{-1}) }{q-q^{-1}} = [\pm c_{ij}]_q \delta \left( z q^{\pm c_{ij}} \right )\, . \ee
\begin{rem}
\label{rem:Gij}
In type ${\mathfrak a}_1$, $\dot{I}=\{0,1\}$, $c_{ij} = 4\delta_{ij} -2$ and we have an additional identity, namely $G_{10}^\pm(z) = G_{11}^\mp(z)$. We refer to section \ref{sec:structfunc} of the Appendix for more identities involving the formal power series $G_{ij}^\pm(z)$.
\end{rem}
\noi $\qaff(\dot{\mathfrak a}_1)$ is obviously a $\Z$-graded algebra, \ie we have
\be \uqslthh = \bigoplus_{n \in \Z} \uqslthh_n \, , \qquad \mbox{where for all $n \in \Z$} \qquad \uqslthh_n := \{x \in \uqslthh: DxD^{-1}=q^n x \}\, .\label{Zgrad}\ee
It was proven in \cite{Hernandez05} to admit a triangular decomposition $(\uqmslthh, \uqzeroslthh, \uqpslthh)$, where $\uqpmslthh$ and $\uqzeroslthh$ are the subalgebras of $\qaff(\dot{\mathfrak a}_1)$ respectively generated by $\left\{x_{i, m}^\pm : i \in \dot I, m \in \Z\right \}$ and 
$$\left\{C, C^{-1}, D, D^{-1}, k_{i, m}^+, k_{i, m}^-: i \in \dot I, m \in \Z\right\}\,.$$
Observe that $\uqpmslthh$ admits a natural gradation over $\dot Q^\pm$ that we shall denote by
\be \uqpmslthh = \bigoplus_{\alpha \in \dot Q^\pm} \uqpmslthh_\alpha\, .\ee
Of course $\qaff(\dot{\mathfrak a}_1)$ is graded over the root lattice $\dot Q$. We finally remark that the two Dynkin diagram subalgebras $\qaff(\mathfrak a_1)^{(0)}$ and $\qaff(\mathfrak a_1)^{(1)}$ of $\qaff(\dot{\mathfrak a}_1)$ generated by 
$$\left \{D, D^{-1}, C^{1/2}, C^{-1/2}, k_{i, n}^+, k_{i, -n}^-,  x_{i,m}^+, x_{i,m}^- :  m \in Z, n \in \N\right\}\,,$$ 
with $i=0$ and $i=1$ respectively, are both isomorphic to $\qaff(\mathfrak{a}_1)$, thus yielding two injective algebra homomorphisms $\iota^{(i)}: \qaff(\mathfrak a_1) \hookrightarrow \qaff(\dot{\mathfrak a}_1)$.

\subsection{Automorphisms of $\uqslthh$}
\begin{prop}
\begin{enumerate}
\item[i.] For every Dynkin diagram automorphism $\pi: \dot{I} \isom \dot{I}$, there exists a unique $\F$-algebra automorphism $T_\pi \in \Aut \uqslthh$ such that
\be T_\pi(\x \pm i(z)) = \x{\pm}{\pi(i)}(z)\, ,  \qquad T_\pi(\kk \pm i(z)) = \kk{\pm}{\pi(i)}(z)\, ,  \qquad T_\pi(C) = C \, ,  \qquad T_\pi(D) = D \, . \ee
\item[ii.] For every $i \in \dot I$, there exists a unique $\F$-algebra automorphism $T_{\omega^\vee_i}\in \Aut \uqslthh$ such that
\be T_{\omega^\vee_i}(\x \pm j(z)) = z^{\pm \delta_{ij}}\x \pm j(z) \qquad T_{\omega^\vee_i}(\kk \pm j(z)) = C^{\mp \delta_{ij}}\kk \pm j(z) \qquad T_{\omega^\vee_i}(C) = C  \qquad T_{\omega^\vee_i}(D) = D\ee
\item[iii.] There exists a unique involutive $\F$-algebra anti-homomorphism $\eta \in \Aut \uqslthh$ such that
\be \eta (\x \pm i(z)) = \x \pm i(1/z)\qquad \eta(\kk \pm i(z)) = \kk \mp i(1/z) \qquad \eta(C) = C \qquad  \eta(D) = D 
\ee
\item[iv.] There exists a unique involutive $\K$-algebra anti-homomorphism $\varphi$ such that
\be \varphi (\x \pm i(z)) = \x \mp i(1/z)\qquad \varphi(\kk \pm i(z)) = \kk \mp i(1/z) \qquad \varphi(C) = C^{-1} \qquad  \varphi(D) = D^{-1}  \qquad \varphi(q) = q^{-1} \ee
\end{enumerate}
\end{prop}
\begin{rem}
In the present case, the Dynkin diagram being that of type $\dot{\mathfrak{a}}_1$, $\dot I=\{0,1\}$ and the only nontrivial diagram automorphism is defined by setting $\pi(0)=1$ and $\pi(1)=0$.
\end{rem}
\begin{rem}
Note that $\varphi$ restricts as a non-trivial automorphism of the field $\F$ and that, as such, it yields \eg
\be \varphi(G_{ij}^\pm(z)) =G_{ij}^\mp(z) \, .\ee
\end{rem}

\subsection{The completions $\uqsltc$ and $\uqslthh^{\widehat\otimes m\geq 2}$}
\label{sec:topology}
Let, for every $n\in\N$,
$$\Omega_{n}:= \bigoplus_{\substack{r\geq n\\s \geq n}} \uqslthh \cdot \uqslthh_{-r} \cdot \uqslthh\cdot \uqslthh_{s} \cdot \uqslthh  \,.$$
\begin{prop}
\label{prop:Omegan}
The following hold true:
\begin{enumerate}
\item [i.] For every $n\in\N$, $\Omega_{n}$ is a two-sided ideal of $\qaff(\dot{\mathfrak a}_1)$;
\item[ii.] For every $n\in\N$, $\Omega_{n}\supseteq \Omega_{n+1}$;
\item[iii.] $\Omega_{0} := \bigcup_{n\in\N}\Omega_{n} = \qaff(\dot{\mathfrak a}_1)$;
\item [iv.] $\bigcap_{n\in\N} \Omega_{n} =\{0\}$;
\item[v.] For every $m,n\in\N$, $\Omega_{m} + \Omega_{n} \subseteq \Omega_{\min(m,n)}$;
\item[vi.] For every $m,n\in\N$, $\Omega_{m} \cdot \Omega_{n} \subseteq \Omega_{\max(m,n)}$.
\end{enumerate}
\end{prop}
\begin{proof}
Points \emph{i.} and \emph{ii.} are obvious. As sets, it is clear that $\Omega_0 \subseteq \uqslthh$. Now, $1\in\uqslthh_{0}$ and for every $x\in\uqslthh$, we can write $x=1\cdot x\cdot 1$ thus proving that $x\in\Omega_0$. Point \emph{iii.} follows. Point \emph{v.} is an easy consequence of point \emph{ii.}. Point \emph{vi.} is obvious given \emph{i.}. So let us finally prove point \emph{iv.}. In order to do so, it suffices to prove that for every $x\in\qaff(\dot{\mathfrak a}_1)-\{0\}$, there exists a largest integer $\nu_x\in\N$ such that $x\in\Omega_{\nu_x}$; for then indeed $x\notin\Omega_{\nu_x+1}$, whereas obviously $0\in\Omega_n$, for every $n\in\N$. Relations ((\ref{eq:kpkm}) -- (\ref{eq:relx+x-})) respectively imply that, for every $i,j \in\dot I$, every $m\in\N$ and every $n\in\N^\times$,
\be k^+_{i,m} k^-_{j,-n} = k^-_{j,-n}  k^+_{i,m} - (q^{c_{ij}} - q^{-c_{ij}}) (C-C^{-1}) \sum_{p=1}^{\min(m,n)} \frac{q^{-pc_{ij}} C^{p}- q^{pc_{ij}} C^{-p}}{q^{-c_{ij}} C - q^{c_{ij}} C^{-1}} k^-_{j,-n+p}  k^+_{i,m-p}\,, \nn
\ee
\be k^+_{i,m} x^\pm_{j,-n} = q^{\pm c_{ij}} x^\pm_{j,-n} k^+_{i,m} +(q^{\pm c_{ij}} - q^{\mp c_{ij}}) \sum_{p=0}^m C^{\mp p/2} q^{\pm p c_{ij}} x^\pm_{j,-n+p} k^+_{i,m-p}\,, \nn\ee
\be x^\pm_{i,m}k^-_{j,-n} = q^{\pm c_{ij}} k^-_{j,-n} x^\pm_{i,m} +(q^{\pm c_{ij}} - q^{\mp c_{ij}}) \sum_{p=0}^n C^{\mp p/2} q^{\pm p c_{ij}} k^-_{j,-n+p} x^\pm_{i,m-p}\,, \nn\ee

\bea x^\pm_{i,m} x^\pm_{j,-n} &=& q^{\pm c_{ij}} x^\pm_{j,-n} x^\pm_{i,m} +(q^{\pm c_{ij}} - q^{\mp c_{ij}} ) \sum_{p=0}^{\min(m,n)-1} q^{\pm p c_{ij}} x^\pm_{j,-n+p} x^\pm_{i,m-p}\nn\\
&& - q^{\pm (\min(m,n) -1) c_{ij}} x^\pm_{j,\min(m,n) - n} x^\pm_{i,m-\min(m,n)} +q^{\pm \min(m,n) c_{ij}} x^\pm_{i,m-\min(m,n)} x^\pm_{j,\min(m,n)-n}\,,\nn\eea
\be x^\pm_{i,m} x^\mp_{j,-n}  = x^\mp_{j,-n}  x^\pm_{i,m} \pm \frac{\delta_{ij}}{q-q^{-1}}  
\begin{cases}
C^{\pm\frac{m+n}{2}} k^{+}_{i,m-n} & \mbox{if $m>n$;}\\
-C^{\mp \frac{m+n}{2}} k^-_{i,n-m} & \mbox{if $m<n$;}\\
\left [C^{\pm m} k^+_{i,0} - C^{\mp m} k^-_{i,0}\right ]& \mbox{if $m=n$.}
\end{cases}
\nn\ee
Now let
$$B = \left \{b_{\textbf{a}, \textbf{m}} = \overrightarrow{\prod_{p\in\rran n}} \xi_{a_p, m_p}: n\in\N, \quad \textbf{a} =(a_1,\dots,a_n) \in(\dot \Phi\sqcup -\dot\Phi \sqcup \dot I)^n, \quad \textbf{m} =(m_1,\dots, m_n) \in \Z^n\right\}\,,$$
where, for every $(a,m)\in(\dot\Phi\sqcup -\dot\Phi \sqcup\dot I)\times \Z$,
\be \xi_{a, m} = \begin{cases}
x_{i,m}^\pm &\mbox{if $a =\pm \alpha_i \in\pm\dot\Phi$, $i\in\dot I$;}\\
k_{i,m}^\pm &\mbox{if $a=i\in\dot I$ and $m\in \Z^\pm$.}
\end{cases}\nn\ee
If we omit $C^{\pm 1/2}$ and $D^{\pm1}$ which are clearly irrelevant for the present discussion, $B$ is obviously a spanning set for $\qaff(\dot{\mathfrak a}_1)$. Making repeated use of the above relations, one then easily shows that, for every $n\in\N$, every $\textbf{a} \in(\dot \Phi\sqcup-\dot\Phi\sqcup \dot I)^n$ and every $\textbf{m} \in \Z^n$,
$$b_{\textbf{a}, \textbf{m}} - c_{\textbf{a}, \textbf{m}} \overrightarrow{\prod_{\substack{p\in\rran n\\ m_p< 0}}} \xi_{a_p, m_p} \overrightarrow{\prod_{\substack{p\in\rran n\\m_p\geq 0}}} \xi_{a_p, m_p} \in \Omega_{N(\textbf{m})-1}- \Omega_{N(\textbf{m})}\,,$$
where $c_{\textbf{a},\textbf{m}} \in \F^\times$ and
$$N(\textbf{m}) = \min\left (-\sum_{\substack{p\in\rran n\\ m_p< 0}} m_p, \sum_{\substack{p\in\rran n\\ m_p\geq 0}} m_p\right ) \,.$$
As a consequence, $\nu_{b_{\textbf{a},\textbf{m}}} \leq N(\textbf{m})$,
which concludes the proof.
\end{proof}

Similarly, making use of the natural $\Z$-grading of the tensor algebras $\uqslthh^{\otimes m}$, $m\in\N^\times$, we let, for every $n\in\N$, 
$$\Omega_{n}^{(m)} := \bigoplus_{\substack{r\geq n\\s \geq n}} \uqslthh^{\otimes m} \cdot \left (\uqslthh^{\otimes m}\right )_{-r} \cdot \uqslthh^{\otimes m}\cdot \left (\uqslthh^{\otimes m}\right )_{s} \cdot \uqslthh^{\otimes m}  \,.$$
One easily checks that for every $m\in\N^\times$, $\{\Omega_n^{(m)}:n\in\N\}$ has the same properties as the ones established in proposition \ref{prop:Omegan} for $\{\Omega_n = \Omega_n^{(1)}:n\in\N\}$. 
\begin{defprop}
\label{defprop:topol}
We endow $\qaff(\dot{\mathfrak a}_1)$ with the topology $\tau$ whose open sets are either $\emptyset$ or nonempty subsets $\mathcal O\subseteq \qaff(\dot{\mathfrak a}_1)$ such that for every $x\in \mathcal O$, $x+\Omega_{n}\subseteq \mathcal O$ for some $n\in\N$. Similarly, we endow each tensor power $\qaff(\dot{\mathfrak a}_1)^{\otimes m\geq 2}$ with the topology induced by $\{\Omega_n^{(m)}:n\in\N\}$. These turn $\qaff(\dot{\mathfrak a}_1)$ into a (separated) topological algebra. We then let $\widehat{\qaff(\dot{\mathfrak a}_1)}$ denote its completion and we extend by continuity to $\widehat{\qaff(\dot{\mathfrak a}_1)}$ all the (anti)-automorphisms defined over $\qaff(\dot{\mathfrak a}_1)$ in the previous section. We eventually denote by $\qaff(\dot{\mathfrak a}_1)^{\widehat\otimes m\geq 2}$ the corresponding completions of $\qaff(\dot{\mathfrak a}_1)^{\otimes m\geq 2}$.
\end{defprop}
\begin{proof}
The addition is automatically continuous in the above defined topology of $\qaff(\dot{\mathfrak a}_1)$. The continuity of the multiplication follows from point \emph{vi.} of proposition \ref{prop:Omegan}. Point \emph{iv.}, in turn, implies that $\qaff(\dot{\mathfrak a}_1)$, as a topological space, is Hausdorff. The continuity of the unit map $\eta : \F\to\qaff(\dot{\mathfrak a}_1)$ is easily checked -- remember that $\F$ is given the discrete topology.
\end{proof}
\begin{rem}
It is worth noting that the above topology is actually ultrametrizable. In the notations of the previous proof, let indeed, for every $x\in\uqslthh$,
$$\|x\| = \begin{cases} \exp\left( - \nu_x\right ) & \mbox{if $x\in \qaff(\dot{\mathfrak a}_1)-\{0\}$;}\\
0 & \mbox{if $x=0$.}
\end{cases}$$
Since obviously $\nu_{x+y} \geq \min(\nu_x,\nu_y)$ for every $x,y\in\qaff(\dot{\mathfrak a}_1)$, the ultrametric inequality for the metric defined by $d(x,y) = \|x-y\|$ follows immediately as a consequence of the inequality $\|x+y\| \leq \max (\|x\|,\|y\|)$.
\end{rem}

\subsection{Continuous Lusztig automorphisms}
Following \cite{Macdonald} we make the following
\begin{defn}
The affine braid group $\dot{\mathfrak B}$ of type $\dot {\mathfrak a}_1$ is generated by $t$ and $y$ subject to the relation $ty^{-1}t =y$.
\end{defn}
The coweight lattice $P^\vee$ of $\dot {\mathfrak a}_1$ is an abelian group whose generators we shall denote as $x_{\lambda}$ for every $\lambda \in P^\vee$. In particular, we shall write
\be x_\lambda x_\mu = x_\mu x_\lambda = x_{\lambda + \mu}\, ,\ee
assuming that $x_0=1$. There exists a unique group homomorphism $\dot{\mathfrak B} \rightarrow \Aut{P^\vee}$ defined by letting
\be t(x_\lambda) =  x_{s_{\alpha_1}(\lambda)} \, , \qquad y(x_\lambda) = x_{\lambda}\, ,\ee
where $s_{\alpha_1}$ denotes the reflection in the simple root $\alpha_1$, \ie $s_{\alpha_1}(\lambda ) = \lambda - (\alpha_1^\vee, \lambda) \alpha_1$. This action allows us to make the following
\begin{defn}
We let $\ddot{\mathfrak B} := \dot{\mathfrak B}  \ltimes P^\vee$, \ie $\ddot{\mathfrak B}$ is isomorphic to the group with generators $t, y$ and $(x_\lambda)_{\lambda \in P^\vee}$ obeying the relations
\be ty^{-1}t = y\, , \qquad tx_\lambda t^{-1} = x_{s_{\alpha_1}(\lambda)} \, , \qquad x_\lambda y = y x_{\lambda}\, ,\ee 
for every $\lambda \in P^\vee$.
\end{defn}

We now define an action of $\br$ on $\uqsltc$ by bicontinuous algebra automorphisms, \ie we construct a group homomorphism $\br \rightarrow \Aut\uqsltc$. In order to do so, we first describe the image of the latter, following \cite{DingKhoroshkin}.
\begin{prop}
There exists a unique bicontinuous algebra automorphism $T \in \Aut\uqsltc$ such that
\be T(C) = C \qquad T(D)=D \qquad T(\kk \pm 0(z)) =  \kk \pm 0(zq^2)\kk \pm 1(z) \kk\pm 1(zq^2) \qquad  T(\kk \pm 1(z)) =  \kk \pm 1(z)^{-1}  \ee
\be T(\xp 0(z)) = \frac{1}{[2]_q} \res_{z_1,z_2} z_1^{-1}z_2^{-1} \left [ \x +1(z_1), \left [ \x +1(z_2),  \x +0(zq^2)\right]_{G_{10}^-(z_2/zq^2)} \right]_{G_{11}^-(z_1/z_2) G_{10}^-(z_1/zq^2)} \ee
\be T(\xm 0(z)) =  \frac{1}{[2]_q} \res_{z_1,z_2} z_1^{-1}z_2^{-1}  \left [  \left [ \x -0(zq^2), \x -1(z_1) \right]_{G_{10}^+(zq^2/z_1)}, \x -1(z_2) \right]_{G_{11}^+(z_1/z_2) G_{10}^+(zq^2/z_2)} \ee
\be T(\xp 1(z)) = -\x - 1(C^{-1} z) \kk + 1(C^{-1/2} z)^{-1}\ee
\be T(\xm 1(z)) = -\kk - 1(C^{-1/2} z)^{-1} \x + 1(C^{-1}z) \ee
\end{prop}
\begin{proof}
It suffices to check all the relations, which is cumbersome but straightforward. The inverse automorphism is given by
\be T^{-1}(C) = C \qquad T^{-1}(D)=D \qquad T^{-1}(\kk \pm 0(z)) =  \kk \pm 0(zq^{-2})\kk \pm 1(z) \kk\pm 1(zq^{-2}) \qquad  T^{-1}(\kk \pm 1(z)) =  \kk \pm 1(z)^{-1}  \ee
\be T^{-1}(\xp 0(z)) = \frac{1}{[2]_q} \res_{z_1,z_2} z_1^{-1}z_2^{-1}  \left [  \left [ \x +0(zq^{-2}), \x +1(z_1) \right]_{G_{10}^-(zq^{-2}/z_1)}, \x +1(z_2) \right]_{G_{11}^-(z_1/z_2) G_{10}^-(zq^{-2}/z_2)} \ee
\be T^{-1}(\xm 0(z)) =  \frac{1}{[2]_q} \res_{z_1,z_2} z_1^{-1}z_2^{-1}  \left [ \x -1(z_1), \left [ \x -1(z_2),  \x -0(zq^{-2})\right]_{G_{10}^+(z_2/zq^{-2})} \right]_{G_{11}^+(z_1/z_2) G_{10}^+(z_1/zq^{-2})}\ee
\be T^{-1}(\xp 1(z)) = -\kk -1(C^{1/2}z)^{-1} \x - 1(C z) \ee
\be T^{-1}(\xm 1(z)) = - \x + 1(Cz) \kk + 1(C^{1/2} z)^{-1}\ee
\end{proof}
\begin{rem}
Making use of the defining relations of $\qaff(\dot{\mathfrak a}_1)$, one easily shows that indeed
\be\label{eq:Tx0+}\left [ \x +1(z_1), \left [ \x +1(z_2),  \x +0(zq^2)\right]_{G_{10}^-(z_2/zq^2)} \right]_{G_{11}^-(z_1/z_2) G_{10}^-(z_1/zq^2)} = [2]_q\, \delta\left (\frac{z_1}{q^2z_2}\right ) \delta \left (\frac{z_2}{z}\right ) T\left( \xp0(z)\right ) \,,\ee
\be\label{eq:Tx0-}  \left [  \left [ \x -0(zq^2), \x -1(z_1) \right]_{G_{10}^+(zq^2/z_1)}, \x -1(z_2) \right]_{G_{11}^+(z_1/z_2) G_{10}^+(zq^2/z_2)} = [2]_q\, \delta\left(\frac{z_1q^2}{z_2}\right )\delta\left(\frac{z_1}{z}\right ) T(\xm 0(z)) \,. \ee
\end{rem}
\noi The following is straightforward but will be useful.
\begin{prop}
We have
\begin{enumerate}
\item[i.] $\varphi \circ T_\pi = T_\pi \circ \varphi$; 
\item[ii.] $\varphi \circ T = T \circ \varphi$; 
\item[iii.] $T^{-1} = \eta \circ T \circ \eta$.
\end{enumerate}
\end{prop}

We have finally,
\begin{thm}
The assignement
\be
 t \mapsto T \qquad  y \mapsto Y:= T_\pi \circ T  \qquad x_{\omega_i^\vee}  \mapsto T_{\omega_i^\vee} \ee
extends to a group homomorphism $\ddot{\mathfrak B} \rightarrow \Aut\uqsltc$.
\end{thm}
\begin{proof}
This is a cumbersome but straightforward exercise that we leave to the reader.
\end{proof}

\begin{rem}
In \cite{Miki99}, Miki constructed an algebraic action by automorphisms of the extended elliptic braid group on $\uqslthh$ which should not be confused with the topological action of $\br$ on $\uqsltc$ provided by the above theorem.
\end{rem}

\subsection{Topological Hopf algebra structure on $\widehat\uqslthh$}
\begin{defn}
We endow the topological $\F$-algebra $\widehat\uqslthh$ with: 
\begin{enumerate}
\item[i.] the comultiplication $\Delta:\widehat\uqslthh\to \uqslthh\widehat{\otimes} \uqslthh$ defined by
\be
\Delta(C^{\pm1/2}) = C^{\pm 1/2} \otimes C^{\pm 1/2}\,,\qquad \Delta(D^{\pm 1}) = D^{\pm 1}\otimes D^{\pm 1}\,,\ee
\be\Delta(\kk\pm i(z)) =\kk\pm i(z C^{\mp 1/2}_{(2)})\otimes \kk\pm i(z C^{\pm 1/2}_{(1)})\,,\ee
\be \Delta(\x+i(z)) =\x+i(z)\otimes 1 + \kk-i(z C^{-1/2}_{(1)})\widehat{\otimes} \x+i(z C^{-1}_{(1)})\,, \ee
\be \Delta(\x-i(z)) =\x-i(z C^{-1}_{(2)})\widehat{\otimes} \kk+i(z C^{-1/2}_{(2)})+ 1  \otimes \x-i(z)\,, \ee
where $C^{\pm 1/2}_{(1)} = C^{\pm 1/2} \otimes 1$ and $C^{\pm 1/2}_{(2)} = 1\otimes C^{\pm 1/2}$;
\item[ii.] the counit $\varepsilon : \widehat\uqslthh\to \F$, defined by $\varepsilon(D^{\pm1}) = \varepsilon(C^{\pm 1/2}) = \varepsilon (\kk\pm i(z))=1$, $\varepsilon(\x\pm i(z))=0$ and; 
\item[iii.] the antipode $S:\widehat\uqslthh\to \widehat\uqslthh$, defined by $S(D^{\pm 1}) = D^{\mp1}$, $S(C^{\pm 1/2}) = C^{\mp 1/2}$ and
$$S(\kk\pm i(z)) = \kk\pm i(z)^{-1}\,, \qquad S(\x+i(z)) = - \kk-i(z C^{1/2})^{-1} \x+i(z C)\,, \qquad S(\x-i(z))=- \x-i(z C) \kk+i(zC^{1/2})^{-1} \,.$$
\end{enumerate}
With these operations so defined and the topologies defined in section \ref{sec:topology}, $\widehat\uqslthh$ is a topological Hopf algebra.
\end{defn}

\subsection{Non-degenerate Hopf algebra pairing}
Define $\uqgslthh$ (resp. $\uqlslthh$) as the subalgebra of $\uqslthh$ generated by $\left\{k^-_{i,-m}, x^+_{i,n} : i\in I, m\in\N, n\in\Z\right \}$ (resp. $\left\{ k^+_{i,m}, x^-_{i,n}: i\in I, m\in\N, n\in\Z\right \}$). In view of the triangular decompositon of $\qaff(\dot{\mathfrak a}_1)$ -- see \cite{Hernandez05} -- and of its defining relations, it is clear that $\uqgslthh$ (resp. $\uqlslthh$), as an $\F$-vector space, is spanned by
\bea \left\{x_{i_1, r_1}^+\cdots x_{i_m, r_m}^+ k_{j_1, -s_1}^- \cdots k_{j_n, -s_n}^- \right .&:& m,n\in\N, \left ((i_1, r_1) ,\dots, (i_m, r_m) \right)\in (\dot I \times \Z )^m
\nn\\&& 
\left .
\left ((j_1, s_1), \dots , (j_n, s_n)\right )\in (\dot I \times \N )^n 
\right\}\label{eq:basisu>}\eea
\bea \left (\mbox{resp.} \qquad \left\{x_{i_1, r_1}^-\cdots x_{i_m, r_m}^- k_{j_1, s_1}^+ \cdots k_{j_n, s_n}^+ \right .\right. &:& m,n\in\N, \left((i_1,r_1),\dots, (i_m, r_m)\right )\in (\dot I \times \Z)^m 
, \nn\\
&&\left . \left .\left((j_1, s_1), \dots , (j_n,s_n)\right )\in (\dot I \times \N )^n
\right\}\right ).\label{eq:basisu<}\eea
\begin{prop}
There exists a unique non-degenerate Hopf algebra pairing $\pair{}{}: \uqgslthh \times \uqlslthh \to \F$, defined by setting
\be \pair{\x+i(z)}{\x-j(v)} = \frac{\delta_{ij}}{q-q^{-1}}\delta\left (\frac{z}{v}\right )\,,\ee
\be \pair{\kk-i(z)}{\kk+j(v)}  = G^-_{ij}\left (\frac{z}{v}\right )\,,\ee
\be \pair{\kk-i(z)}{\x-j(v)} = \pair{\x+i(z)}{\kk+j(v)} =0\,.\ee
By definition, it is such that, for every $a,b\in \uqgslthh$ and every $x,y\in\uqlslthh$,
$$\left\langle a, xy\right \rangle = \sum \left\langle a_{(1)}, x\right \rangle\left\langle a_{(2)}, y\right \rangle\,, $$
$$\left\langle ab, x\right \rangle = \sum \left\langle a, x_{(2)}\right \rangle\left\langle b, x_{(1)}\right \rangle\,,$$
$$\left\langle a, 1\right \rangle = \varepsilon_{\geq}(a)\qquad  \left\langle 1, x\right \rangle=\varepsilon_{\leq}(x)\,,$$
where we have set $\varepsilon_\leq = \varepsilon_{|\qaff^\leq(\dot{\mathfrak a}_1)}$, $\varepsilon_\geq = \varepsilon_{|\qaff^\geq(\dot{\mathfrak a}_1)}$ and we have made use of Sweedler's notation for the comultiplication
$$\Delta(x) = \sum x_{(1)} \widehat\otimes x_{(2)}\,.$$
\end{prop}
\begin{proof}
A proof can be found in {\cite{Negut}}.
\end{proof}
\noi Before we can establish the continuity of the above defined pairing, we need the following
\begin{lem}
\label{pairingbasis}
For every $m_+, m_-, n_+, n_-\in\N$, $(i_1^\pm, \dots, i_{m_{\pm}}^\pm)\in\dot I^{m_{\pm}}$ and every $(j_1^\pm, \dots, j_{n_{\pm}}^\pm)\in\dot I^{n_\pm}$, we have
\bea
&&\left \langle \xp{i_1^+}(u_1)\cdots \xp{i_{m_+}^+}(u_{m_+}) \km{j_1^+}(v_1) \cdots \km{j_{n_+}^+}(v_{n_+}), \xm{i_{1}^-}(w_1) \cdots \xm{i_{m_-}^-}(w_{m_-}) \kp{j_1^-}(z_1) \cdots \kp{j_{n_-}^-}(z_{n_-} )\right \rangle\qquad \quad\nn\\
&=&\delta_{m_+,m_-} \left (\prod_{\substack{r\in\rran{n_+}\\s\in\rran{n_-}}} G_{j_{r}^+, j_s^-}^- \left (\frac{v_r}{z_s}\right ) \right ) \sum_{\sigma\in S_{m_+}}  \left (\prod_{\substack{1\leq r<s\leq m_+\\\sigma(r)>\sigma(s)}} G_{i_r^+, i_s^+}^-\left(\frac{u_r}{u_s}\right ) \right ) \prod_{t\in\rran{m_+}} \frac{\delta_{i_t^+, i_{\sigma(t)}^-}}{q-q^{-1}} \,\delta\left(\frac{w_{\sigma(t)}}{u_t}\right )\,.
\eea
\end{lem}
\begin{proof}
One easily proves by recursion the results for $n_+=n_-=0$ and $m_+=m_-=0$, respectively. The general case then follows by a straightforward calculation.
\end{proof}
\noi It follows that -- remember $\F$ is given the discrete topology --
\begin{cor}
The Hopf algebra pairing $\left\langle, \right\rangle$ is (separately) continuous.
\end{cor}
\begin{proof}
It suffices to prove that for every $x\in \uqgslthh$ there exists an $m\in \N$ such that, for every $n\geq m$
$$\left\langle x, \Omega_n \cap \qaff^\leq(\dot{\mathfrak a}_1) \right\rangle = \left\{0\right \}\,.$$
In order to prove the latter, it suffices to prove it over the spanning sets of (\ref{eq:basisu>}) and (\ref{eq:basisu<}). Now this easily follows by inspection, making use of lemma \ref{pairingbasis} and of the fact that, for any $y\in \qaff(\dot{\mathfrak a}_1)-\{0\}$, there exists $\nu_y \in\N$ such that $y\notin \Omega_{\nu_y+1}$ -- see proof of proposition \ref{prop:Omegan}.
\end{proof}
\noi We can now extend $\left\langle , \right\rangle$ from  $\uqgslthh \times \uqlslthh$ to $ \uqgslthh \times \widehat{\uqlslthh}$ by continuity. Importantly, we have
\begin{prop}
The extended pairing $\left \langle, \right\rangle : \uqgslthh \times \widehat{\uqlslthh} \to \F$ is non-degenerate in the sense that, if for every $x\in\uqgslthh$, $\left\langle x, y\right\rangle =0$ for some $y\in \widehat{\uqlslthh}$, then $y=0$. 
\end{prop}
\begin{proof}
Let $\{\mathcal O_n:n\in\N\}$ be any neighbourhood basis at $0\in\F$ for the discrete topology on $\F$. Then, let for every $n\in\N$,
$$ A_n:= \left\langle \uqgslthh , - \right\rangle^{-1}(\mathcal O_n) = \left \{ y\in \uqlslthh : \forall x\in \uqgslthh \qquad \left\langle x, y\right\rangle \in \mathcal O_n\right \}\,.$$
We clearly have, for every $n\in\N$, $\{0\} \subseteq A_n \subseteq  \uqlslthh$ and $A_{n} \supseteq A_{n+1}$. The non-degeneracy of the pairing further implies that
$$\bigcap_{n\in\N} A_n = \{0\}\,.$$
As a consequence, for every $n\in\N$ and every $y\in A_n-\{0\}$, there exists an $N\in\N$ such that for every $m\geq N$, $y\notin A_m$. Now, given $n_1\in\N$, let $\mu(n_1)\in\N$ be the largest integer such that $A_{n_1}\subseteq\Omega_{\mu(n_1)}$. By the previous discussion, for every point $y\in A_{n_1}-\Omega_{\mu(n_1)+1}$, there exists (a smallest) $n_2\in\N$ such that for every $m\geq n_2$, $y\notin A_{m}$. Hence, for every $m\geq n_2$, $A_m\subseteq \Omega_{\mu(n_1)+1}$ and we conclude that $\mu(n)=\mu(n_1)$ for every $n\in\range{n_1}{n_2-1}$, whereas $\mu(n_2)=\mu(n_1)+1$. By induction, it follows that $\mu:\N\to\N$ so defined is increasing and that, as a consequence, $\lim_{n\to+\infty} \mu(n) = +\infty$. We have therefore proven that, for every $n\in\N$,
\be\label{eq:contnondeg}\forall x\in \uqgslthh \qquad \left\langle x,y\right\rangle \in \mathcal O_n \quad \Rightarrow \quad y\in \Omega_{\mu(n)}\,.\ee
If we finally let $(y_n)_{n\in\N}\in \uqlslthh^\N$ be any Cauchy sequence that does not converge to $0$, the proposition is obviously equivalent to claiming that there exists an $x\in \uqgslthh$ such that 
$$\lim_{n\to+\infty}\left\langle x,y_n\right \rangle \neq 0\,.$$ 
Indeed, since $(y_n)_{n\in\N}$ does not converge to $0$, there exist $m\in\N$ such that for every $N\in\N$, $y_n\notin \Omega_m$ for some $n\geq N$. We can therefore extract a subsequence $(y_{n_k})_{k\in\N}$ such that $y_{n_k} \notin \Omega_m$ for every $k\in\N$. The contrapositive of (\ref{eq:contnondeg}) then implies that there exists $(x_{k})_{k\in\N}\in \uqgslthh^\N$ such that, for every $k\in\N$,
$$\left\langle x_{k}, y_{n_k} \right \rangle \notin \mathcal O_{\nu(m)}$$
 where $\nu(m)=\min \{ n\in\N: \mu(n) = m\}$. But since $(y_n)_{n\in\N}$ is Cauchy, so is $(y_{n_k})_{k\in\N}$ and, upon taking $k,l\in\N$ large enough, we can make $\left\langle x_{k}, y_{n_l}-y_{n_k}\right\rangle$ arbitrary small. This eventually concludes the proof.
\end{proof}

\section{Double quantum affinization in type $\mathfrak{a}_1$}
\label{sec:dqaff}
We now define and study the main object of interest in this paper; the double quantum affinization in type $\mathfrak a_1$, $\qdaff(\mathfrak a_1)$. We let $I=\{1\}$ be the labeling of the unique node of the type $\mathfrak a_1$ Dynkin diagram and we let $Q^\pm = \Z^\pm \alpha_1$. We denote by $Q=\Z\alpha_1$ the type $\mathfrak a_1$ root lattice.
\subsection{Definition of $\qdaff(\mathfrak a_1)$}
\begin{defn}
\label{defn:qdaff}
The \emph{double quantum affinization} $\qdaff(\mathfrak{a}_1)$ 
of type $\mathfrak{a}_1$ is defined as the $\F$-algebra generated by $$\{\Dsf_1, \Dsf_1^{-1}, \Dsf_2, \Dsf_2^{-1},\Csf^{1/2}, \Csf^{-1/2},   \csf+{m},\csf-{-m}, \Ksf+{1,0,m}, \Ksf-{1,0,-m}, \Ksf+{1,n,r}, \Ksf-{1,-n,r}, \Xsf+{1,r,s}, \Xsf-{1,r,s}:m\in\N, n\in \N^\times, r,s\in\Z\}$$
subject to the relations
\be\label{eq:Csfcentral} \mbox{$\Csf^{\pm 1/2}$ and $\cbsf\pm(z)$ are central} \ee
\be\label{eq:csbf} \cbsf+(v) \cbsf-(w) = \cbsf-(w) \cbsf+(v) = 1 \mod \frac{w}{v}\,,  \ee
\be \Dsf^{\pm 1}_1 \Dsf^{\mp 1}_1 = 1 \qquad \Dsf^{\pm 1}_2 \Dsf^{\mp 1}_2 = 1 \qquad \Dsf_1\Dsf_2= \Dsf_2\Dsf_1\ee
\be \Dsf_1 \Kbsf\pm{1,\pm m}(z) \Dsf_1^{-1} = q^{\pm m} \Kbsf\pm{1,\pm m}(z) \qquad \Dsf_1 \Xbsf\pm{1,r}(z) \Dsf_1^{-1} = q^{r} \Xbsf\pm{1,r}(z)\,, \ee
\be \Dsf_2 \Kbsf\pm{1,\pm m}(z) \Dsf_2^{-1} = \Kbsf\pm{1,\pm m}(zq^{-1}) \qquad \Dsf_2 \Xbsf\pm{1,r}(z) \Dsf_2^{-1} =  \Xbsf\pm{1,r}(zq^{-1})\,, \ee
\be \res_{v,w} \frac{1}{vw} \Kbsf\pm{1,0}(v) \Kbsf\mp{1,0}(w) 
= 1 \,,  \ee
\be\label{eq:K+K+} (v-q^{\pm 2}z)(v-q^{2(m-n \mp 1)}z) \Kbsf\pm{1,\pm m}(v) \Kbsf\pm{1,\pm n}(z) = (vq^{\pm 2}-z)(vq^{\mp2}-q^{2(m-n)}z) \Kbsf\pm{1,\pm n}(z) \Kbsf\pm{1,\pm m}(v)\,,\ee
\be\label{eq:K+K-} (\Csf q^{2(1-m)}v-w)(q^{2(n-1)}v-\Csf w) \Kbsf+{1, m}(v) \Kbsf-{1,-n}(w) = (\Csf q^{-2m}v - q^{2}w)(q^{2n}v-\Csf q^{-2}w) \Kbsf-{1,-n}(w)\Kbsf+{1, m}(v) \,,\ee
\be\label{eqbf:K+X+} (v-q^{\pm 2}z) \Kbsf\pm{1,\pm m}(v) \Xbsf\pm{1,r}(z) = (q^{\pm 2}v-z)\Xbsf\pm{1,r}(z)\Kbsf\pm{1,\pm m}(v)\,,\ee
\be\label{eq:K+X-} (\Csf v - q^{2(m\mp 1)}z) \Kbsf\pm{1,\pm m}(v) \Xbsf\mp{1,r}(z) = (\Csf q^{\mp 2} v - q^{2m}z) \Xbsf\mp{1,r}(z) \Kbsf\pm{1,\pm m}(v)\,,\ee
\be
(v-q^{\pm 2} w) \Xbsf\pm{1,r}(v) \Xbsf\pm{1,s}(w) =(vq^{\pm 2}- w) \Xbsf\pm{1,s}(w) \Xbsf\pm{1,r}(v)  \,,\label{eq:X+rX+s}\ee
\bea[\Xbsf+{1,r}(v),\Xbsf-{1,s}(z)] &=& \frac{1}{q-q^{-1}} \left \{\delta\left ( \frac{\Csf v}{q^{2(r+s)}z}  \right ) \prod_{p=1}^{|s|} \cbsf-\left(\Csf^{-1/2}q^{\left (2p-1\right ) \sign(s)-1}
z\right )^{-\sign(s)} \Kbsf+{1,r+s}(v) \right .\nn\\
&& \left . -\delta\left ( \frac{\Csf^{-1} v}{q^{2(r+s)}z}  \right ) \prod_{p=1}^{|r|} \cbsf+\left(\Csf^{-1/2}q^{\left (1-2p\right)\sign(r)-1}
v\right )^{\sign(r)} \Kbsf-{1,r+s}(z)\right \}\,,
\label{eq:X+X-}
\label{eq:X+X-KK}
\eea
where $m,n \in \N$, $r,s\in\Z$ and we have set
\be \cbsf\pm(z) = \sum_{m\in\N} \csf\pm{\pm m} z^{\mp m}\,,\ee
\be \Kbsf\pm{1,0}(z) = \sum_{m\in\N} \Ksf\pm{1,0,\pm m} z^{\pm m}\,,\ee
and, for every $m\in\N^\times$ and $r\in\Z$,
\be \Kbsf\pm{1,\pm m} (z) = \sum_{s\in\Z} \Ksf\pm{1,\pm m, s} z^{-s}\,,\ee
\be \Xbsf\pm{1,r}(z) = \sum_{s\in\Z} \Xsf\pm{1,r,s} z^{-s}\,.\ee
In (\ref{eq:X+X-KK}), we further assume that $\Kbsf\pm{1,\mp m}(z) =0$ for every $m\in\N^\times$.
\end{defn}
\begin{defn}
We define $\qdaff^0(\mathfrak a_1)$ as the subalgebra of $\qdaff(\mathfrak{a}_1)$ generated by 
$$\left\{ \Csf, \csf+{m},\csf-{-m}, \Ksf+{1,0,m} , \Ksf-{1,0,-m}, \Ksf+{1,n,r} , \Ksf-{1,-n,r} : m\in\N, n\in\N^\times,r\in\Z \right \}\,.$$
We define similarly $\qdaff^\pm(\mathfrak{a}_1)$ as the subalgebra of $\qdaff(\mathfrak{a}_1)$ generated by $\left \{\Xsf\pm{1,r,s} : r,s\in \Z\right \}$.
\end{defn}
\begin{rem}
Obviously, $\qdaff^\pm(\mathfrak{a}_1)$ is graded over $Q^\pm$ whereas $\qdaff(\mathfrak a_1)$ is graded over the root lattice $Q$ of $\mathfrak a_1$. $\qdaff(\mathfrak a_1)$ is also graded over $\Z^2 = \Z_{(1)}\times \Z_{(2)}$;
$$\qdaff(\mathfrak a_1) = \bigoplus_{(n_1, n_2)\in\Z^2} \qdaff(\mathfrak a_1)_{(n_1, n_2)}\,,$$
where, for every $(n_1, n_2)\in\Z^2$, we let
$$\qdaff(\mathfrak a_1)_{(n_1, n_2)} = \left\{x\in\qdaff(\mathfrak a_1) : \Dsf_1 x\Dsf_1^{-1} = q^{n_1} x, \quad \Dsf_2 x\Dsf_2^{-1} = q^{n_2} x \right \}\,.$$
\end{rem}
\begin{rem}
\label{rem:qdafftopology}
It is worth emphasizing that, were it not for relation (\ref{eq:X+X-}), the above definition of $\qdaff(\mathfrak{a}_1)$ would be purely algebraic. However, the r.h.s. of (\ref{eq:X+X-}) involves two infinite series and we may equip $\qdaff(\mathfrak a_1)$ with a topology, along the lines of what was done in section \ref{sec:topology} for $\qaff(\dot{\mathfrak a}_1)$, making use of the $\Z_{(2)}$-grading in order to construct a basis $\left\{\dot\Omega_n:n\in\N\right \}$ of open neighbourhoods of $0$. In that case, both series are convergent in the corresponding completion $\widehat{\qdaff(\mathfrak a_1)}$ and we shall further require that the subalgebras $\qdaff^-(\mathfrak a_1)$, $\qdaff^0(\mathfrak a_1)$ and $\qdaff^+(\mathfrak a_1)$ be defined as closed subalgebras of $\qdaff(\mathfrak a_1)$. We shall eventually denote with a hat their respective completions. An alternative point of view on this question, which might actually prove more useful when it comes to studying representation theory, consists in observing that $\qdaff(\mathfrak a_1)$ is \emph{proalgebraic}. Indeed, for every $N\in\N$, let $\qdaff(\mathfrak a_1)^{(N)}$ be the $\F$-algebra generated by 
$$\{\Csf^{1/2}, \Csf^{-1/2},   \csf+{n},\csf-{-n}, \Ksf+{1,0,m}, \Ksf-{1,0,-m}, \Ksf+{1,p,r}, \Ksf-{1,-p,r}, \Xsf+{1,r,s}, \Xsf-{1,r,s}: m\in\N, n\in\range{0}{N}, p\in\N^\times, r,s\in\Z\}$$ 
subject to relations ((\ref{eq:Csfcentral}) -- (\ref{eq:X+X-})), where, this time,
\be \cbsf\pm(z) = \sum_{m=0}^N \csf\pm{\pm m} z^{\mp m}\,.\ee
Now clearly, each $\qdaff(\mathfrak a_1)^{(N)}$ is algebraic since the sums on the r.h.s. of (\ref{eq:X+X-}) are both finite. Moreover, letting $\mathcal I_N$ be the two-sided ideal of $\qdaff(\mathfrak a_1)^{(N)}$ generated by $\{\csf+N, \csf-{-N}\}$ (resp. $\{\csf+0-1, \csf-0-1\}$) for every $N>1$ (resp. for $N=0$), we obviously have a surjective algebra homomorphism
\be \qdaff(\mathfrak a_1)^{(N)} \longrightarrow \qdaff(\mathfrak a_1)^{(N-1)} \cong \frac{\qdaff(\mathfrak a_1)^{(N)}}{\mathcal I_N}\ee
and we can define $\qdaff(\mathfrak a_1)$ as the inverse limit 
$$\qdaff(\mathfrak a_1) = \lim_{\longleftarrow} \qdaff(\mathfrak a_1)^{(N)}$$
of the system of algebras
$$\cdots \longrightarrow \qdaff(\mathfrak a_1)^{(N)} \longrightarrow \qdaff(\mathfrak a_1)^{(N-1)} \longrightarrow \cdots \longrightarrow \qdaff(\mathfrak a_1)^{(0)}\longrightarrow \qdaff(\mathfrak a_1)^{(-1)}\,.$$
\end{rem}

\begin{defn}
In $\widehat{\qdaff^0(\mathfrak a_1)}$, we define
$$\pbsf\pm (z) = \sum_{m\in\N} \psf\pm{\pm m} z^{\mp m} = \cbsf\pm(z) \Kbsf\mp{1,0}(\Csf^{-1/2}z)^{-1} \Kbsf\mp{1,0}(\Csf^{-1/2}zq^2)$$
and for every $m\in\N^\times$,
$$\tbsf+{1,m}(z) =  \sum_{n\in\N} \tsf+{1,m,n} z^{-n} =-\frac{1}{q-q^{-1}}\Kbsf+{1,0}(zq^{-2m})^{-1}\Kbsf+{1,m}(z)\,,$$
$$\tbsf-{1,-m}(z)= \sum_{n\in\N} \tsf-{1,- m,n} z^{n} = \frac{1}{q-q^{-1}} \Kbsf-{1,-m}(z) \Kbsf-{1,0}(zq^{-2m})^{-1}\,.$$
Then, we let $\qdaff^{0^+}(\mathfrak a_1)$ be the closed subalgebra of $\widehat{\qdaff^0(\mathfrak a_1)}$ generated by 
$$\{\Csf^{1/2},\Csf^{-1/2}, \psf+{m},\psf-{-m}, \tsf +{1,p,n}, \tsf -{1,-p,n}: m\in\N, n\in\Z, p \in \N^\times\}\,.$$
\end{defn}
\begin{defn}
We denote by $\qdaff'(\mathfrak a_1)$ the subalgebra of $\qdaff(\mathfrak a_1)$ generated by
$$\{\Dsf_2, \Dsf_2^{-1},\Csf^{1/2}, \Csf^{-1/2},   \csf+{m},\csf-{-m}, \Ksf+{1,0,m}, \Ksf-{1,0,-m}, \Ksf+{1,n,r}, \Ksf-{1,-n,r}, \Xsf+{1,r,s}, \Xsf-{1,r,s}:m\in\N, n\in \N^\times, r,s\in\Z\}\,,$$
i.e. the subalgebra generated by all the generators of $\qdaff(\mathfrak a_1)$ except $\Dsf_1$ and $\Dsf_1^{-1}$. We shall denote by
$$\jmath:\qdaff'(\mathfrak a_1) \hookrightarrow \qdaff(\mathfrak a_1)$$
the natural injective algebra homomorphism. We extend it by continuity into
$$\widehat \jmath : \widehat{\qdaff'(\mathfrak a_1)} \hookrightarrow \widehat{\qdaff(\mathfrak a_1)}\,.$$
\end{defn}
\noi The main result of the present paper is the following
\begin{thm}
\label{thm:main}
There exists a bicontinuous $\F$-algebra isomorphism $\widehat\Psi:\widehat{\qaff(\dot{\mathfrak a}_1)} \stackrel{\sim}{\longrightarrow} \widehat{\qdaff'(\mathfrak a_1)}$.
\end{thm}
\begin{proof}
Relations ((\ref{eq:K+K+})-(\ref{eq:K+X-})) respectively imply
\be\label{eq:K0+K0+} \Kbsf\pm{1,0}(v) \Kbsf\pm{1,0}(z) =  \Kbsf\pm{1,0}(z) \Kbsf\pm{1,0}(v)\,,\ee
\be\label{eq:K0+K0-} \Kbsf+{1, 0}(v) \Kbsf-{1,0}(w) = G_{11}^+(\Csf v/w)G_{11}^-(\Csf^{-1} v/w) \Kbsf-{1,0}(w)\Kbsf+{1, 0}(v) \ee
\be \Kbsf\pm{1,0}(v) \Xbsf\pm{1,r}(z) = G_{11}^\mp(v/z) \Xbsf\pm{1,r}(z)\Kbsf\pm{1,0}(v)\,,\ee
\be \Kbsf\pm{1,0}(v) \Xbsf\mp{1,r}(z) = G_{11}^\pm(\Csf v/z) \Xbsf\mp{1,r}(z) \Kbsf\pm{1,0}(v)\,,\ee
since $\Kbsf\pm{1,0}(z)\in \qdaff'(\mathfrak a_1)[[z^{\pm 1}]]$. It also easily follows from relation (\ref{eq:X+rX+s}) that
\be \left [ \Xbsf+{1,0}(v), \Xbsf+{1,-1}(w)\right]_{G_{11}^-(v/w)} = \delta\left(\frac{vq^{-2}}{w} \right) \Upsilon^+(w) \,,\ee
\be \left [ \Xbsf-{1,1}(v), \Xbsf-{1,0}(w)\right]_{G_{11}^+(v/w)} = \delta\left(\frac{vq^{2}}{w} \right) \Upsilon^-(w) \,,\ee
for some $\Upsilon^\pm(w)\in\widehat{\qdaff'(\mathfrak a_1)}[[w,w^{-1}]]$. Hence, the only possible obstructions to setting 
$$\Psi(D^{\pm 1}) = \Dsf_2^{\pm1} \qquad \Psi (C^{\pm 1/2})=\Csf^{\pm 1/2}\,,$$
$$\Psi(\kk\pm0(z))= -\cbsf\pm(z)\Kbsf\mp{1,0}(\Csf^{-1/2} z)^{-1} \qquad \Psi(\kk\pm1(z))= - \Kbsf\mp{1,0}(\Csf^{-1/2} z)$$
$$\Psi(\x+{0}(z)) =-\cbsf-(\Csf^{1/2} z) \Kbsf+{1,0}(z)^{-1} \Xbsf-{1,1}(\Csf z) \qquad \Psi(\x-{0}(z)) =-\Xbsf+{1,-1}(\Csf z)\cbsf+(\Csf^{1/2} z) \Kbsf-{1,0}(z)^{-1} $$
$$\Psi(\x\pm{1}(z)) = \Xbsf\pm{1,0}(z)\,,$$
and to extending it as an algebra homomorphism $\Psi:  \qaff(\dot{\mathfrak a}_1) \to \widehat{\qdaff'(\mathfrak a_1)}$ are $\Upsilon^\pm(w)$ and the images under $\Psi$ of the l.h.s. of the quantum Serre relations (\ref{eq:qaffserre}). We shall see in section \ref{sec:DamianiBeck} that both obstructions actually vanish. We also postpone until section \ref{sec:DamianiBeck} the construction of the continuous algebra homomorphism $\Psi^{-1}:\qdaff'(\mathfrak a_1)\longrightarrow \widehat{\qaff(\dot{\mathfrak a}_1)}$.
\end{proof}

\subsection{The subalgebra $\uqhhzeroslt$ and the elliptic Hall algebra}
Another remarkable feature of $\qdaff(\mathfrak a_1)$ and, more particularly of its subalgebra $\qdaff^0(\mathfrak a_1)$, is the existence of an algebra homomorphism onto it, from the elliptic Hall algebra that we now define.
\begin{defn}
Let $q_1, q_2, q_3$ be three (dependent) formal variables such that $q_1q_2q_3=1$. The \emph{elliptic Hall algebra} $\mathcal E_{q_1,q_2,q_3}$ is the $\Q(q_1,q_2,q_3)$-algebra generated by $\left \{ C, \psi^+_m, \psi^-_{-m}, e^+_n, e^-_n : m\in\N, n\in\Z\right \}$ subject to the relations
\be \mbox{$C$ is central\,,}\ee
\be\label{eq:psi+psi+} \psib\pm (z) \psib\pm(w) = \psib\pm(w)\psib\pm(z)\,,\ee
\be\label{eq:psi+psi-} g(Cz,w)g(Cw,z)\psib+(z)\psib-(w)=g(z,Cw)g(w,Cz) \psib-(w)\psib+(z)\,,\ee
\be\label{eq:psie+} g(C^{\frac{1\pm 1}{2}}z,w)\psib\pm(z) \eb+(w) = - g(w, C^{\frac{1\pm 1}{2}}z) \eb+(w) \psib\pm(z)  \,,\ee
\be\label{eq:psie-} g(w, C^{\frac{1\mp 1}{2}}z)  \psib\pm(z) \eb-(w)=  -g(C^{\frac{1\mp 1}{2}}z,w) \eb-(w) \psib\pm(z)\,, \ee
\be\label{eq:e+e-} [\eb+(z),\eb-(w) ] = \frac{1}{g(1,1)} \left [\delta\left (\frac{Cw}{z} \right )\psib+(w) -\delta\left (\frac{w}{Cz} \right )\psib-(z)\right ]\,,\ee
\be\label{eq:e+e+} g(z,w)\eb+(z)\eb+(w)= -g(w,z)\eb+(w)\eb+(z)\,,\ee
\be\label{eq:e-e-} g(w,z)\eb-(z)\eb-(w)= -g(z,w)\eb-(w)\eb-(z)\,,\ee
\be\label{eq:e+e+e+} \res_{v,w,z} (vwz)^m(v+z)(w^2-vz) \eb\pm(v) \eb\pm(w) \eb\pm(z)=0\,,\ee
where $m\in\Z$ and we have introduced
\be g(z,w) =(z-q_1w)(z-q_2w)(z-q_3w)\,, \ee
\be \psib\pm(z) = \sum_{m\in\N} \psi^\pm_{\pm m} z^{\mp m}\,,\ee
\be \eb\pm(z) = \sum_{m\in\Z} e^\pm_m z^{-m}\,.\ee
\end{defn}
\begin{rem}
The elliptic Hall algebra $\E{q_1}{q_2}{q_3}$ is $\Z$-graded and can be equipped with a natural topology along the lines of what we did for $\qaff(\dot{\mathfrak a}_1)$ in section \ref{sec:topology}. It then becomes a topological algebra and we denote by $\widehat{\E{q_1}{q_2}{q_3}}$ its completion.
\end{rem}
\begin{prop}
\label{prop:Hall}
There exists a unique continuous $\F$-algebra homomorphism $f:\widehat{\mathcal E_{q^{-4}, q^2,q^2}} \to \qdaff^{0^+}(\mathfrak a_1)$ such that
\be\label{eq:f(C)} f(C) = \Csf\,,\ee
\be f(\psib\pm(z)) =  (q^2-q^{-2})^2 \, \pbsf\pm (\Csf^{1/2} zq^{-2}) \,,\ee
\be\label{eq:f(ebpm)} f(\eb\pm(z)) =  \tbsf\pm{1,\pm 1}(z)
\,.\ee
\end{prop}
\begin{proof}
We prove that, starting from ((\ref{eq:f(C)}) -- (\ref{eq:f(ebpm)})), we can extend $f$ as an algebra homomorphism. For that purpose, it suffices to check the relations in $\E{q^{-4}}{q^2}{q^2}$, observing that, in addition to (\ref{eq:K0+K0+}) and (\ref{eq:K0+K0-}), we also have
\be \Kbsf\pm{1,0}(v) \Kbsf\pm{1,\pm 1}(z) = G_{11}^\mp (v/z)G_{11}^\pm(vq^2/z) \Kbsf\pm{1,\pm1}(z) \Kbsf\pm{1,0}(v)\,, \ee
\be \Kbsf\mp{1,0}(v) \Kbsf\pm{1,\pm 1}(w) = G_{11}^\mp(\Csf v/w) G_{11}^\pm(\Csf^{-1}q^2 v/w) \Kbsf\pm{1,\pm 1}(w)\Kbsf\mp{1,0}(v)\,,\ee
as direct consequences of (\ref{eq:K+K+}) and (\ref{eq:K+K-}) respectively, since $\Kbsf\pm{1,0}(z) \in \qdaff'(\mathfrak a_1)[[z^{\pm 1}]]$. One then easily obtains ((\ref{eq:psi+psi+}) -- (\ref{eq:psie-})) and ((\ref{eq:e+e+}) -- (\ref{eq:e-e-})). For example, we have
\bea 
g(v,z) f(\eb+(v)) f(\eb+(z))  &=& \frac{1}{(q-q^{-1})^2} g(v,z)G_{11}^+(z/v)G_{11}^-(zq^{-2}/v) \Kbsf+{1,0}(vq^{-2}) \Kbsf+{1,0}(q^{-2}z) \Kbsf+{1,1}(v)\Kbsf+{1,1}(z) \nn\\
&=&\frac{v-z}{(q-q^{-1})^2} (v-q^2 z )(v-q^{-2} z) \Kbsf+{1,0}(vq^{-2}) \Kbsf+{1,0}(q^{-2}z) \Kbsf+{1,1}(v)\Kbsf+{1,1}(z) \nn\\
\eea
\bea
&=&\frac{v-z}{(q-q^{-1})^2} (vq^2 - z )(vq^{-2}- z) \Kbsf+{1,0}(vq^{-2}) \Kbsf+{1,0}(q^{-2}z) \Kbsf+{1,1}(z)\Kbsf+{1,1}(v) \nn\\
&=&\frac{v-z}{(q-q^{-1})^2} (vq^2 - z )(vq^{-2}- z) G_{11}^+(vq^{-2}/z) G_{11}^-(v/z) \nn\\
&&\qquad \qquad \qquad \qquad \qquad \qquad \qquad \times \Kbsf+{1,0}(q^{-2}z) \Kbsf+{1,1}(z)\Kbsf+{1,0}(vq^{-2}) \Kbsf+{1,1}(v) \nn\\
&=& -g(z,v) f(\eb+(z)) f(\eb+(v))\,.\nn
\eea
Considering (\ref{eq:e+e-}), we observe that (\ref{eq:K+K-}) implies that there exist $\theta^\pm(z) \in \widehat{\qdaff'(\mathfrak a_1)}[[z,z^{-1}]]$ such that
$$\left [\Kbsf+{1,1}(v), \Kbsf-{1,-1}(w) \right ]_{G_{11}^+(\Csf vq^{-2}/w) G_{11}^-(\Csf^{-1}vq^2/w)} = \delta \left (\frac{\Csf v}{w} \right ) \theta^-(v) + \delta\left(\frac{v}{\Csf w} \right ) \theta^+(w)  $$
and one easily sees that
\bea \left[f(\eb+(v)), f(\eb-(w))\right ]  &=& -\frac{1}{(q-q^{-1})^2} \Kbsf+{1,0}(vq^{-2})^{-1} \left [\Kbsf+{1,1}(v), \Kbsf-{1,-1}(w) \right ]_{G_{11}^+(\Csf vq^{-2}/w) G_{11}^-(\Csf^{-1}vq^2/w)} \Kbsf-{1,0}(wq^{-2})^{-1}\nn\\
&=& -\frac{1}{(q-q^{-1})^2} \Kbsf+{1,0}(vq^{-2})^{-1} \left \{\delta \left (\frac{\Csf v}{w} \right ) \theta^-(v) + \delta\left(\frac{v}{\Csf w} \right ) \theta^+(w)\right \} \Kbsf-{1,0}(wq^{-2})^{-1}\,.\nn\eea
Therefore, it suffices to prove that
\be -\frac{1}{(q-q^{-1})^2} \Kbsf+{1,0}(\Csf wq^{-2})^{-1} \theta^+(w) \Kbsf-{1,0}(wq^{-2})^{-1} =  \frac{(q^2-q^{-2})^2}{g(1,1)} \pbsf+(\Csf^{1/2}q^{-2}w)\label{eq:thetap}\ee
\be -\frac{1}{(q-q^{-1})^2} \Kbsf+{1,0}(vq^{-2})^{-1} \theta^-(v) \Kbsf-{1,0}(\Csf vq^{-2})^{-1} =  -\frac{(q^2-q^{-2})^2}{g(1,1)} \pbsf-(\Csf^{1/2}q^{-2}v)\label{eq:thetam}\ee
We postpone the proof of ((\ref{eq:thetap}) -- (\ref{eq:thetam})), as well as that of
\be\label{eq:fe+fe+fe+} \res_{v,w,z} (vwz)^m(v+z)(w^2-vz) f(\eb\pm(v)) f(\eb\pm(w)) f(\eb\pm(z))=0\,,\ee
until section \ref{sec:DamianiBeck}.
\end{proof}

\noi We now naturally make the following
\begin{conj}
$f:\widehat{\mathcal E_{q^{-4}, q^2,q^2}} \to \qdaff^{0^+}(\mathfrak a_1)$ is a bicontinuous $\F$-algebra isomorphism.
\end{conj}
\begin{rem}
It is worth mentioning that the above conjecture is supported by the fact that, in view of ((\ref{eq:e+e+}) -- (\ref{eq:e-e-})), there clearly exists $\eb\pm_{\pm 2}(z)\in \widehat{\E{q_1}{q_2}{q_3}}[[z,z^{-1}]]$ such that
$${}_{G_{01}^\mp(q^{\mp 2}v/w) G_{11}^\mp(v/w)}\left[\eb\pm(w), \eb\pm(v)\right]_{G_{01}^\mp(q^{\mp 2}w/v)G_{11}^\mp(w/v)} = \pm \left [2\right ]_q \left\{\delta\left (\frac{q^{2}v}{w}\right ) \eb\pm_{\pm 2}(w) -\delta\left (\frac{wq^{2}}{v}\right ) \eb\pm_{\pm 2}(v)\right\} $$
and that we can therefore set
$$f^{-1}(\tbsf\pm{1,\pm 2} (v)) = \eb\pm_{\pm 2}(v)\,.$$
In order to complete the proof, one would similarly need to construct $f^{-1}(\tbsf\pm{1,\pm m}(v))$ for any $m>2$.
\end{rem}

\subsection{$\qaff(\mathfrak{a}_1)$ subalgebras of $\qdaff(\mathfrak{a}_1)$}
Interestingly, $\qdaff(\mathfrak{a}_1)$ admits countably many embeddings of the quantum affine algebra $\qaff(\mathfrak{a}_1)$. This is the content of the following
\begin{prop}
For every $m\in\Z$, there exists a unique injective algebra homomorphism $\iota_m : \qaff(\mathfrak{a}_1)\hookrightarrow\widehat\qdaff(\mathfrak{a}_1)$ such that
\be \iota_m(C^{\pm 1/2}) = \Csf^{\pm 1/2} \qquad \iota_m (D^{\pm 1}) = \Dsf_2^{\pm1}\label{eq:iotamC}\ee
\be \iota_m(\kk\pm 1(z)) =  -\prod_{p=1}^{|m|} \cbsf\pm\left(q^{2\left (p\pm \frac{\sign(m)\pm1}{2}\right )}z\right )^{\pm\sign(m)} \Kbsf\mp{1,0}(\Csf^{-1/2}z)\,,\ee
\be \iota_m(\x\pm1(z)) = \Xbsf\pm{1,\pm m}(z)\,.\label{eq:iotamx}\ee
\end{prop}
\begin{proof}
Let $\iota^{(1)}:\qaff(\mathfrak a_1) \hookrightarrow \qaff(\dot{\mathfrak a}_1)$ be the injective algebra homomorphism mapping $\qaff(\mathfrak a_1) $ to the Dynkin diagram subalgebra of $\qaff(\dot{\mathfrak a}_1)$ associated with the vertex labeled $1\in\dot{I}$ -- see section \ref{sec:defnqtor}. It naturally extends to an injective algebra homomorphism $\widehat\iota^{(1)} : \qaff(\mathfrak a_1) \hookrightarrow \widehat{\qaff(\dot{\mathfrak a}_1)}$. Then, let for every $m\in\Z$, $\iota_m$ be the composite
$$\iota_m : \qaff(\mathfrak a_1)  
\underset{\widehat\iota^{(1)}}{\xhookrightarrow{\quad}}
 \widehat{\qaff(\dot{\mathfrak a}_1)} \underset{Y^{-m}}{\stackrel{\sim}{\longrightarrow}} \widehat{\qaff(\dot{\mathfrak a}_1)}\underset{\widehat\Psi}{\stackrel{\sim}{\longrightarrow}} \widehat{\qdaff'(\mathfrak{a}_1)} \underset{\widehat \jmath}{\xhookrightarrow{\quad}}\widehat{\qdaff(\mathfrak{a}_1)}\,.$$
Thus, $\iota_m$ is clearly injective. Moreover, one easily checks ((\ref{eq:iotamC}) -- (\ref{eq:iotamx})) -- see next section.
\end{proof}

\subsection{Automorphisms of $\widehat{\qdaff'(\mathfrak a_1)}$}
$\widehat{\qdaff'(\mathfrak a_1)}$ naturally inherits, through $\widehat \Psi$, the automorphisms defined over $\widehat{\qaff(\dot{\mathfrak a}_1)}$ in the previous section.
\begin{prop}
Conjugation by $\widehat \Psi$ clearly provides a group isomorphism $\Aut{\widehat{\qaff(\dot{\mathfrak a}_1)}} \cong \Aut{\widehat{\qdaff'(\mathfrak a_1)}}$. In particular, for every $f\in \Aut{\widehat{\qaff(\dot{\mathfrak a}_1)}}$, we let $\dot f = \widehat \Psi \circ f \circ \widehat \Psi^{-1} \in \Aut{\widehat{\qdaff'(\mathfrak a_1)}}$.
\end{prop}

\subsection{Triangular decomposition of $\widehat{\qdaff(\mathfrak a_1)}$}
\begin{defn}
Let $A$ be a complete  topological algebra with closed subalgebras $A^\pm$ and $A^0$. We shall say that $(A^-,A^0,A^+)$ is a \emph{triangular decomposition} of $A$ if the multiplication induces a bicontinuous isomorphism of vector spaces $A^-\widehat\otimes A^0 \widehat\otimes A^+ \stackrel{\sim}{\rightarrow} A$.  
\end{defn}
\noi In order to prove the triangular decomposition of $\widehat{\qdaff(\mathfrak a_1)}$, we shall make use of the following classic
\begin{lem}
\label{lem:triangulardecomp}
Let $A$ be a complete topological algebra with a triangular decomposition $(A^-,A^0,A^+)$. Let $\mathcal I^\pm$ be a closed two-sided ideal of $A^\pm$ such that $\mathcal I^+. A \subseteq A.\mathcal I^+$ and $A.\mathcal I^- \subseteq \mathcal I^-. A$. Then the quotient algebra $B=A/(A.(\mathcal I^++\mathcal I^-).A)$ admits a triangular decomposition $(B^-, A^0, B^+)$ where $B^\pm$ is the set of equivalence classes of $A^\pm$ in $B$. Moreover, there exists a bicontinuous algebra isomorphism $B^\pm \cong A^\pm/\mathcal I^\pm$. 
\end{lem}
\begin{proof}
See e.g. \cite{Jantzen}.
\end{proof}
\noi Recalling the definitions of $\qdaff^\pm(\mathfrak a_1)$ and $\qdaff^0(\mathfrak a_1)$ from definition \ref{defn:qdaff}, we have 
\begin{prop}
\label{prop:triang}
$(\qdaff^-(\mathfrak a_1), \qdaff^0(\mathfrak a_1), \qdaff^+(\mathfrak a_1))$ is a triangular decomposition of $\widehat{\qdaff(\mathfrak a_1)}$ and $\qdaff^\pm(\mathfrak a_1)$ is bicontinuously isomorphic to the algebra generated by 
$\{\Xsf\pm{1,r, s} : r,s\in\Z\}$ subject to relation (\ref{eq:X+rX+s}).
\end{prop}
\begin{proof}
Let $A$ be the $\F$-algebra generated by
$$\{\Dsf_1, \Dsf_1^{-1}, \Dsf_2, \Dsf_2^{-1},\Csf^{1/2}, \Csf^{-1/2},   \csf+{m},\csf-{-m}, \Ksf+{1,0,m}, \Ksf-{1,0,-m}, \Ksf+{1,n,r}, \Ksf-{1,-n,r}, \Xsf+{1,r,s}, \Xsf-{1,r,s}:m\in\N, n\in \N^\times, r,s\in\Z\}$$
subject to the relations ((\ref{eq:csbf}) -- (\ref{eq:K+X-})) and (\ref{eq:X+X-}), i.e. all the defining relations of $\qdaff(\mathfrak a_1)$ but relation (\ref{eq:X+rX+s}). Endow $A$ with a topology along the lines of what was done in section \ref{sec:topology} for $\qaff(\dot{\mathfrak a}_1)$, making use of its $\Z_{(2)}$-grading. This yields a basis $\left\{\dot\Omega_n:n\in\N\right \}$ of open neighbourhoods of $0$. Let furthermore $A^0$ be the closed subalgebra of $A$ generated by 
$$\left\{\Dsf_1, \Dsf_1^{-1}, \Dsf_2, \Dsf_2^{-1},\Csf^{1/2}, \Csf^{-1/2},   \csf+{m},\csf-{-m}, \Ksf+{1,0,m}, \Ksf-{1,0,-m}, \Ksf+{1,n,r}, \Ksf-{1,-n,r} :m\in\N, n\in \N^\times, r\in\Z \right\}$$ 
and $A^\pm$ be the closed subalgebra of $A$ generated by $\left\{\Xsf\pm{1,r,s} : r,s\in\Z \right \}$. An easy recursion proves that relations (\ref{eqbf:K+X+}) and (\ref{eq:K+X-}) imply that, for every $N\in\N$ and every $m\in\N$, $l,r,s\in\Z$,
\be \Xsf+{1,r,s} \Ksf+{1,m,l} -q^2\Ksf+{1,m,l} \Xsf+{1,r,s}  -(q^2-q^{-2}) \sum_{p=1}^N q^{2p} \Ksf+{1,m,l+p}  \Xsf+{1,r,s-p}  +q^{2N}  \Ksf+{1,m,l+N+1} \Xsf+{1,r,s-N-1}  \in\dot\Omega_{\nu_{s,l}^+(N)}\nn\ee
\be \Ksf-{1,-m,l} \Xsf-{1,r,s} - q^{-2} \Xsf-{1,r,s} \Ksf-{1,-m,l} +(q^2-q^{-2}) \sum_{p=1}^N q^{-2p} \Xsf-{1,r,s+p}\Ksf-{1,-m,l-p} +q^{2N} \Xsf-{1,r,s+N+1}\Ksf-{1,-m,l-N-1}\in\dot\Omega_{\nu_{s,l}^-(N)}\nn\ee
\bea&&\Ksf+{1,m,l} \Xsf-{1,r,s} -q^{-2} \Xsf-{1,r,s}  \Ksf+{1,m,l} +(q^2-q^{-2}) \sum_{p=1}^N \Csf^{-p} q^{2p(m-1)} \Xsf-{1,r,s+p}\Ksf+{1,m,l-p}\nn\\ 
&&\qquad\qquad\qquad \qquad\qquad\qquad \qquad \qquad \qquad \qquad +\Csf^{-(N+1)} q^{2(N+1)(m-1)+2}\Xsf-{1,r,s+N+1}\Ksf+{1,m,l-N-1} \in\dot\Omega_{\nu_{s,l}^-(N)} \nn\eea
\bea &&\Xsf+{1,r,s} \Ksf-{1,-m,l} -q^2  \Ksf-{1,-m,l} \Xsf+{1,r,s}-(q^2-q^{-2}) \sum_{p=1}^N \Csf^p q^{2p(1-m)}  \Ksf-{1,-m,l+p} \Xsf+{1,r,s-p} \nn\\
&&\qquad\qquad\qquad\qquad\qquad\qquad\qquad\qquad\qquad\qquad\qquad  +\Csf^{N+1} q^{2(N+1)(1-m)} \Ksf-{1,-m,l+N+1}\Xsf+{1,r,s-N-1} \in\dot\Omega_{\nu_{s,l}^+(N)}\nn\eea
where $\nu_{s,l}^\pm(N) = \min(\pm l,\mp s)+N+1$. It obviously follows that $(A^-,A^0,A^+)$ is a triangular decomposition of $A$. Now let $\mathcal I^\pm$ be the closed two-sided ideal of $A^\pm$ generated by
$$\left\{\Xsf\pm{1,r,m+1}\Xsf\pm{1,s,n}-q^{\pm 2} \Xsf\pm{1,r,m}\Xsf\pm{1,s,n+1} - q^{\pm 2} \Xsf\pm{1,s,n}\Xsf\pm{1,r,m+1} + \Xsf\pm{1,s,n+1}\Xsf\pm{1,r,m} : r,s,m,n\in\Z\right \}\,.$$
Clearly $\qdaff(\mathfrak a_1)\cong A/(A.(\mathcal I^++\mathcal I^-).A)$. In view of the above rewritings of (\ref{eqbf:K+X+}) and (\ref{eq:K+X-}), it is clear that $\mathcal I^+. A^0 \subseteq A^0. \mathcal I^+$ and $A^0 .\mathcal I^- \subseteq \mathcal I^-. A^0$. Moreover, relations (\ref{eqbf:K+X+}), (\ref{eq:K+X-}) and (\ref{eq:X+X-}) are easily shown to imply that, for every $r,s,t\in\Z$,
$$\left[ (v-q^{\pm 2} w) \Xbsf\pm{1,r}(v) \Xbsf\pm{1,s}(w) -(vq^{\pm 2}- w) \Xbsf\pm{1,s}(w) \Xbsf\pm{1,r}(v) , \Xbsf\mp{1,t}(u)
\right ] = 0\,,$$
hence proving that $\mathcal I^+ .A^- \subseteq A .\mathcal I^+$ and $A^+ .\mathcal I^- \subseteq \mathcal I^- .A$. The claim eventually follows as a consequence of lemma \ref{lem:triangulardecomp}
\end{proof}

\subsection{Quasi-finite highest $t$-weight modules}
\begin{defn}
We shall say that a (topological) $\qdaff(\mathfrak a_1)$-module $M$ is a \emph{$t$-weight} module if there exists a countable set $\left\{M_\alpha : \alpha\in A\right \}$ of simple $\uqhhzeroslt$-modules, called \emph{$t$-weight spaces} of $M$, such that, as $\uqhhzeroslt$-modules,
\be M \cong \bigoplus_{\alpha\in A} M_\alpha\,.\ee
We shall say that $M$ is \emph{quasi-finite} if, in addition, $A$ is finite. A vector $v\in M-\{0\}$ is a \emph{highest $t$-weight vector} of $M$ if $v\in M_\alpha$ for some $\alpha\in A$ and, for every $r,s\in\Z$,
\be \Xsf+{1,r,s} . v = 0\,.\ee
We shall say that $M$ is \emph{highest $t$-weight} if $M\cong \qdaff(\mathfrak a_1) . v$ for some highest $t$-weight vector $v\in M-\{0\}$.
\end{defn}
\begin{defprop}
Let $M$ be a $t$-weight $\qdaff(\mathfrak a_1)$-module that admits a highest $t$-weight vector $v\in M-\{0\}$. Denote by $M_0$ the $t$-weight space of $M$ containing $v$. Then, for every $r,s\in\Z$,
\be \Xsf+{1,r,s} .M_0 = \{0\}\,.\ee
We shall say that $M_0$ is a \emph{highest $t$-weight space} of $M$. 
\end{defprop}
\begin{proof}
It is an easy consequence of the triangular decomposition of $\qdaff(\mathfrak a_1)$ -- see propsition \ref{prop:triang} -- and of the root grading that, indeed, $\Xsf+{1,r,s}.\left (\qdaff^0(\mathfrak a_1).v\right ) =\{0\}$, for every $r,s\in\Z$. Now $M_0 \cong \qdaff^0(\mathfrak a_1).v$, since $\qdaff^0(\mathfrak a_1).v \neq \{0\}$ is a submodule of the simple $\qdaff^0(\mathfrak a_1)$-module $M_0$.
\end{proof}
\begin{prop}
Any quasi-finite simple (topological) $\qdaff(\mathfrak a_1)$-module is highest $t$-weight. Moreover, for any highest $t$-weight vector $v\in M-\{0\}$, we have
\be\label{eq:Mhgighesttweight}M\cong \qdaff^-(\mathfrak a_1). \qdaff^0(\mathfrak a_1) . v\,.\ee
\end{prop}

\begin{proof}
Let $M$ be a quasi-finite simple (topological) $\qdaff(\mathfrak a_1)$-module and assume for a contradiction that, for every $w\in M$, there exist two sequences $(r_n)_{n\in\N}, (s_n)_{n\in\N} \in\Z^\N$, such that 
$$0\notin \left\{ w_n = \Xsf+{r_1,s_1} \dots \Xsf+{r_n,s_n} .w : n\in\N\right\}\,.$$ 
Choosing $w\in M-\{0\}$ to be an eigenvector of $\Ksf+{1,0,0}$ with eigenvalue $\lambda$, one easily sees from the relations that, for every $n\in\N$, $\Ksf+{1,0,0}.w_n = \lambda q^{2n} w_n$. As a consequence of the commutativity of $\Ksf+{1,0,0}$ with the generators of $\qdaff^0(\mathfrak a_1)$, no two $w_n$ can be elements of the same $t$-weight space. But that implies that there should be infinitely many $t$-weight spaces, a contradiction with the assumed quasi-finiteness of $M$. Thus, we conclude that there exists a highest $t$-weight vector $v\in M$. Obviously, $M\cong \qdaff(\mathfrak a_1). v$, for $\qdaff(\mathfrak a_1). v \neq\{0\}$ is a submodule of the simple $\qdaff(\mathfrak a_1)$-module $M$. Thus $M$ is highest $t$-weight. Moreover, (\ref{eq:Mhgighesttweight}) follows by the triangular decomposition of $\qdaff(\mathfrak a_1)$ -- see proposition \ref{prop:triang} -- and the fact that $v$ is a highest $t$-weight vector.
\end{proof}

\begin{rem}
In view of (\ref{eq:Mhgighesttweight}), quasi-finite simple (topological) $\qdaff(\mathfrak a_1)$-module are entirely determined as $M\cong \qdaff^-(\mathfrak a_1). M_0$, by the data of their unique highest $t$-weight space $M_0 \cong \qdaff^0(\mathfrak a_1) . v$. Classifying quasi-finite simple (topological) $\qdaff(\mathfrak a_1)$-modules therefore amounts to classifying those simple $\qdaff^0(\mathfrak a_1)$-modules that appear as their highest $t$-weight spaces. We intend to undertake that classification in a future work.
\end{rem}

\begin{rem}
(\ref{eq:Mhgighesttweight}) induces a partial ordering of the $t$-weight spaces through the $Q^-$-grading of $\qdaff^-(\mathfrak a_1)$.
\end{rem}

\begin{defn}
For every $N\in\N^\times$, we shall say that a module $M$ over $\qdaff(\mathfrak a_1)$ is of type $(1,N)$ if
\begin{enumerate}
\item[i.] $\Csf^{\pm 1/2}$ acts as $1$ on $M$;
\item[ii.] $\Ksf\pm{1, 0,0}$ acts semisimply on $M$;
\item[iii.] $\csf\pm{\pm m}$ acts as $0$ on $M$, for every $m\geq N$.
\end{enumerate}
We shall say that $M$ is of type $(1,0)$ if points i. and ii. above hold and, in addition, $\cbsf\pm(z)$ acts as $1$ on $M$.
\end{defn}
\begin{rem}
Let $N\in\N$. Then the $\qdaff(\mathfrak a_1)$-modules of type $(1,N)$ are in one-to-one correspondence with the $\qdaff(\mathfrak a_1)^{(N)}$-modules -- see remark \ref{rem:qdafftopology} for a definition of $\qdaff(\mathfrak a_1)^{(N)}$.
\end{rem}

\subsection{Topological Hopf algebra structure on $\widehat{\qdaff'(\mathfrak a_1)}$}
\begin{defprop}
We define
\be \dot\Delta= \left (\widehat\Psi \widehat{\otimes} \widehat\Psi\right ) \circ \Delta\circ \widehat \Psi^{-1} \,,\ee
\be \dot S = \widehat\Psi\circ S\circ\widehat\Psi^{-1}\,,\ee
\be \dot\varepsilon = \varepsilon\circ\widehat\Psi^{-1}\,.\ee
Equipped with the above comultiplication, antipode and counit, $\widehat{\qdaff'(\mathfrak a_1)}$ is a topological Hopf algebra. The latter is easily extended into a topological Hopf algebraic structure on $\widehat{\qdaff(\mathfrak a_1)}$ by setting, in addition,
$$\Delta(\Dsf_1^{\pm 1}) = \Dsf_1^{\pm 1} \otimes \Dsf_1^{\pm 1}\,,$$
$$S(\Dsf_1^{\pm 1}) = \Dsf_1^{\mp 1}\,,$$
$$\varepsilon(\Dsf_1^{\pm 1}) = 1\,.$$
\end{defprop}

\section{Doubly Affine Damiani-Beck isomorphism}
\label{sec:DamianiBeck}
In this last section, we complete the proof of theorem \ref{thm:main} by constructing  $\Psi^{-1}:\qdaff'(\mathfrak a_1)\to\uqsltc$; i.e. by constructing a realization of the generators of $\qdaff'(\mathfrak a_1)$ in $\uqsltc$.

\subsection{Double loop generators}
\begin{defn}
For every $m \in \Z$, we set $\X\pm{1,m}(z):= Y^{\mp m}(\x\pm 1(z))$.
\end{defn}
\noi It is clear that
\begin{prop}
For every $m\in\Z$, we have
\be\varphi \left (\X\pm{1,m}(z)\right ) = \X\mp{1,-m}\left(1/z\right )\,.\ee
\end{prop}

\begin{prop}
\begin{enumerate}
\item[i.]\label{defpsi+11} There exists a unique $\bpsi +{1,1}(z) \in \uqsltc [[z, z^{-1}]]$ such that
\be\label{eq:defpsi+11} \left [Y\left (\km 1(w)^{-1} \xm 1(C^{1/2}w) \right ) , \xp1(z) \right ]_{G_{10}^-(C^{-1/2}w/z)} = -\delta \left ( \frac{C^{-1/2} q^2 w}{z} \right ) \bpsi +{1,1}(z)\,. \ee
\item[ii.]\label{defpsi-1-1} Set $\bpsi -{1,-1}(z) = \varphi\left(\bpsi +{1,1}(1/z)\right )$. Then, we have
\be\label{eq:defpsi-1-1} \left [\xm1(z) , Y\left ( \xp 1(C^{1/2}w) \kp 1(w)^{-1}\right ) ,  \right ]_{G_{10}^+(C^{1/2}z/w)} = -\delta \left ( \frac{C^{-1/2} q^2 w}{z} \right ) \bpsi -{1,-1}(z)\,.\ee
\end{enumerate}
\end{prop}
\begin{proof}
The proof of \emph{i.} is immediate from the definitions. \emph{ii.} then follows by applying $\varphi$ to (\ref{eq:defpsi+11}).
\end{proof}
\begin{rem}
It is worth noting that $\bpsi \pm{1,\pm 1}(z) \notin \uqslthh[[z, z^{-1}]]$.
\end{rem}
\begin{cor}
\label{cor:kpsi1}
For every $i\in\dot I$, we have
\begin{enumerate}
\item[i.] $\kk-i(v) \bpsi\pm{1,\pm 1}(z) = G_{i,0}^\mp(C^{\mp 1/2} q^{2}v/z) G_{i,1}^\mp (C^{\mp 1/2}v/z) \bpsi\pm{1,\pm 1} (z)\kk-i(v)$;
\item[ii.] $\bpsi\pm{1,\pm 1}(z) \kk+i(v)  = G_{i,0}^\mp(C^{\mp 1/2} q^{-2}z/v) G_{i,1}^\mp (C^{\mp 1/2}z/v) \kk+i(v) \bpsi\pm{1,\pm 1} (z)$;
\end{enumerate}
\end{cor}
\begin{proof}
\emph{ii.} follows by applying $\varphi$ to \emph{i.} and \emph{i.} is a direct consequence of  (\ref{eq:defpsi+11}) and (\ref{eq:defpsi-1-1}) on one hand and of (\ref{eq:kpxpm}) and (\ref{eq:kmxpm}) on the other hand.
\end{proof}

Let us then define the following $\uqslthh$-valued formal power series
\be \Gam \pm(z) := \kk \pm 0(z) \kk \pm1(z) \in \uqslthh[[z^{\mp 1}]]\, .\ee
Denoting by ${\mathcal Z}(\uqslthh)$ the center of $\uqslthh$, it is straightforward to check that indeed
\begin{prop}
$\Gam \pm(z) \in {\mathcal Z}(\uqslthh)[[z^{\mp1}]]$.
\end{prop}
\noi Similarly, define
\be\label{eq:defwp} \bwp \pm(z):= \kk \pm 0(z) \kk \pm 1(z q^2) \in  \uqslthh[[z^{\mp 1}]]\,. \ee
\noi Then we establish an important result.
\begin{prop}
\label{prop:psiinv}
We have the following fixed points of $Y$;
\be\label{eq:Ywp} Y \left ( \bwp \pm(z)\right) = \bwp \pm(z)\, ,\ee
\be\label{eq:Ypsi} Y\left( \bpsi \pm{1,\pm 1}(z) \right ) = \bpsi \pm{1,\pm 1}(z)\, . \ee
Moreover
\be\label{eq:Ygamma} Y\left (\bz{\pm}{0}(z)\right ) = \bz{\pm}{0}(zq^2)\,,\ee
\end{prop}
\begin{proof}
(\ref{eq:Ywp}) and (\ref{eq:Ygamma}) are obvious. We prove (\ref{eq:Ypsi}) for the upper choice of signs. In order to do so, we first rewrite (\ref{eq:defpsi+11}) as
\be\left [\xp0(w) , \xp1(z) \right ]_{G_{10}^-(w/z)} = \delta \left ( \frac{q^2 w}{z} \right ) \bpsi +{1,1}(z)\,. \nn\ee
Now, (\ref{eq:Tx0+}) and the definition of $Y$ imply that, on one hand,
\bea
&&\left [\left [ \x +0(z_1), \left [ \x +0(z_2),  \x +1(wq^2)\right]_{G_{10}^-(z_2/wq^2)} \right]_{G_{11}^-(z_1/z_2) G_{10}^-(z_1/wq^2)}, \x-0(C^{-1}z) \kp0(C^{-1/2}z)^{-1}\right ]_{G_{10}^-(w/z)}\nn\\
&=& -\left[2\right]_q \delta\left(\frac{z_1}{z_2q^2}\right ) \delta\left(\frac{z_2}{w}\right ) \left [Y\left (\xp0(w)\right ) ,  Y\left (\xp1(z)\right )\right ]_{G_{10}^-(w/z)} \nn\\
&=&-\left[2\right]_q \delta\left(\frac{z_1}{z_2q^2}\right ) \delta\left(\frac{z_2}{w}\right ) Y\left ( \left [\xp0(w) ,  \xp1(z)\right ]_{G_{10}^-(w/z)}\right ) =-\left[2\right]_q \delta\left(\frac{z_1}{z_2q^2}\right ) \delta\left(\frac{z_2}{w}\right ) \delta\left(\frac{wq^2}{z}\right ) Y\left ( \bpsi+{1,1}(z)\right ) \,,
 \nn
\eea
whereas, on the other hand, (\ref{eq:Tx0+}), (\ref{eq:kpxpm}), (\ref{eq:kmxpm}) and (\ref{eq:relx+x-}), as well as corollary \ref{cor:kpsi1}, imply that
\bea &&\left [\left [ \x +0(z_1), \left [ \x +0(z_2),  \x +1(wq^2)\right]_{G_{10}^-(z_2/wq^2)} \right]_{G_{11}^-(z_1/z_2) G_{10}^-(z_1/wq^2)}, \x-0(C^{-1}z) \kp0(C^{-1/2}z)^{-1}\right ]_{G_{10}^-(w/z)}\nn\\
&=& \left [\left [ \x +0(z_1), \left [ \x +0(z_2),  \x +1(wq^2)\right]_{G_{10}^-(z_2/wq^2)} \right]_{G_{11}^-(z_1/z_2) G_{10}^-(z_1/wq^2)}, \x-0(C^{-1}z) \right ] \kp0(C^{-1/2}z)^{-1}\nn\\
&=& \frac{1}{q-q^{-1}} \left\{  \delta \left ( \frac{z_1}{z} \right ) \delta\left(\frac{z_2}{w}\right )\left [\kk + 0(z_1C^{-1/2}) , \bpsi+{1,1}(wq^2) \right]_{G_{11}^-(z_1/z_2) G_{10}^-(z_1/wq^2)}  \right. \nn\\
&&+\left .\delta \left ( \frac{z_2}{z} \right ) \left [ \x +0(z_1), \left [ \kk + 0(z_2C^{-1/2}) ,  \x +1(wq^2)\right]_{G_{10}^-(z_2/wq^2)} \right]_{G_{11}^-(z_1/z_2) G_{10}^-(z_1/wq^2)} 
\right \}\kp0(C^{-1/2}z)^{-1}\nn\\
&=&  \delta \left ( \frac{z_1}{z} \right ) \delta\left(\frac{z_2}{w}\right ) \frac{G_{00}^+(w/z_1) G_{01}^+(q^2w/z_1)-G_{11}^-(z_1/w) G_{10}^-(z_1/wq^2)}{q-q^{-1}} \bpsi+{1,1}(wq^2) \nn\\
&&+\delta \left ( \frac{z_2}{z} \right ) \delta\left (\frac{z_1}{w}\right ) \frac{G_{01}^+(q^2w/z_2)-G_{10}^-(z_2/wq^2)}{q-q^{-1}} \bpsi+{1,1}(wq^2)\,.\nn
\eea
Making use of (\ref{eq:G+G-}) and (\ref{eq:GG+GG-}) -- for the latter, see Appendix --, we eventually get
\bea
&&\left [\left [ \x +0(z_1), \left [ \x +0(z_2),  \x +1(wq^2)\right]_{G_{10}^-(z_2/wq^2)} \right]_{G_{11}^-(z_1/z_2) G_{10}^-(z_1/wq^2)}, \x-0(C^{-1}z) \kp0(C^{-1/2}z)^{-1}\right ]_{G_{10}^-(w/z)}\nn\\
&=&  \left[2\right]_q\delta \left ( \frac{z_1}{z} \right ) \delta\left(\frac{z_2}{w}\right )\left [ \delta\left (\frac{w}{z_1}\right ) -\delta\left (\frac{wq^2}{z_1}\right ) \right ]  \bpsi+{1,1}(wq^2) -\left[2\right ]_q\delta \left ( \frac{z_2}{z} \right ) \delta\left (\frac{z_1}{w}\right ) \delta\left(\frac{w}{z_2}\right ) \bpsi+{1,1}(wq^2)\nn\\
&=&-\left[2\right]_q \delta \left ( \frac{z_1}{z} \right ) \delta\left(\frac{z_2}{w}\right ) \delta\left (\frac{wq^2}{z_1}\right )   \bpsi+{1,1}(z)\,,\nn
\eea
thus proving the result. The case with lower choice of signs follows by applying $\varphi$.
\end{proof}

\begin{prop}
\label{prop:psixm}
For every $m \in \Z$, we have
\begin{enumerate}
\item[i.] $\left [ \bpsi +{1,1}(z), \Xm{1,m}(v)  \right ] = - [2]_q \delta \left ( \frac{Cz}{v} \right ) \bwp -(C^{1/2}q^{-2}z) \Xm{1,m+1}(Cq^{-2}z)$;
\item[ii.] $\left [ \bpsi +{1,1}(z), \Xp{1,m}(v)  \right ]_{G_{10}^-(z/vq^2)G_{11}^-( z/v)} = [2]_q \delta \left ( \frac{z}{vq^2} \right ) \Xp{1,m+1}(z)$.
\item[iii.] $\left [ \bpsi-{1,-1}(z), \Xp{1,-m}(v)  \right ] =  [2]_q \delta \left ( \frac{Cz}{v} \right )  \Xp{1,-(m+1)}(Cq^{-2}z) \bwp +(C^{1/2}q^{-2}z)$;
\item[iv.] ${}_{G_{10}^+(vq^2/z)G_{11}^+( v/z)}\left [ \bpsi -{1,-1}(z), \Xm{1,-m}(v)  \right ] = -[2]_q \delta \left ( \frac{z}{vq^2} \right ) \Xm{1,-(m+1)}(z)$.
\item[v.] $\left[\bpsi+{1,1}(z), \bpsi-{1,-1}(v)\right ] =  \frac{[2]_q}{q-q^{-1}} \left [\delta \left (\frac{z}{Cv}\right ) \bwp+(C^{-1/2}q^{-2}z) - \delta \left (\frac{Cz}{v}\right ) \bwp-(C^{-1/2}q^{-2}v) \right]$.
\end{enumerate}
\end{prop}
\begin{proof}
\emph{i.} and \emph{ii.} are readily checked for $m=0$. Then, assuming they hold for some $m\in \Z$ and applying $Y^{\pm 1}$, it follows from propositon \ref{prop:psiinv} that they also hold for $m \pm 1$. \emph{iii.} and \emph{iv.} are obtained by applying $\varphi$ to \emph{i.} and \emph{ii.} respectively. Finally \emph{v.} is obtained by direct calculation from the definitions of $\bpsi+{1,1}(z)$ and $\bpsi-{1,-1}(v)$, i.e.
\bea &&\delta \left (\frac{C^{-1/2}q^2w}{z}\right ) \delta \left (\frac{C^{1/2}q^{-2}u}{v}\right )\left[\bpsi+{1,1}(z), \bpsi-{1,-1}(u)\right ] \nn\\
&=& \left [\left [\x+0(C^{-1/2} w), \x+1(z) \right ]_{G_{10}^-(C^{-1/2}w/z)}, \left [\x-1(u), \x-0(C^{-1/2}v) \right ]_{G_{10}^+(C^{1/2}u/v)}\right ] \nn\\
&=& [2]_q \delta \left (\frac{C^{-1/2}q^2 v}{u} \right ) \left \{\delta \left (\frac{z}{Cu} \right ) \left [\x+0(C^{-1/2} w), \x-0(C^{-1/2}v)\kk+1(C^{-1/2}z) \right ]_{G_{10}^-(C^{-1/2}w/z)} 
\right .\nn\\ 
&& \left .
\qquad \qquad \qquad  \qquad - \delta \left (\frac{Cw}{v} \right ) \left [\kk-0(C^{-1}v)\x-1(u), \x+1(z) \right ]_{G_{10}^-(C^{-1/2}w/z)} \right \}\nn\\
&=&\frac{[2]_q}{q-q^{-1}} \delta\left (\frac{C^{1/2} q^{-2} u}{v} \right )\delta \left (\frac{C^{-1/2} q^2 w}{z} \right )\left \{\delta \left (\frac{z}{Cu} \right ) \kk+0(C^{-1}w)\kk+1(C^{-1/2}z) \right . \nn\\
&& \left . \qquad \qquad \qquad \qquad \qquad \qquad \qquad \qquad \qquad -  \delta \left (\frac{Cz}{u}\right ) \kk-0(C^{-1} v) \kk-1(C^{-1/2} u) \right \} \,. \nn
\eea
Compare with (\ref{eq:defwp}) to conclude the proof. 
\end{proof}

\begin{defprop}
	For every $m \in \N^\times$ there exist $\bpsi +{1,m}(z), \bz{+}{m}(z) \in \uqsltc[[z, z^{-1}]]$, such that
	\be \label{eq:bz+1} \bz{+}{1}(v) = 0 \ee
	and, for every $m, n\in\N^\times$,
	\bea\label{eq:Ym} \left [ Y^m \left ( \kk -1(z)^{-1} \x -1(C^{1/2}z) \right ) , \x +1(v) \right ]_{G_{01}^-\left(z/C^{1/2}v\right)} &=& -\delta\left (\frac{z}{C^{1/2}v}\right )\bz{+}{m}(v) \nn\\
	&&+(q-q^{-1})\sum_{k=1}^{m-2}\delta \left ( \frac{q^{2k}z}{C^{1/2}v} \right ) \bpsi +{1,k}(v)\bz{+}{m-k}(v) \nn\\
	&&- \delta \left ( \frac{q^{2m}z}{C^{1/2}v} \right ) \bpsi +{1,m}(v)\, , \eea
	\be\label{eq:Ypsim} Y\left (\bpsi +{1,m}(v)\right ) = \bpsi +{1,m}(v) \,,\ee
	\be\label{eq:Yqsim} Y\left (\bz{+}{m}(v)\right ) =\bz{+}{m}(vq^2)\,, \ee
	\bea\label{eq:psipsimm}{}_{G^-_{01}(q^{-2m}v/w)G^-_{11}(q^{2(1-m)}v/w)} \left [ \bpsi +{1,1}(w), \bpsi +{1,m}(v) \right ]_{G^-_{01}(w/vq^2)G^-_{11}(w/v)} &=&
	[2]_q\delta\left(\frac{w}{vq^2}\right) \bpsi+{1,m+1}(q^2v) \nn\\&& -[2]_q\delta\left(\frac{q^{2m}w}{v}\right) \bpsi+{1,m+1}(v)
	\,,\eea
	\be\label{eq:commbz+1} [\bpsi{+}{1,n}(w), \bz{+}{m}(v)]=0\,.\ee
\end{defprop}

\begin{proof}
	It suffices to prove the proposition with $n=1$ since the general case follows by an easy recursion on $n$ once we have (\ref{eq:psipsimm}). The proof for $n=1$ is by recursion on $m$. For $m=1$, (\ref{eq:bz+1}) and (\ref{eq:Ym}) are definition-proposition \ref{defpsi+11}, whereas (\ref{eq:Ypsim}) is proposition \ref{prop:psiinv}. (\ref{eq:Yqsim}) and (\ref{eq:commbz+1}) -- with $n=1$ -- automatically follow from (\ref{eq:bz+1}). Making use of propsition \ref{prop:psixm}, it is straightforward to prove that, for every $m\in\N^\times$,
	\bea
	\label{eq:recursion}
	&&[2]_q \delta \left( \frac{z}{uq^2}\right) Y^{-1} \left ( \left[ Y^{m+1} \left(\kk-1(C^{-1/2}v)^{-1} \x-1(v) \right ), \x+1(uq^2) \right]_{G^-_{01}(C^{-1}q^{-2}v/u)} \right ) \nn\\
	&-& [2]_q \delta \left( \frac{Cz}{v}\right) \left[ Y^{m+1} \left(\kk-1(C^{1/2}q^{-2}z)^{-1} \x-1(C q^{-2}z) \right ), \x+1(u) \right]_{G_{01}^-(z/uq^2)}  \\
	&=& {}_{G^-_{10}(v/Cz) G^-_{11}(vq^2/Cz)}\left[\bpsi+{1,1}(z) , \left[ Y^{m} \left(\kk-1(C^{-1/2}v)^{-1} \x-1(v) \right ), \x+1(u) \right]_{G^-_{01}(C^{-1}v/u)}  \right]_{G^-_{10}(z/uq^2) G^-_{11}(z/u)}\,.\nn
	\eea
	If $m=1$, (\ref{eq:psipsimm}) is an easy consequence of the above equation. Now assume that the proposition holds up to some $m\in\N^\times$. Then (\ref{eq:recursion}) reads, for that $m$,
	\bea
	\label{eq:recursion2}
	&&[2]_q \delta \left( \frac{z}{uq^2}\right) Y^{-1} \left ( \left[ Y^{m+1} \left(\kk-1(C^{-1/2}v)^{-1} \x-1(v) \right ), \x+1(uq^2) \right]_{G^-_{01}(C^{-1}q^{-2}v/u)} \right ) \nn\\
	&-& [2]_q \delta \left( \frac{Cz}{v}\right) \left[ Y^{m+1} \left(\kk-1(C^{1/2}q^{-2}z)^{-1} \x-1(C q^{-2}z) \right ), \x+1(u) \right]_{G_{01}^-(z/uq^2)}  \nn\\
	&=& -\delta\left (\frac{v}{Cu}\right ){}_{G^-_{10}(v/Cz) G^-_{11}(vq^2/Cz)}\left[\bpsi+{1,1}(z) , \bz{+}{m}(u)\right]_{G^-_{10}(z/uq^2) G^-_{11}(z/u)} \nn\\
	&&+(q-q^{-1})\sum_{k=1}^{m-2}\delta \left ( \frac{q^{2k}v}{Cu} \right ) {}_{G^-_{10}(v/Cz) G^-_{11}(vq^2/Cz)}\left[\bpsi+{1,1}(z) , \bpsi +{1,k}(u)\right]_{G^-_{10}(z/uq^2) G^-_{11}(z/u)} \bz{+}{m-k}(u) \nn\\
	&&- \delta \left ( \frac{q^{2m}v}{Cu} \right ) {}_{G^-_{10}(v/Cz) G^-_{11}(vq^2/Cz)}\left[\bpsi+{1,1}(z) , \bpsi +{1,m}(u)\right]_{G^-_{10}(z/uq^2) G^-_{11}(z/u)}\,\nn\\
	&=& -[2]_q(q-q^{-1}) \delta\left (\frac{v}{Cu}\right )\left \{\delta\left(\frac{{v}}{Cz}\right)-\delta \left (\frac{vq^2}{Cz}\right) \right \}\bpsi+{1,1}(z)\bz{+}{m}(u) \nn\\ &&+[2]_q(q-q^{-1})\sum_{k=1}^{m-2}\delta \left ( \frac{q^{2k}v}{Cu} \right ) \left\{\delta\left (\frac{z}{uq^2}\right )\bpsi{+}{1,k+1}(uq^2) - \delta\left (\frac{zq^{2k}}{u}\right )\bpsi{+}{k+1}(u)\right\}  \bz{+}{m-k}(u) \nn\\
	&&- [2]_q\delta \left ( \frac{q^{2m}v}{Cu} \right )\left\{\delta\left(\frac{z}{uq^2}\right)\bpsi{+}{1,m+1}(uq^2) - \delta \left(\frac{zq^{2m}}{u}\right ) \bpsi{+}{1,m+1}(u)\right\}\,.\nn\\
	\nn
	\eea
	It immediately follows that (\ref{eq:Ym}) holds at rank $m+1$, for some $\bz{+}{m+1}(z)\in\uqsltc[[z,z^{-1}]]$ satisfying (\ref{eq:Yqsim}). Considering (\ref{eq:recursion}) at rank $m+1$, and substituting the above results, we get
	\bea
	\label{eq:recursion3}
	&&[2]_q \delta \left( \frac{z}{uq^2}\right) Y^{-1} \left ( \left[ Y^{m+2} \left(\kk-1(C^{-1/2}v)^{-1} \x-1(v) \right ), \x+1(uq^2) \right]_{G^-_{01}(C^{-1}q^{-2}v/u)} \right ) \nn\\
	&-& [2]_q \delta \left( \frac{Cz}{v}\right) \left[ Y^{m+2} \left(\kk-1(C^{1/2}q^{-2}z)^{-1} \x-1(C q^{-2}z) \right ), \x+1(u) \right]_{G_{01}^-(z/uq^2)}  \nn\\
	&=& -\delta\left (\frac{v}{Cu}\right ){}_{G^-_{10}(v/Cz) G^-_{11}(vq^2/Cz)}\left[\bpsi+{1,1}(z) , \bz{+}{m+1}(u)\right]_{G^-_{10}(z/uq^2) G^-_{11}(z/u)} \nn\\
	&&+[2]_q(q-q^{-1})\sum_{k=1}^{m-1}\delta \left ( \frac{q^{2k}v}{Cu} \right ) \left\{\delta\left (\frac{z}{uq^2}\right )\bpsi{+}{1,k+1}(uq^2) - \delta\left (\frac{zq^{2k}}{u}\right )\bpsi{+}{k+1}(u)\right\} \bz{+}{m+1-k}(u) \nn\\
	&&- \delta \left ( \frac{q^{2(m+1)}v}{Cu} \right ) {}_{G^-_{10}(v/Cz) G^-_{11}(vq^2/Cz)}\left[\bpsi+{1,1}(z) , \bpsi +{1,m+1}(u)\right]_{G^-_{10}(z/uq^2) G^-_{11}(z/u)}\,.\nn
	\eea
	It readily follows that, on one hand, there exists some $\bpsi+{1,m+2}(v) \in \widehat{\dqaslt}[[v,v^{-1}]]$ such that (\ref{eq:psipsimm}) holds for $m+1$ and that, on the other hand,
	\be (uq^2-z)(u-z)\left[\bpsi+{1,1}(z) , \bz{+}{m+1}(u)\right] = 0\,.\nn\ee
	Since $Y(\bz{+}{m+1}(u))=\bz{+}{m+1}(uq^2)$, we have that
	\be (uq^{2(p+1)}-z)(uq^{2p}-z)\left[\bpsi+{1,1}(z) , \bz{+}{m+1}(u)\right] = 0\nn\ee
	for every $p\in\Z$ and, as a consequence, (\ref{eq:commbz+1}) holds for $m+1$. Finally, (\ref{eq:Ypsim}) for $m+1$ follows from the corresponding case of (\ref{eq:psipsimm}), which concludes the proof.
\end{proof}
\begin{rem}
Note that since $[\bpsi+{1,n}(z),\bz{+}{m}(v)]=0$ for every $m,n\in\N^\times$, we have that
\be \bpsi{+}{1,n,k} \bz{+}{m,l} = \bz{+}{m,l} \bpsi{+}{1,n,k} \in \Omega_{l-k} \cap \Omega_{k-l} \,,\ee
guaranteeing the convergence in $\uqsltc$ of each of the terms of the series $\bpsi +{1,k}(z)\bz{+}{m-k}(z)$ on the the r.h.s of eq. (\ref{eq:Ym}).
\end{rem}
\begin{defn}
For every $m\in\N^\times$, let 
\be \bz-{-m}(z) = \varphi(\bz+m(1/z))\qquad \mbox{and} \qquad  \bpsi -{1,-m}(z) = \varphi(\bpsi +{1,m}(1/z))\,.\ee
\end{defn}
\noi Then,
\begin{cor}
 We have 
\be \bz{-}{-1}(v) = 0 \ee
	and, for every $m, n\in\N^\times$,
	\bea \label{eq:x-X+m}\left [\x -1(v), Y^m \left ( \x +1(C^{1/2}z) \kk +1(z)^{-1}  \right )   \right ]_{G_{01}^+\left(C^{1/2}v/z\right)} &=& -\delta\left (\frac{z}{C^{1/2}v}\right )\bz{-}{-m}(v) \nn\\
	&&-(q-q^{-1})\sum_{k=1}^{m-2}\delta \left ( \frac{q^{2k}z}{C^{1/2}v} \right ) \bz{-}{-(m-k)}(v) \bpsi -{1,-k}(v) \nn\\
	&&- \delta \left ( \frac{q^{2m}z}{C^{1/2}v} \right ) \bpsi -{1,-m}(v)\, , \eea
	\be Y\left (\bpsi -{1,-m}(v)\right ) = \bpsi -{1,-m}(v) \,,\ee
	\be Y\left (\bz{-}{-m}(v)\right ) =\bz{-}{-m}(vq^2)\,, \ee
	\bea{}_{G^+_{01}(q^{2m}w/v)G^+_{11}(q^{2(m-1)}w/v)} \left [ \bpsi -{1,-m}(v), \bpsi -{1,-1}(w)  \right ]_{G^+_{01}(vq^2/w)G^+_{11}(v/w)} 
	&=&
	[2]_q\delta\left(\frac{w}{vq^2}\right) \bpsi-{1,-(m+1)}(q^2v) \nn\\&& -[2]_q\delta\left(\frac{q^{2m}w}{v}\right) \bpsi-{1,-(m+1)}(v)
	\,,\label{eq:psi-psi-}\eea
	\be [\bpsi{-}{1,-n}(w), \bz{-}{-m}(v)]=0\,.\ee
\end{cor}
\begin{proof}
It suffice to apply $\varphi$ to the results of the previous proposition.
\end{proof}
\begin{prop}
For every $i\in\dot I$ and for every $m\in\N^\times$, we have
\begin{enumerate}
\item[i.] $\kk-i(v) \bpsi\pm{1,\pm m}(z) = G_{i,0}^\mp(C^{\mp 1/2} q^{2m}v/z) G_{i,1}^\mp (C^{\mp 1/2}v/z) \bpsi\pm{1,\pm m} (z)\kk-i(v)$;
\item[ii.] $\bpsi\pm{1,\pm m}(z) \kk+i(v)  = G_{i,0}^\mp(C^{\mp 1/2} q^{-2m}z/v) G_{i,1}^\mp (C^{\mp 1/2}z/v) \kk+i(v) \bpsi\pm{1,\pm m} (z)$;
\end{enumerate}
\end{prop}
\begin{proof}
Clearly \emph{ii.} follows by applying $\varphi$ to \emph{i.}. We prove \emph{i.} by induction on $m\in\N^\times$. The case $m=1$ is corollary \ref{cor:kpsi1}\emph{i.} Now, assuming that \emph{i.} holds for some $m\in\N^\times$, we can make use of (\ref{eq:psipsimm}) and (\ref{eq:psi-psi-}) to show that
\bea \kk-i(v) \bpsi\pm{1,\pm (m+1)}(z)  &=& G_{i,0}^\mp(C^{\mp 1/2} q^{2(m+1)}v/z) G_{i,1}^\mp (C^{\mp 1/2}q^{2m}v/z )\nn\\&&\qquad \qquad \quad \times G_{i,0}^\mp(C^{\mp 1/2} q^{2m} v/z ) G_{i,1}^\mp(C^{\mp 1/2} z/v) \bpsi\pm{1,\pm m} (z) \kk-i(v) \nn\\
&=& G_{i,0}^\mp(C^{\mp 1/2} q^{2(m+1)}v/z)G_{i,1}^\mp(C^{\mp 1/2} z/v) \bpsi\pm{1,\pm m} (z) \kk-i(v) \nn
\eea
whch completes the recursion.
\end{proof}
The above proposition has the obvious
\begin{cor}
\label{cor:wppsi}
For every $m\in\N^\times$, we have
\bea \bwp-(v)\bpsi\pm{1,\pm m}(z) &=& G_{00}^\mp(C^{\mp1/2}q^{2m}v/z)G_{01}^\mp(C^{\mp1/2}v/z)\nn\\
&&\qquad\qquad\qquad G_{01}^\mp(C^{\mp1/2}q^{2(m+1)}v/z)G_{11}^\mp(C^{\mp1/2}q^{2}v/z)\bpsi\pm{1,\pm m}(z) \bwp-(v)\,;\eea
\bea \bpsi\pm{1,\pm m}(z) \bwp+(v)&=&G_{00}^\mp(C^{\mp1/2}q^{-2m}z/v)G_{01}^\mp(C^{\mp1/2}z/v)\nn\\
&&\qquad\qquad\qquad G_{01}^\mp(C^{\mp1/2}q^{-2(m+1)}z/v)G_{11}^\mp(C^{\mp1/2}q^{-2}z/v) \bwp+(v) \bpsi\pm{1,\pm m}(z)\,. \eea
\end{cor}

\begin{prop}
For every $m,n\in\N^\times$, we have
\bea \left [\bpsi+{1,m}(v), \bpsi-{1,-n}(w)\right ] &=& \left[2\right ]_q (q-q^{-1}) \left \{\delta\left (\frac{C q^{2(1-m)}v}{w} \right )\bwp-(C^{-1/2}q^{-2m}v) \bpsi-{1,-(n-1)}(wq^{-2}) \bpsi+{1,m-1}(v)  \right.\nn\\
&& \left . - \delta \left( \frac{q^{2(n-1)}v}{Cw} \right ) \bpsi-{1,-(n-1)}(w) \bpsi+{1,m-1}(vq^{-2}) \bwp+(C^{1/2}q^{-2} v)\right \}\,, \nn \eea
where we assume that
\be \bpsi\pm{1,0}(z) = \frac{1}{q-q^{-1}}\,.\ee
\end{prop}
\begin{proof}
The case $m=n=1$ follows immediately by proposition \ref{prop:psixm}.\emph{v}. Now, applying $a\mapsto [a, \bpsi-{1,-n}(w)]$ and $a\mapsto [\bpsi+{1,n}(w), a]$ to (\ref{eq:psipsimm}) and (\ref{eq:psi-psi-}) respectively and making use of corollary \ref{cor:wppsi}, one easily completes the recursion.
\end{proof}

\subsection{Exchange relations}
\begin{prop}
\label{prop:X+1X+1}
	For every $m\in\N$, there exists some $\xi_m(z)\in\uqsltc[[z, z^{-1}]]$ such that, for every $n\in\Z$,
	\be \label{eq:exchrel}[\X{-}{1,m+n+1}(w), \X{-}{1,n}(z)]_{G_{01}^-(w/z)} = -[\X{-}{1,n+1}(w),\X{-}{1,m+n}(z)]_{G_{01}^-(w/z)} = \delta\left (\frac{wq^2}{z}\right )Y^n\left(\xi_m (z)\right)\,.\ee
\end{prop}
\begin{proof}
	Assume first that $n=0$. The case $m=0$ then follows immediately from the definition of $\X{-}{1,1}(w)$ and relations (\ref{eq:kmxpm}) and (\ref{eq:relx+x-}), leading to $\xi_0(z)=0$, as it should. Taking the commutator of the case $m=0$ with $\bpsi+{1,1}(v)$, we get
\bea 0&=& [[\X{-}{1,1}(w), \X{-}{1,0}(z)]_{G_{01}^-(w/z)},\bpsi+{1,1}(v)]\nn\\
&=&[[\X{-}{1,1}(w), \bpsi+{1,1}(v)],\X{-}{1,0}(z)]_{G_{01}^-(w/z)}+ [\X{-}{1,1}(w), [\X{-}{1,0}(z),\bpsi+{1,1}(v)]]_{G_{01}^-(w/z)}\nn\\
&=&[2]_q \bwp -(v)\left\{\delta\left(\frac{C^{1/2}q^2v}{w}\right )[\X{-}{1,2}(wq^{-2}), \X{-}{1,0}(z)]_{G_{01}^-(wq^{-2}/z)}\right .\nn\\
&& \qquad\qquad\qquad \left .  + \delta\left(\frac{C^{1/2}q^2v}{z}\right ){}_{G_{01}^-(zq^{-2}/w)G_{11}^-(z/w)}[\X{-}{1,1}(w), \X{-}{1,1}(zq^{-2})]_{G_{01}^-(w/z)}\right\} \,.\nn\eea 	
The latter implies that 
\be\label{eq:I} [\X{-}{1,2}(wq^{-2}), \X{-}{1,0}(z)]_{G_{01}^-(wq^{-2}/z)} = \delta\left(\frac{w}{z}\right)\xi_1(z)\,,\ee
\be\label{eq:II} {}_{G_{01}^-(zq^{-2}/w)G_{11}^-(z/w)}[\X{-}{1,1}(w), \X{-}{1,1}(zq^{-2})]_{G_{01}^-(w/z)} = -\delta\left(\frac{w}{z}\right)\xi_1(z)\,,\ee
for some $\xi_1(z) \in \uqsltc[[z,z^{-1}]]$. Multiplying (\ref{eq:II}) by $(zq^{-2}-w)$ and subsequently factoring $(z-q^{-2}w)$, we get that
\be\label{eq:prem1} {}_{G_{01}^-(zq^{-2}/w)}[\X{-}{1,1}(w), \X{-}{1,1}(zq^{-2})]= \delta\left(\frac{w}{z}\right)\xi_1(z)+ \delta\left( \frac{w}{zq^2} \right ) \eta_0(z)\,,\ee
for some $\eta_0(z) \in \uqsltc[[z,z^{-1}]]$. Multiplying the above equation by $q^{-2}(z-w)$, we get 
\be\label{eq:III} (zq^{-4}-w) \X{-}{1,1}(w) \X{-}{1,1}(zq^{-2}) - q^{-2}(z-w) \X{-}{1,1}(zq^{-2})\X{-}{1,1}(w) = z(1-q^2) \delta\left( \frac{w}{zq^2}\right )\eta_0(z)\,. \ee
But, on the other hand,
\bea &&(zq^{-4}-w) \X{-}{1,1}(w) \X{-}{1,1}(zq^{-2}) - q^{-2}(z-w) \X{-}{1,1}(zq^{-2})\X{-}{1,1}(w) \nn\\
&=& Y\left((zq^{-4}-w) \xm 1(w) \xm 1(zq^{-2})- q^{-2}(z-w) \xm{1}(zq^{-2})\xm{1}(w)  \right )=0\nn\eea
by relation (\ref{eq:xpmxpm}). Substituting back into (\ref{eq:III}) proves that $\eta_0(z)=0$ and that (\ref{eq:prem1}) eventually reads
\be\label{eq:2ndhalfm1} {}_{G_{01}^-(zq^{-2}/w)}[\X{-}{1,1}(w), \X{-}{1,1}(zq^{-2})]= \delta\left(\frac{w}{z}\right)\xi_1(z)\, .\ee
Combining (\ref{eq:I}) and (\ref{eq:2ndhalfm1}), we get the case $m=1$. Now assume that the result holds for all nonnegative integer less than $m\in\N$. Taking the commutator of (\ref{eq:exchrel}) with $\bpsi{+}{1,1}(v)$ yields
	\bea
	&&[2]_q\bwp -(v) \left \{\delta\left(\frac{C^{1/2}q^2v}{w}\right )[\X{-}{1,m+2}(wq^{-2}), \X{-}{1,0}(z)]_{G_{01}^-(wq^{-2}/z)} \right .\nn\\
	&&\qquad\qquad\qquad\left .+\delta\left(\frac{C^{1/2} q^2v}{z}\right ){}_{G_{01}^-(zq^{-2}/w) G_{11}^-(z/w)}[\X{-}{1,m+1}(w), \X{-}{1,1}(zq^{-2})]_{G_{01}^-(w/z)} \right \} \nn\\
	&=&- [2]_q\bwp -(v) \left \{\delta\left(\frac{C^{1/2}q^2v}{w}\right )[\X{-}{1,2}(wq^{-2}), \X{-}{1,m}(z)]_{G_{01}^-(wq^{-2}/z)} \right .\nn\\
	&&\qquad\qquad\qquad\left .+\delta\left(\frac{C^{1/2} q^2v}{z}\right ){}_{G_{01}^-(zq^{-2}/w) G_{11}^-(z/w)}[\X{-}{1,1}(w), \X{-}{1,m+1}(zq^{-2})]_{G_{01}^-(w/z)} \right \} \nn\\
	&=& \delta\left(\frac{wq^2}{z}\right) [\xi_m(z), \bpsi{+}{1,1}(v)]\nn\eea
The latter implies that
\be\label{eq:i} [\X{-}{1,m+2}(wq^{-2}), \X{-}{1,0}(z)]_{G_{01}^-(wq^{-2}/z)} = \delta\left ( \frac{w}{z} \right ) \xi_{m+1}(z) +\delta\left ( \frac{wq^2}{z} \right ) \eta_1(z)\,,  \ee
\be\label{eq:ii} {}_{G_{01}^-(zq^{-2}/w) G_{11}^-(z/w)}[\X{-}{1,m+1}(w), \X{-}{1,1}(zq^{-2})]_{G_{01}^-(w/z)} = -\delta\left ( \frac{w}{z} \right ) \xi_{m+1}(z) + \delta\left ( \frac{wq^2}{z} \right ) \eta_1(z)\,,\ee
\be\label{eq:iii} [\X{-}{1,2}(wq^{-2}), \X{-}{1,m}(z)]_{G_{01}^-(wq^{-2}/z)}  = \delta\left ( \frac{w}{z} \right ) \eta_{3}(z) - \delta\left ( \frac{wq^2}{z} \right ) \eta_1(z)\,,\ee
\be\label{eq:iv} {}_{G_{01}^-(zq^{-2}/w) G_{11}^-(z/w)}[\X{-}{1,1}(w), \X{-}{1,m+1}(zq^{-2})]_{G_{01}^-(w/z)} =-\delta\left ( \frac{w}{z} \right ) \eta_{3}(z) - \delta\left ( \frac{wq^2}{z} \right ) \eta_2(z) \,, \ee
for some $\xi_{m+1}(z), \eta_1(z), \eta_2(z), \eta_3(z) \in \uqsltc[[z,z^{-1}]]$. Multiplying (\ref{eq:iv}) by $(z-wq^2)$ and subsequently factoring $(zq^2-w)$, we get that
\be\label{eq:v} [\X{-}{1,m+1}(z), \X{-}{1,1}(w)]_{G_{01}^-(z/w)} =-\delta\left ( \frac{w}{zq^2} \right ) \eta_{3}(w) + \delta\left ( \frac{w}{zq^4} \right ) \eta_4(z)\,, \ee
for some $\eta_4(z) \in \uqsltc[[z,z^{-1}]]$. But, by the recursion hypothesis,
\be [\X{-}{1,m+1}(z), \X{-}{1,1}(w)]_{G_{01}^-(z/w)} = Y\left ([\X{-}{1,m}(z), \X{-}{1,0}(w)]_{G_{01}^-(z/w)}\right ) = \delta\left ( \frac{w}{zq^2} \right ) Y\left (\xi_{m-1}(w)\right ) \,.\nn\ee
Comparing with (\ref{eq:v}), it follows that
\be \eta_3(w) = -Y\left (\xi_{m-1}(w)\right )\qquad \mbox{and} \qquad \eta_4(z)=0\,.\nn\ee
By the recursion hypothesis, we also have
\bea [\X{-}{1,2}(wq^{-2}), \X{-}{1,m}(z)]_{G_{01}^-(wq^{-2}/z)}  &=& Y\left ( [\X{-}{1,1}(wq^{-2}), \X{-}{1,m-1}(z)]_{G_{01}^-(wq^{-2}/z)}\right ) = -\delta\left ( \frac{w}{z} \right ) Y\left(\xi_{m-1}(z)\right )   \nn\\
&=&\delta\left ( \frac{w}{z} \right ) \eta_{3}(z) \nn\eea
Comparing the above result with (\ref{eq:iii}), we conclude that $\eta_1(z)=0$. As a consequence, (\ref{eq:i}) now reads
\be\label{eq:1sthalfrec}[\X{-}{1,m+2}(w), \X{-}{1,0}(z)]_{G_{01}^-(w/z)} = \delta\left ( \frac{wq^2}{z} \right ) \xi_{m+1}(z)\,. \ee
On the other hand, multiplying (\ref{eq:ii}) by $(z-wq^2)$ and subsequently factoring $(zq^2-w)$, we get that
\be\label{eq:pre2ndhalfrec} {}_{G_{01}^-(zq^{-2}/w)}[\X{-}{1,m+1}(w), \X{-}{1,1}(zq^{-2})] = \delta\left ( \frac{w}{z} \right ) \xi_{m+1}(z) + \delta\left ( \frac{w}{zq^2} \right ) \eta_5(z)\,,\ee
for some $\eta_5(z) \in \uqsltc[[z,z^{-1}]]$. Multiplying the above equation by $(z-w)$ yields
\be\label{eq:eta5zero} Y\left ( (zq^{-2}-wq^2) \X{-}{1,m}(w) \X{-}{1,0}(zq^{-2}) -(z-w)\X{-}{1,0}(zq^{-2}) \X{-}{1,m}(w) \right )= z(1-q^2)\delta \left ( \frac{zq^2}{w}\right ) \eta_5(z)\, .\ee
But the recursion hypothesis
\be [\X{-}{1,m}(w), \X{-}{1,0}(zq^{-2})]_{G_{01}^-(wq^2/z)}= \delta\left(\frac{wq^4}{z} \right ) \xi_{m-1}(z)\ee
implies, upon multiplication by $(zq^{-2}-wq^2)$, that
\be (zq^{-2}-wq^2) \X{-}{1,m}(w) \X{-}{1,0}(zq^{-2}) -(z-w)\X{-}{1,0}(zq^{-2}) \X{-}{1,m}(w)=0\,.\ee
Substituting back into (\ref{eq:eta5zero}) proves that $\eta_5(z)=0$ and that (\ref{eq:pre2ndhalfrec}) eventually reads
\be\label{eq:2ndhalfrec} {}_{G_{01}^-(w/z)}[\X{-}{1,m+1}(z), \X{-}{1,1}(w)]= \delta\left ( \frac{wq^2}{z} \right ) \xi_{m+1}(z) \,.\ee
Combining (\ref{eq:1sthalfrec}) and (\ref{eq:2ndhalfrec}) completes the recursion and the result holds for any $m\in\N$, assuming $n=0$. The cases $n\in\Z^\times$ are then obtained by applying $Y^n$ to the case $n=0$. 
\end{proof}
\begin{cor}
For every $m\in\N$ and every $n\in\Z$, we have
\be [\Xp{1,m+n+1}(z), \Xp{1,n}(w)]_{G_{01}^+(z/w)} = -[\Xp{1,n+1}(z), \Xp{1,m+n}(w)]_{G_{01}^+(z/w)}=\delta\left(\frac{wq^2}{z}\right)\varphi\circ Y^{-m-n-1}\left(\xi_{m}(1/z)\right) \,.\ee
\end{cor}
\begin{proof}
It suffices to apply $\varphi\circ Y^{-m-n-1}$ to (\ref{eq:exchrel}).
\end{proof}
We now return to the proof of theorem \ref{thm:main} and to the map $\Psi : \qaff(\dot{\mathfrak a}_1) \to \widehat{\qdaff'(\mathfrak a_1)}$. 
\begin{cor}
We have
\begin{enumerate}
\item[i.] $\Upsilon^\pm(w) = 0$;
\item[ii.] and for every $i\neq j$,
$$\sum_{\sigma \in S_{3}} \sum_{k=0}^{3} (-1)^k {{3}\choose{k}}_q \Psi(\x \pm i(z_{\sigma(1)}))  \cdots \Psi(\x \pm i(z_{\sigma(k)})) \Psi( \x \pm j(z)) \Psi(  \x \pm i(z_{\sigma(k+1)})) \cdots \Psi(\x\pm i(z_{\sigma(3)})) =0\,.$$
\end{enumerate}
\end{cor}
\begin{proof}
The proof of proposition \ref{prop:X+1X+1} makes it clear that the relations (\ref{eq:relx+x-}) with $i\neq j$ there, both follow from the relations
\be \left[\X+{1,0}(v), \X+{1,-1}(w)\right]_{G_{11}^-(v/w)} =0 \label{eq:X+10X+1-1}\ee
and
\be\left[\X-{1,1}(v), \X-{1,0}(w)\right]_{G_{11}^+(v/w)} =0 \label{eq:X-11X-10}\ee
in the completion $\widehat{\qaff(\dot{\mathfrak a}_1)}$. A tedious but straightforward calculation shows that the quantum Serre relations (\ref{eq:qaffserre}) similarly follow from
\be \left[\X+{1,-1}(v), \X+{1,-2}(w)\right]_{G_{11}^-(v/w)} =0\ee
and
\be\left[\X-{1,2}(v), \X-{1,1}(w)\right]_{G_{11}^+(v/w)} =0\,,\ee
which in turn are a consequence of ((\ref{eq:X+10X+1-1}) -- (\ref{eq:X-11X-10})) -- just apply $Y$ there. We can therefore extend $\Psi:\qaff(\dot{\mathfrak a}_1)\to \widehat{\qdaff'(\mathfrak a_1)}$ by continuity~\footnote{$\Psi$ is obviously $\Z_{(2)}$-graded, hence continuous.} into $\widehat\Psi:\widehat{\qaff(\dot{\mathfrak a}_1)}\to \widehat{\qdaff'(\mathfrak a_1)}$ and it suffices to check point \emph{i.} Since by construction $\qaff(\dot{\mathfrak a}_1)$ is dense in $\widehat{\qaff(\dot{\mathfrak a}_1)}$, there exists a sequence $(u_n(v,w))_{n\in\N}\in \qaff(\dot{\mathfrak a}_1)[[v, v^{-1},w,w^{-1}]]^\N$ such that
\be \lim_{n\to +\infty} u_n(v,w) =0\,,\ee
whereas, on the other hand,
\be \lim_{n\to +\infty} \widehat\Psi (u_n(v,w)) = \delta\left (\frac{vq^{\mp 2}}{w}\right )\Upsilon^\pm(w)\,.\ee
Take for example the partial sum of the series involved on the l.h.s. of equations ((\ref{eq:X+10X+1-1}) -- (\ref{eq:X-11X-10})) above. The result now follows by the continuity of $\widehat\Psi$.
\end{proof}
\begin{rem}
We have therefore completed the proof of that part of theorem \ref{thm:main} that claims the existence of a continuous algebra homomorphism $\widehat\Psi:\widehat{\qaff(\dot{\mathfrak a}_1)} \to \widehat{\qdaff'(\mathfrak a_1)}$. We still have to construct the inverse continuous algebra homomorphism $\widehat\Psi^{-1}: \widehat{\qdaff'(\mathfrak a_1)}\to \widehat{\qaff(\dot{\mathfrak a}_1)}$. This shall be done at the end of the present section.
\end{rem}

\subsection{Weight grading relations}
The results of the previous subsection have the following
\begin{cor}
\label{cor:psim}
For every $m \in \N^\times$ and every $n\in\Z$, we have:
\begin{enumerate}
\item[i.] $[\bz+{m+1}(u), \Xm{1,n}(z)]=0 $;
\item[ii.]\label{cor:psi+x-} $[\bpsi+{1,m+1}(u), \Xm{1,n}(z)]=-\bwp-(C^{1/2}uq^{-2(m+1)}) {}_{G_{01}^+(Cuq^{2(1-m)}/z)}[\Xm{1,n+1}(zq^{-2}), \bpsi+{1,m}(u)]_{G_{01}^{-}(z/Cuq^{2(1-m)})}\propto \delta \left ( \frac{Cu}{zq^2m}\right )$;
\item[iii.] $[\bz+{m+1}(u), \Xp{1,n}(z)]=0 $;
\item[iv.]\label{cor:psi+x+} $[\bpsi+{1,m+1}(v), \Xp{1,n}(z)]_{G_{01}^+(v/z)G_{11}^+(v/zq^{2(m+1)})}=-{}_{G_{01}^-(z/vq^{2m})}[\Xp{1,n+1}(v),\bpsi{+}{1,m}(z)]_{G_{01}^+(v/z)}\propto \delta \left ( \frac{zq^2}{v}\right )$.
\end{enumerate}
\end{cor}
\begin{proof}
	It suffices to prove the proposition for $n=0$ as the general case then follows by applying $Y^n$ for any $n\in\Z$. Assuming that $n=0$ in \emph{i.} and \emph{ii.}, it then suffices to take the commutator of (\ref{eq:exchrel}) -- for $n=1$ there -- with $\xp{1}(z)$. 
\end{proof}
\begin{rem}
	It turns out the, for every $m\in\N^\times$, $\bz{+}{m}(z)\in{\mathcal Z}(\uqsltc)[[z,z^{-1}]]$. Indeed, in the next section we actually establish that these central elements consistently vanish.
\end{rem}

\subsection{The central elements $\bz\pm{m>2}(z)$}
Before we can actually establish that these central elements vanish, we need to establish a few lemmas. In what follows, we let $\qaff^{<}(\dot{\mathfrak a}_1) = \qaff^{\leq}(\dot{\mathfrak a}_1) - \qaff^{\leq}(\dot{\mathfrak a}_1)\cap \qaff^{0}(\dot{\mathfrak a}_1)$.
\begin{lem}
\label{lem:deltapsi}
For every $p\in\N^\times$, 
\begin{enumerate}
\item[i.] $\Delta(\bpsi-{1,-p}(v)) = 1\otimes\bpsi-{1,-p}(v) \mod \qaff^<(\dot{\mathfrak a}_1) \widehat{\otimes} \qaff(\dot{\mathfrak a}_1)$;
\item[ii.] $\Delta(\X+{1,-p}(v))= \prod_{\ell=1}^{p-1} \bz+0(C^{-1/2}q^{2\ell} v)^{-1} \kk+0(C^{-1/2}v)^{-1} \widehat{\otimes} \X+{1,-p}(v) \mod \qaff^<(\dot{\mathfrak a}_1) \widehat{\otimes} \qaff(\dot{\mathfrak a}_1)$.
\end{enumerate}
\end{lem}
\begin{proof}
First one easily checks that
$$\Delta(\bpsi-{1,-1}(z)) = 1\otimes \bpsi-{1,-1}(z) + \left[2\right]_q (q-q^{-1}) \x-1(z)\widehat{\otimes} \x-0(q^{-2}z) \kk+1(z) + \bpsi-{1,-1}(z) \widehat{\otimes} \bwp+(q^{-2}z)\,,$$
which proves \emph{i.} for $p=1$. Assuming \emph{i.} holds for some $p\in\N$, the result for $p+1$ easily holds making use of (\ref{eq:psi-psi-}) and of the recursion hypothesis. 

Similarly, one easily checks that
$$\Delta(\X+{1,-1}(v)) = \X+{1,-1}(v)\otimes 1 + \kk+0(C^{-1/2}v)^{-1}\widehat{\otimes} \X+{1,-1}(v)\,,$$
which proves \emph{ii.} in the case $p=1$. Assuming the result holds for some $p\in\N$, the result for $p+1$ easily follows making use of proposition \ref{prop:psixm}.\emph{iii.} and of the recusrion hypothesis.
\end{proof}
For every $N\in\N^\times$, we let 
\bea 
S_{2N-1}^< := \left\{\sigma\in S_{2N-1} \right. &:& \sigma(1) = 1 \nn\\
&& \forall p\in \rran{N-1}\quad \sigma(2p) < \sigma(2p+1)\nn\\
&&  \left. \sigma(2N-4)< \sigma(2N-1) \right \}
\eea 
Define $\varpi:\Z\to \dot{I}=\{0,1\}$ by setting, for every $n\in\Z$,
\be 
\varpi(n) := \begin{cases}
0 & \mbox{if $n$ is even;}\\
1 & \mbox{if $n$ is odd.}
\end{cases}
\ee
\begin{lem}
For every $r\in\N$ and every $i_1, \dots, i_{2r-1} \in \dot{I}$, there exists $(\beta_{r,\sigma})_{\sigma \in S_{2r-1}^<} \in \F^{S_{2r-1}^<}$ such that
\be
\label{eq:pairingX+}
\left \langle \x+{i_1}(z_1) \dots \x+{i_{2r-1}}(z_{2r-1}), \X+{1,-r}(v) \right \rangle = -\frac{\left[2\right ]_q^{r-1}}{q-q^{-1}} \sum_{\sigma \in S_{2r-1}^<} \beta_{r,\sigma} \prod_{n=1}^{2r-1} \delta_{i_{\sigma(n)}, \pi(n)} \delta \left (\frac{z_{\sigma(n)} q^{\nu_r(n)}}{v} \right )  
\ee
where we have defined $\pi:\N\to\dot{I}$ and $\nu_r:\N\to\Z$ by setting, for every $n \in \N$,
\be 
\pi(n) = \begin{cases}
0 & \mbox{if $n=1$;}\\
\varpi(m) &\mbox{if $n>1$}
\end{cases}
\ee
and 
\be\nu_r(n)= \begin{cases}
2(1-r) & \mbox{if $n=1$;}\\
2(1-r)+n-3\varpi(n) & \mbox{if $n>1$.}
\end{cases}
\ee
\end{lem}
\begin{proof}
The case $r=0$ holds by definition of the pairing. Assume that (\ref{eq:pairingX+}) holds for some $r\in\N$. Then, making use of the previous lemma, one easily shows that, for every $i_1, \dots, i_{2r+1} \in \dot{I}$ 
\bea && \left\langle \x+{i_1}(z_1) \dots \x+{i_{2r+1}}(z_{2r+1}), \left [\bpsi-{1,-1}(z), \X+{1,-r}(v) \right ]\right \rangle\nn\\
 &=& \frac{\left[2\right]_q}{q-q^{-1}} \sum_{\bA\in\Prt{(2,2r-1)}{\rran{2r+1}}} \prod_{m\in\rran{2}} \delta_{i_{A^{(1)}_m}, 1-\varpi(m)} \delta\left(\frac{z_{A^{(1)}_m} q^{2\varpi(m)}}{z}\right) \left\langle \x+{i_{A^{(2)}_1}}(z_{A^{(2)}_1}) \dots \x+{i_{A^{(2)}_{2r-1}}}(z_{A^{(2)}_{2r-1}}) , \X+{1,-r}(v)\right \rangle \nn\\
&& \times \left\{ R_{\bA}^<(z_\bA) - G_{i_{A^{(1)}_1}, 0}^+ (C^{-1/2}z_{A^{(1)}_1}/v) G_{i_{A^{(1)}_2}, 0}^+(C^{-1/2}z_{A^{(1)}_2}/v)R_{\bA}^>(z_\bA^{-1})  \right \} \,,\nn
\eea
where
$$R_{\bA}^<(z_\bA) = \prod_{\substack{m\in\rran{2}\\n\in\rran{2r-1}\\A^{(2)}_n <A^{(1)}_m}} G_{i_{A^{(2)}_n}, i_{A^{(1)}_m}}^-(C^{-1/2} z_{A^{(2)}_n}/z_{A^{(1)}_m})\,;$$
$$R_{\bA}^>(z_\bA^{-1}) = \prod_{\substack{m\in\rran{2}\\n\in\rran{2r-1}\\A^{(2)}_n >A^{(1)}_m}} G_{i_{A^{(2)}_n}, i_{A^{(1)}_m}}^-(C^{1/2} z_{A^{(1)}_m}/z_{A^{(2)}_n})\,.$$
Making use of proposition \ref{prop:psixm}.\emph{iii.} on the l.h.s. and of the recursion hypothesis on the r.h.s., we get
\bea
&& \left[2\right]_q \delta\left (\frac{Cz}{v}\right ) \left\langle \x+{i_1}(z_1) \dots \x+{i_{2r+1}}(z_{2r+1}), \X+{1,-(r+1)}(vq^{-2}) \right \rangle\nn\\
 &=& -\frac{\left[2\right]_q^r}{(q-q^{-1})^2}\sum_{\substack{\bA\in\Prt{(2,2r-1)}{\rran{2r+1}}\\\sigma\in S_{2r-1}^<}} \beta_{r,\sigma} \prod_{m\in\rran{2}} \delta_{i_{A^{(1)}_m}, 1-\varpi(m)} \delta\left(\frac{z_{A^{(1)}_m} q^{2\varpi(m)}}{z}\right)   \prod_{n\in\rran{2r-1}} \delta_{i_{A^{(2)}_{\sigma(n)}}, \pi(n)} \delta\left (\frac{z_{A^{(2)}_{\sigma(n)}} q^{\nu_r(n)}}{v}\right )\nn\\
 &&\times \left\{Q_{\sigma, \bA}^< (v/z) - G_{0, 0}^+ (C^{-1/2}zq^{-2}/v) G_{1, 0}^+(C^{-1/2}z/v) Q_{\sigma, \bA}^>(z/v)\right \}\,,\label{eq:pairxX}
\eea
where
$$Q_{\sigma,\bA}^< (v/z) = \prod_{\substack{m\in\rran{2}\\n\in\rran{2r-1}\\A^{(2)}_{\sigma(n)} <A^{(1)}_m}} G_{\pi(n),1- \varpi(m)}^-(C^{-1/2} vq^{\lambda_r(m,n)}/z)\,;$$
and
$$Q_{\sigma, \bA}^>(z/v)= \prod_{\substack{m\in\rran{2}\\n\in\rran{2r-1}\\A^{(2)}_{\sigma(n)} >A^{(1)}_m}} G_{\pi(n),1- \varpi(m)}^-(C^{1/2} z/vq^{\lambda_r(m,n)})\,;$$
where $\lambda_r(m,n) = 2\varpi(m)-\nu_r(n)$. In view of the $\delta(Cz/v)$ factor on the l.h.s of (\ref{eq:pairxX}), it is clear that the relevant factors in $Q_{\sigma,\bA}^< (v/z)$ and $Q_{\sigma, \bA}^>(z/v)$ are the ones contributing to a pole at $Cz=v$, i.e. the ones for which $\lambda_r(m,n)=  c_{\pi(n),1-\varpi(m)}$ or $\lambda_r(m,n)=-c_{\pi(n),1-\varpi(m)}$ respectively. We thus let
$$L_r^\pm :=\left \{(m,n)\in \rran{2}\times \rran{2r-1}: \lambda_r(m,n) =\pm c_{\pi(n),1-\varpi(m)}\right \}$$
and determine, by inspection, that, for every $r\geq 3$,
$$L_r^+ =\{(1,2r-2), (2,2r-3)\} \,, \qquad \mbox{whereas} \qquad L_r^-=\{(2,2r-4)\}\,.$$ 
Since we cannot have $A^{(2)}_{\sigma(2r-4)}>A^{(1)}_2$ while $A^{(2)}_{\sigma(2r-3)} < A^{(1)}_2$ for $\sigma\in S_{2r-1}^<$, we see that the relevant pole is necessarily a simple pole; as one might have expected, given the absence of a $\delta'(Cz/v)$ factor on the l.h.s of (\ref{eq:pairxX}). It easily follows that
$$\left\{Q_{\sigma, \bA}^< (v/z) - G_{0, 0}^+ (C^{-1/2}zq^{-2}/v) G_{1, 0}^+(C^{-1/2}z/v) Q_{\sigma, \bA}^>(z/v)\right \} =\left[2\right]_q(q-q^{-1}) \gamma_{\sigma, \bA} \delta\left(\frac{Cz}{v}\right )$$
for every $(\sigma, \bA) \in S_{2r-1}^<\times \Prt{(2,2r-1)}{\rran{2r+1}}$ such that $A^{(2)}_{\sigma(2r-2)} < A^{(1)}_1$ and either:
\begin{itemize}
\item  $A^{(2)}_{\sigma(2r-4)} > A^{(1)}_2$ (and then necessarily, $A^{(2)}_{\sigma(2r-3)}>A^{(1)}_2$); or
\item $A^{(2)}_{\sigma(2r-4)} < A^{(1)}_2$ and $A^{(2)}_{\sigma(2r-3)} < A^{(1)}_2$;
\end{itemize}
and, for each such pair $(\sigma, \bA)$, $\gamma_{\sigma, \bA} \in\F$. Note that the above conditions impose that $A^{(2)}_{\sigma(1)=1}  <A^{(1)}_1$ and hence $A^{(2)}_1= 1$. Now, for each pair $(\sigma, \bA)$ as above, define
\be \sigma':= \begin{pmatrix}
1 & 2 &\dots & 2r-1& 2r & 2r+1\\
1& A^{(2)}_{\sigma(2)} & \dots & A^{(2)}_{\sigma(2r-1)} & A^{(1)}_1 & A^{(1)}_2
\end{pmatrix}\,.\nn\ee
It is obvious that $\sigma'\in S_{2r+1}^<$. Actually, setting $(\sigma,\bA)\mapsto \sigma'$ defines a map $S_{2r-1}^<\times \Prt{(2,2r-1)}{\rran{2r+1}} \to S_{2r+1}^<$ which is easily seen to be a bijection. Observing furthermore that $\nu_r-2=\nu_{r+1}$ and setting $\beta_{r+1,\sigma'} = \beta_{r,\sigma} \gamma_{\sigma, \bA}$, we can rewrite (\ref{eq:pairxX}) as
\be \left\langle \x+{i_1}(z_1) \dots \x+{i_{2r+1}}(z_{2r+1}), \X+{1,-(r+1)}(v) \right \rangle = -\frac{\left[2\right ]_q^{r}}{q-q^{-1}} \sum_{\sigma' \in S_{2r+1}^<} \beta_{r+1,\sigma'} \prod_{n=1}^{2r+1} \delta_{i_{\sigma'(n)}, \pi(n)} \delta \left (\frac{z_{\sigma'(n)} q^{\nu_{r+1}(n)}}{v} \right ) \,,\nn\ee
which completes the recursion.
\end{proof}

\begin{prop}
For every $m\in\N^\times$, we actually have $\bz+m(v)=\bz-{-m}(v)=0$.
\end{prop}
\begin{proof}
It suffices to prove that, say $\bz-{-m}(z)=0$ for every $m\in\N^\times$ and to apply $\varphi^{-1}$ to get the result for $\bz+m(z)$. Considering the root space decomposition, it is obvious that having
$$\left\langle \x+{i_1}(z_1) \cdots \x+{i_{2m}}(z_{2m}) , \bz-{-m}(z) \right \rangle = 0\,,$$
for every $i_1,\dots, i_{2m}\in I$, is a sufficient condition. Now, making use of the previous lemma, one easily shows that
\bea && \left\langle \x+{i_1}(z_1) \cdots \x+{i_{2m}}(z_{2m}) , \left [\X+{1,-m}(v), \x-{1}(z)\right ]\right \rangle \nn\\
&=& -\frac{\left[2\right]_q}{(q-q^{-1})^2} \sum_{\substack{\bA\in\Prt{(1,2m-1)}{\rran{2m}}\\ \sigma\in S_{2m-1}^< }}\beta_{m,\sigma} \delta_{i_{A^{(1)}_1}, 1} \delta\left(\frac{z_{A^{(1)}_1}}{z}\right ) \prod_{n\in\rran{2m-1}} \delta_{i_{A^{(2)}_n}, \pi(n)} \delta\left (\frac{z_{A^{(2)}_{\sigma(n)}} q^{\nu_m(n)}}{v}\right )
\nn\\
&&\times\left\{ G_{i_{A^{(1)}_1},0 }^+(C^{1/2}z_{A^{(1)}_1}/v) R_\bA^<(z_\bA) - R_\bA^>(z_\bA^{-1})\right \}\,,\nn
\eea 
where
\be
R_\bA^<(z_\bA) = \prod_{\substack{n\in\rran{2m-1}\\ A^{(2)}_n>A^{(1)}_1}} G_{i_{A^{(1)}_1}, i_{A^{(2)}_n}}^-(C^{-1/2}z_{A^{(1)}_1}/z_{A^{(2)}_n}) \,,\nn\ee
\be
R_\bA^>(z_\bA^{-1})= \prod_{\substack{n\in\rran{2m-1}\\ A^{(2)}_n<A^{(1)}_1}} G_{i_{A^{(1)}_1}, i_{A^{(2)}_n}}^-(C^{1/2}z_{A^{(2)}_n}/z_{A^{(1)}_1})\,.\nn
\ee
Hence, upon rewriting, we get
\bea && \left\langle \x+{i_1}(z_1) \cdots \x+{i_{2m}}(z_{2m}) , \left [\X+{1,-m}(v), \x-{1}(z)\right ]\right \rangle \nn\\
&=& -\frac{\left[2\right]_q}{(q-q^{-1})^2} \sum_{\substack{\bA\in\Prt{(1,2m-1)}{\rran{2m}}\\ \sigma\in S_{2m-1}^< }}\beta_{m,\sigma} \delta_{i_{A^{(1)}_1}, 1} \delta\left(\frac{z_{A^{(1)}_1}}{z}\right ) \prod_{n\in\rran{2m-1}} \delta_{i_{A^{(2)}_n}, \pi(n)} \delta\left (\frac{z_{A^{(2)}_{\sigma(n)}} q^{\nu_m(n)}}{v}\right )
\nn\\
&&\times
\left\{ G_{1,0 }^+(C^{1/2}z/v) Q_{\sigma, \bA}^<(z/v) - Q_{\sigma, \bA}^>(v/z)\right \}\,,\nn
\eea 
where
\be
Q_{\sigma, \bA}^<(z/v) = \prod_{\substack{n\in\rran{2m-1}\\ A^{(2)}_{\sigma(n)}>A^{(1)}_1}} G_{1, \pi(n)}^-(C^{-1/2}zq^{\nu_m(n)}/v)\,, 
\nn\ee
\be
Q_{\sigma, \bA}^>(v/z) =  \prod_{\substack{n\in\rran{2m-1}\\ A^{(2)}_{\sigma(n)}<A^{(1)}_1}} G_{1, \pi(n)}^-(C^{1/2}vq^{-\nu_m(n)}/z)\,. \nn\ee
In view of (\ref{eq:x-X+m}), the contributions to $\left\langle \x+{i_1}(z_1) \cdots \x+{i_{2m}}(z_{2m}) , \bz-{-m}(z) \right \rangle$ in the above expression must come from terms with a pole at $z=C^{1/2}v$. The latter happen for factors in $Q_{\sigma, \bA}^<(z/v) $ or $Q_{\sigma, \bA}^>(v/z)$ such that $\nu_m(n) = c_{1,\pi(n)}$ or $\nu_m(n)=-c_{1,\pi(n)}$ respectively. We thus let
$$M_m^\pm=\{n\in\rran{2m-1} : \nu_m(n) = \pm c_{1,\pi(n)} \}\,.$$
Upon inspection, one easily sees that
$$M_m^+ = \{2m-4\}\,, \qquad \mbox{whereas} \qquad M_m^-=\{2m-1\}\,.$$
Now, for $\sigma\in S_{2m-1}^<$, we have $\sigma(2m-4)<\sigma(2m-1)$ and no term has a pole at $z=C^{1/2}v$. We conclude that $\left\langle \x+{i_1}(z_1) \cdots \x+{i_{2m}}(z_{2m}) , \bz-{-m}(z) \right \rangle=0$.
\end{proof}

\subsection{Relations in $\Psi^{-1}(\qdaff^0(\mathfrak a_1))$
}
\begin{defn}
We set $\KK{+}{1,0}(v) := - \km{1}(C^{1/2}v)$ and, for every $m\in\N^\times$, $$\KK{+}{1,m}(v):=(q-q^{-1})\km{1}(C^{1/2}vq^{-2m})\bpsi{+}{1,m}(v)\,.$$ We then let
$$\KK{-}{1,0}(v) := \varphi \left (\KK{+}{1,0}(1/v) \right ) = - \kp{1}(C^{1/2}v)$$
and, for every $m\in\N^\times$,
$$\KK{-}{1,-m}(v):= \varphi\left (\KK{+}{1,m}(1/v) \right ) =-(q-q^{-1}) \bpsi{-}{1,-m}(v)\kp{1}(C^{1/2}vq^{-2m})\,.$$
\end{defn}
It is straigthforward to establish that
\be \km{1}(C^{1/2} w) \bpsi+{1,m}(v) = G^+_{11}\left (\frac{wq^{2m}}{v}\right ) G^-_{11} \left (\frac{w}{v} \right ) \bpsi+{1,m}(v) \km{1}(C^{1/2}w)\,.\ee
By making repeated use of the above relation, one readily checks that, in terms of $(\KK{+}{1,m}(v))_{m\in \N^\times}$, the relations (\ref{eq:Ym}) and (\ref{eq:psipsimm}), as well as the relations in corollary \ref{cor:psim}\emph{ii.} and \emph{iv.} of the previous subsections respectively read
\bea [\x +1(v), \Xm{1,n}(z)] &=& \frac{1}{q-q^{-1}}\delta\left(\frac{zq^{2n}}{Cv}\right) \left (\prod_{p=0}^{n-1} \bz-{0}(C^{-1/2} zq^{2p})^{-1} \right )\KK{+}{1,n}(v) 
\label{eq:xpxmk}
\eea
\bea [\KK{+}{1,1}(w), \KK{+}{1,m}(v)]_{G^-_{11}(w/v)G^+_{11}(wq^{2(m-1)}/v)} &=& [2]_q \left \{\delta \left (\frac{wq^{2m}}{v}\right ) \KK{+}{1,0}(vq^{-2m}) \KK{+}{1,m+1}(v) \right . \nn\\
&& \left .  -\delta \left (\frac{w}{vq^2}\right ) \KK{+}{1,0}(v) \KK{+}{1,m+1}(vq^2)  \right \}\eea
\bea [\KK{+}{1,m+1}(v), \Xm{1,n}(z)]_{G^+_{11}(Cv/zq^{2(m+1)})} &=& - \bz-{0}(C^{1/2}vq^{-2(m+1)}) [\Xm{1,n+1}(zq^{-2}), \KK{+}{1,m}(v)]_{G^+_{11}(zq^{2(m-1)}/Cv)}\nn\\
&\propto& \delta \left (\frac{zq^{2m}}{Cv} \right) \eea
\bea\label{eq:K+X+} [\KK{+}{1,m+1}(v), \Xp{1,n}(z)]_{G^-_{11}(v/z)} = -[\Xp{1,n+1}(v), \KK{+}{1,m}(z)]_{G^-_{11}(v/z)} \propto \delta \left (\frac{zq^2}{v} \right )\eea
\begin{prop}
For every $m,n\in\N$, we have
\be\label{eq:proofK+K+} (v-q^{\pm 2}z)(v-q^{2(m-n \mp 1)}z) \KK\pm{1,\pm m}(v) \KK\pm{1,\pm n}(z) = (vq^{\pm 2}-z)(vq^{\mp2}-q^{2(m-n)}z) \KK\pm{1,\pm n}(z) \KK\pm{1,\pm m}(v)\,,\ee
\end{prop}
\begin{proof}
We apply the map $a\mapsto [a,\Xm{1,n}(u)]_{G^+_{11}(Cv/uq^{2(m+1)})}$ to the relation (\ref{eq:K+X+}) with $n=0$ there. Making use of identity (\ref{eq:identity1}) on the left hand side, we get 
\bea
&&[\underbrace{[\KK+{1,m+1}(v), \X-{1,n}(u)]_{G_{11}^+(Cq^{-2(m+1)}v/u)}}_{\propto \delta\left(\frac{C^{-1}uq^{2m}}{v}\right)}, \xp 1(z) ]_{G_{10}^+(v/z)} \nn\\
&&\qquad \qquad \qquad \qquad \qquad \qquad \qquad \qquad  + \left[\KK+{1,m+1}(v) , \left[\xp 1(z), \X-{1,n}(u)\right ]\right ]_{G_{10}^+(v/z)G_{11}^+(Cq^{-2(m+1)}v/u)}\propto \delta\left(\frac{zq^2}{v}\right) \nn
\eea
Multiplying through by $\left (C^{-1}uq^{2m}-v\right)\left(zq^2-v\right)$ and making use of (\ref{eq:xpxmk}), it follows that
\bea
0&=& \left (C^{-1}uq^{2m}-v\right)\left(zq^2-v\right)\delta\left (\frac{uq^{2n}}{Cz}\right )\left[\KK+{1,m+1}(v) ,\KK+{1,n}(z)\right ]_{G_{10}^+(v/z)G_{11}^+(Cq^{-2(m+1)}v/u)}\nn\\
&=& \left( zq^{2(m-n)}-v\right ) \left (zq^2-v\right ) \delta\left (\frac{uq^{2n}}{Cz}\right )\left[\KK+{1,m+1}(v) ,\KK+{1,n}(z)\right ]_{G_{10}^+(v/z)G_{11}^+(q^{2(n-m-1)}v/z)}\nn
\eea
Hence the result for the upper choice of signs in (\ref{eq:proofK+K+}). The case with lower choice of signs follows by applying $\varphi$ to the above equation.
\end{proof}

At this point it should be clear that we have obtained $\Psi^{-1}$. Indeed, it suffices to let, for every $m\in\N$ and every $n\in\Z$,
\be \Psi^{-1}(\Dsf_2^{\pm 1}) = D^{\pm 1}\ee
\be\Psi^{-1}(\Csf^{\pm 1/2}) = C^{\pm1/2}\ee
\be \Psi^{-1}(\cbsf\pm(z)) = \bz\pm0(z)\ee
\be \Psi^{-1}(\Kbsf\pm{1,\pm m}(z)) = \KK\pm{1,\pm m}(z)\ee
\be \Psi^{-1}(\Xbsf\pm{1,n}(z)) = \X\pm{1,n}(z)\ee
The relations in $\qdaff'(\mathfrak a_1)$ are obviously all the relations we have derived in the present section. $\Psi^{-1}$ therefore extends as an algebra homomorphism. This concludes the proof of theorem \ref{thm:main}.

Returning to the proof of proposition \ref{prop:Hall}, it is also clear that
\be f(\psi^\pm(z)) = (q^2-q^{-2})^2 \Psi(\bwp\pm(C^{1/2}zq^{-2} ))\ee 
\be \label{eq:epsi}f(\eb\pm(z)) = \Psi(\bpsi\pm{1,\pm 1}(z))\ee
Therefore ((\ref{eq:thetap}) -- (\ref{eq:thetam})) follow from proposition \ref{prop:psixm}.\emph{v.} In order to complete the proof of proposition \ref{prop:Hall}, we still have to prove the compatibility of $f$ with the Serre relations (\ref{eq:e+e+e+}) of $\E{q_1}{q_2}{q_3}$. This is the purpose of the next section.

\subsection{The Serre relations of the elliptic Hall algebra}
By the compatibility of $f$ with (\ref{eq:e+e+e+}), we actually mean that we should have, for every $m\in\Z$,
\be\label{eq:fSerre}  \res_{v,w,z} (vwz)^m(v+z)(w^2-vz) f(\eb\pm(v)) f(\eb\pm(w)) f(\eb\pm(z))=0\,.\ee
Now we have already identified $f(\eb\pm(z))$ with $\Psi(\bpsi\pm{1,\pm 1}(z))$ in (\ref{eq:epsi}) above. The latter means that proving (\ref{eq:fSerre}) is equivalent to proving
\begin{prop}
For every $m\in\Z$, we have
\be \res_{v_1,v_2,v_3} (v_1v_2v_3)^m(v_1+v_3)(v_2^2-v_1v_3) \bpsi\pm{1,\pm 1}(v_1) \bpsi\pm{1,\pm 1}(v_2) \bpsi\pm{1,\pm 1}(v_3)=0\,.\ee
\end{prop}
\begin{proof}
The upper choice of signs immediately follows from the lower one upon applying $\varphi$. Moreover, considering the root space decomposition, it is clear that having
\be \res_{v_1,v_2,v_3} (v_1v_2v_3)^m(v_1+v_3)(v_2^2-v_1v_3) \left\langle \x+{i_1}(z_1) \dots \x+{i_6}(z_6), \bpsi-{1,- 1}(v_1) \bpsi-{1,- 1}(v_2) \bpsi-{1,-1}(v_3)\right \rangle = 0\nn\ee
for every $i_1, \dots, i_6\in \dot{I}$ is a sufficient condition for the result to hold. Now, making use of lemma \ref{lem:deltapsi}, one easily obtains that
\bea &&\left\langle \x+{I_1}(z_1) \dots \x+{i_6}(z_6), \bpsi-{1,- 1}(v_1) \bpsi-{1,- 1}(v_2) \bpsi-{1,-1}(v_3)\right \rangle\nn\\
&=& \left (\frac{\left[2\right]_q}{q-q^{-1}}\right )^3  \sum_{\bA\in\Prt{(2,2,2)}{\rran{6}}} \,\,\prod_{p=1}^3 \,\,\prod_{\substack{m\in A^{(p+1)} \sqcup\dots\sqcup A^{(3)}\\n\in A^{(p)}\\n>m }} \,\,\prod_{k=1}^2 \delta_{i_{A^{(p)}_k, \varpi(k)}} \delta\left (\frac{z_{A^{(p)}_k} q^{\varpi(k)} }{v_k}\right ) G_{i_m,i_n}^-(C^{-1/2} z_m/z_n)\,.\nn \eea
There is obviously an action of $S_3$ on $\Prt{(2,2,2)}{\rran{6}}$ given by setting
$\sigma(\bA) = (A^{(\sigma(1))}, A^{(\sigma(2))}, A^{(\sigma(3))})$
for every $\sigma\in S_3$ and every $\bA\in \Prt{(2,2,2)}{\rran{6}}$. It is also quite clear that
$$\frac{\Prt{(2,2,2)}{\rran{6}}}{S_3} \cong \Trt{(2,2,2)}{\rran{6}} \,,$$
where
$$\Trt{(2,2,2)}{\rran{6}}:= \left\{\bA\in\Prt{(2,2,2)}{\rran{6}} : A^{(1)}_1 < A^{(2)}_1 < A^{(3)}_1 \right \}\,.$$
For every triple $\textbf{n} = \{n_1, n_2, n_3\}\subset \rran{6}$, we further let
$$\Trt{(2,2,2)}{\rran{6}}(\textbf n) := \left\{\bA \in \Trt{(2,2,2)}{\rran{6}} : \left\{A^{(p)}_2  : p\in \rran{3} \right \} = \textbf n \right \}\,.$$
With these notations in place, we can now write
\bea && \res_{v_1,v_2,v_3} (v_1v_2v_3)^m(v_1+v_3)(v_2^2-v_1v_3) \left\langle \x+{I_1}(z_1) \dots \x+{i_6}(z_6), \bpsi-{1,- 1}(v_1) \bpsi-{1,- 1}(v_2) \bpsi-{1,-1}(v_3)\right \rangle \nn\\
&=& \left (\frac{\left[2\right]_q}{q-q^{-1}}\right )^3 \sum_{\substack{\textbf{n} \subset \rran{6}\\\card \textbf{n} =3}} z_{\bn}^m \delta_{i_{\rran{6}-\bn}, 1} \delta_{i_\bn, 0} \sum_{\bA\in \Trt{(2,2,2)}{\rran{6}}(\textbf n)} \prod_{p=1}^3 \delta\left (\frac{z_{A^{(p)}_1}q^2}{z_{A^{(p)}_2}}\right ) c_\bA\,, \nn\eea
where, by definition,
\be z_\bn^m = \prod_{i=1}^3 z_{n_i}^m\,, \qquad \delta_{i_\bn, 0} = \prod_{j=1}^3 \delta_{i_{n_j}, 0}\,, \qquad \delta_{i_{\rran{6}-\bn},1} = \prod_{m\in \rran{6}-\bn} \delta_{i_m, 1}\ee
\be c_\bA = \sum_{\sigma\in S_3} F(z_{A^{(\sigma(1))}_2}, z_{A^{(\sigma(2))}_2}, z_{A^{(\sigma(3))}_2}) \prod_{1\leq p'< p\leq 3} H_{\bA, \sigma, p,p'} (z_{A^{(\sigma(p))}_2}/z_{A^{(\sigma(p'))}_2}) \ee
\be F(x,y,z) = (x+z)(y^2-xz)\,\ee
\be H_{\bA, \sigma, p,p'} (z_{A^{(\sigma(p))}_2}/z_{A^{(\sigma(p'))}_2}) = \prod_{k,k'=1}^2 G_{\varpi(k), \varpi(k')}^-(C^{-1/2} q^{2(k-k')}z_{A^{(\sigma(p))}_2 }/z_{A^{(\sigma(p'))}_2})^{\epsilon(\bA, \sigma, p,p', k,k')}\ee
\be \epsilon(\bA, \sigma, p,p', k,k') =\begin{cases}
1&\mbox{if $A^{(\sigma(p))}_k < A^{(\sigma(p'))}_{k'} $;}\\
0 &\mbox{otherwise.}
\end{cases}\ee
Denoting each $\bA \in \Trt{(2,2,2)}{\rran{6}}$ as the tableau
$$\Yboxdim{22pt}
\young(<A^{(1)}_1><A^{(2)}_1><A^{(3)}_1>,<A^{(1)}_2><A^{(2)}_2><A^{(3)}_2)\,,$$
one easily checks that, actually,
\be \Trt{(2,2,2)}{\rran{6}} = \Trt{(2,2,2)}{\rran{6}}\left (\left\{2,4,6\right \}\right) \sqcup \Trt{(2,2,2)}{\rran{6}}\left (\left\{2,5,6\right \}\right)\sqcup \Trt{(2,2,2)}{\rran{6}}\left (\left\{3,4,6\right \}\right) \sqcup\Trt{(2,2,2)}{\rran{6}}\left (\left\{3,5,6\right \}\right) \sqcup\Trt{(2,2,2)}{\rran{6}}\left (\left\{4,5,6\right \}\right)\, ,\nn
\ee
with
$$\Trt{(2,2,2)}{\rran{6}}\left (\left\{2,4,6\right \}\right) = \left \{
\begin{tabular}{|c|c|c|}
\hline
1&3&5\\
\hline
2&4&6\\
\hline
\end{tabular}
\right \}\,,$$
$$ \Trt{(2,2,2)}{\rran{6}}\left (\left\{2,5,6\right \}\right) = \left \{
\begin{tabular}{|c|c|c|}
\hline
1&3&4\\
\hline
2&6&5\\
\hline
\end{tabular}\,,
\begin{tabular}{|c|c|c|}
\hline
1&3&4\\
\hline
2&5&6\\
\hline
\end{tabular}
\right \}\,,$$
$$\Trt{(2,2,2)}{\rran{6}}\left (\left\{3,4,6\right \}\right)=\left \{
\begin{tabular}{|c|c|c|}
\hline
1&2&5\\
\hline
4&3&6\\
\hline
\end{tabular}\,,
\begin{tabular}{|c|c|c|}
\hline
1&2&5\\
\hline
3&4&6\\
\hline
\end{tabular}
\right \}\,,$$
$$\Trt{(2,2,2)}{\rran{6}}\left (\left\{3,5,6\right \}\right)=\left \{
\begin{tabular}{|c|c|c|}
\hline
1&2&4\\
\hline
6&3&5\\
\hline
\end{tabular}\,,
\begin{tabular}{|c|c|c|}
\hline
1&2&4\\
\hline
3&6&5\\
\hline
\end{tabular}\,,
\begin{tabular}{|c|c|c|}
\hline
1&2&4\\
\hline
5&3&6\\
\hline
\end{tabular}\,,
\begin{tabular}{|c|c|c|}
\hline
1&2&4\\
\hline
3&5&6\\
\hline
\end{tabular}
\right \}\,,$$
$$\Trt{(2,2,2)}{\rran{6}}\left (\left\{4,5,6\right \}\right)=\left\{
\begin{tabular}{|c|c|c|}
\hline
1&2&3\\
\hline
6&5&4\\
\hline
\end{tabular}\,,
\begin{tabular}{|c|c|c|}
\hline
1&2&3\\
\hline
5&6&4\\
\hline
\end{tabular}\,,
\begin{tabular}{|c|c|c|}
\hline
1&2&3\\
\hline
6&4&5\\
\hline
\end{tabular}\,,
\begin{tabular}{|c|c|c|}
\hline
1&2&3\\
\hline
4&6&5\\
\hline
\end{tabular}\,,
\begin{tabular}{|c|c|c|}
\hline
1&2&3\\
\hline
5&4&6\\
\hline
\end{tabular}\,,
\begin{tabular}{|c|c|c|}
\hline
1&2&3\\
\hline
4&5&6\\
\hline
\end{tabular}
\right\}\,.$$
A tedious but straightforward calculation -- see appendix for useful identities -- shows that, e.g.
\bea 
c_{\,\,\Yboxdim{10pt}\young(124,635)} &=& (q^2-q^{-2})^2 (1+q^2)(1-q^2)z_3^3 \left [q^2\delta \left (\frac{z_3q^2}{z_6} \right )\delta\left(\frac{z_6}{z_5}\right ) - \delta\left(\frac{z_3}{z_6}\right ) \delta \left (\frac{z_6q^2}{z_5} \right )   \right ] \nn\\
&& +q^{-2}(q^2-q^{-2}) (1+q^2)^2(1-q^2)^6 H_1(z_3/z_5)\left[z_5 \delta \left (\frac{z_6}{z_5} \right ) - z_3 \delta \left (\frac{z_3}{z_6} \right ) \right ]\nn
\eea
\bea c_{\,\,\Yboxdim{10pt}\young(124,365)} &=& (q^2-q^{-2})^2 (1+q^2)(1-q^2)z_3^3  \delta\left(\frac{z_3}{z_6}\right ) \delta \left (\frac{z_6q^2}{z_5} \right ) \nn\\
&&+q^{-2}(q^2-q^{-2}) (1+q^2)^2(1-q^2)^6 H_1(z_3/z_5)z_3 \delta \left (\frac{z_3}{z_6} \right )  \nn\\
&&+q^{-2}(q^2-q^{-2}) (1+q^2)^2(1-q^2)^6 H_2(z_3/z_5)z_5 \delta \left (\frac{z_6}{z_5} \right )  \nn
\eea
\bea c_{\,\,\Yboxdim{10pt}\young(124,536)} &=& - q^2(q^2-q^{-2})^2(1+q^2)(1-q^2)z_3^3\delta\left(\frac{z_5}{z_6}\right ) \delta\left (\frac{z_3q^2}{z_5}\right )\nn\\
&&-q^{-2}(q^2-q^{-2})(1+q^2)(1-q^2)^6 H_1(z_3/z_5) z_5 \delta \left(\frac{z_5}{z_6}\right ) \nn\\
&&-q^{-2}(q^2-q^{-2})(1+q^2)(1-q^2)^6 H_2(z_3/z_6) z_3 \delta \left(\frac{z_3}{z_5}\right ) \nn
\eea
\be c_{\,\,\Yboxdim{10pt}\young(124,356)} = q^{-2}(q^2-q^{-2})(1+q^2)(1-q^2)^6 H_2(z_3/z_6)\left[z_3 \delta \left(\frac{z_3}{z_5}\right )- z_6 \delta\left(\frac{z_5}{z_6}\right ) \right ]\,.\nn\ee
where we have set
$$H_1(z_3/z_5) = \left(\frac{z_3^2z_5^2(z_3+z_5)^3}{(z_3q^2-z_5)(q^4z_3-z_5)(z_3-q^2z_5)^3}\right )_{|z_5|\gg|z_3|} \,,$$
$$H_2(z_3/z_5)= \left(\frac{z_3^2z_5^2(z_3+z_5)^3}{(z_3q^4-z_5)(z_3-q^2z_5)^4}\right )_{|z_5|\gg|z_3|}\,.$$
It easily follows that
\be \sum_{\bA\in \Trt{(2,2,2)}{\rran{6}}(\left\{3,5,6\right\})} \prod_{p=1}^3 \delta\left (\frac{z_{A^{(p)}_1}q^2}{z_{A^{(p)}_2}}\right ) c_\bA =0\,.\ee
Similar calculations show that, eventually, for every $\bn \subset \rran{6}$ such that $\card \bn =3$ and $\Trt{(2,2,2)}{\rran{6}}(\textbf n)\neq \emptyset$, we have
\be \sum_{\bA\in \Trt{(2,2,2)}{\rran{6}}(\textbf n)} \prod_{p=1}^3 \delta\left (\frac{z_{A^{(p)}_1}q^2}{z_{A^{(p)}_2}}\right ) c_\bA =0\,,\ee
thus proving the result.
\end{proof}

\appendix
\section{Formal distributions}
\subsection{Definitions and main properties}
Let $\K$ be a field of characteristic $0$. For any $\K$-vector space $V$, we let $V[z,z^{-1}]$ denote the ring of $V$-valued Laurent polynomials. Writing
$$v(z) = \sum_{n\in\Z} v_n z^n\,,$$
where the sum runs over finitely many terms, for any $v(z)\in V[z,z^{-1}]$, we can define
$$\supp (v(z)) = \left\{n\in \Z : v_n\neq 0\right \}\,,$$
and set
$$V_n[z,z^{-1}]:= \left\{v(z)\in V[z,z^{-1}] : \supp(v(z)) \subseteq \range{-n}{n}\right\}\,.$$
It is clear that, for every $n\in\N$, $V_n[z,z^{-1}] \cong V^{2n+1}$ as $\K$-vector spaces. Now, if in addition $V$ is a topological vector space with topology $\tau_1$, making use of that isomorphism, we can endow $V_n[z,z^{-1}]$ with the box topology of $V^{2n+1}$, for every $n\in\N$. Denote by $\tau_n$ that topology. 

The obvious inclusions $V_n[z,z^{-1}] \hookrightarrow V_{n+1}[z,z^{-1}]$ are clearly continuous and we define a topology $\tau$ on $V[z,z^{-1}]$ as the inductive limit
$$\tau:= \lim_{\longrightarrow} \tau_n\,.$$
We now assume that $\K$ is a topological field.
\begin{defn}
The space $V[[z,z^{-1}]]$ of $V$-valued \emph{formal distributions} is the $\K$-vector space of continuous $V$-valued linear functions over the ring of $\K$-valued Laurent polynomials $\K[z,z^{-1}]$, the latter being endowed with the final topology induced as above from the topology of $\K$.
\end{defn}
\begin{prop}
Any $V$-valued formal distribution $v(z)\in V[[z,z^{-1}]]$ reads
$$v(z) = \sum_{n\in\Z} v_n z^{n}\,,$$
for some $(v_n)_{n\in\Z} \in V^\Z$ and the action of $v(z)$ on any Laurent polynomial $f(z)\in \K[z,z^{-1}]$ is given by
$$\left \langle v(z), f(z) \right\rangle  = \res_z\left (v(z) f(z)z^{-1}\right )\,,$$
where we let
$$\res_z a(z) = \res_z \left (\sum_{n\in\Z} a_n z^n \right) = a_{-1}\,,$$
for any $a(z)\in V[[z,z^{-1}]]$. $V[[z,z^{-1}]]$ is given the weak $*$-topology. It is actually a module over the ring $\K[z,z^{-1}]$ of $\K$-valued Laurent polynomials.
\end{prop}
\begin{proof}
It is clear that, due to its linearity, any $v(z)\in V[[z,z^{-1}]]$ is entirely characterized by the data, for every $n\in\N$, of
$$v_n = \left\langle v(z), z^{-n}\right\rangle \in V\,.$$
Now, writing $v(z)=\sum_{n\in\Z} v_nz^n$, we also have
$$v_n = \left\langle v(z), z^{-n}\right\rangle = \res_z v(z)z^{-n-1}\,,$$
for every $n\in\N$ and the claim follows.\end{proof}
Let $A$ be a topological $\K$-algebra. Then $A[[z,z^{-1}]]$ is the space of $A$-valued formal distributions, i.e. of $A$-valued linear functions over $A[z,z^{-1}]$. In that case, the action of $a(z)\in A[[z,z^{-1}]]$ on $b(z)\in A[z,z^{-1}]$ is given by
\be\left\langle a(z), b(z) \right\rangle = \res_z a(z)b(z)z^{-1}\,.\ee
Clearly, $A[[z,z^{-1}]]$ is a module over the ring $A[z,z^{-1}]$ of $A$-valued Laurent polynomials. It is generally impossible to consistently extend that structure into a full-fledged product over $A[[z,z^{-1}]]$. However, since $A$ is a topological algebra, we can set
$$a(z)b(z) =  \sum_{p\in\Z}  \left (\sum_{m\in\Z} a_m b_{p-m}\right ) z^p\,,$$
whenever the series
$$\sum_{m\in\Z} a_mb_{p-m}$$
is convergent for every $p\in \Z$. If $A$ is complete as a topological algebra, it suffices that the above series be Cauchy.

We let similarly $V[[z_1,z_1^{-1}, \dots, z_n, z_n^{-1}]]$ denote the space of $V$-valued formal distributions in $n\in \N$ variables, so that any $V$-valued formal distribution $v(z_1, \dots, z_n)$ in $n$ variable reads 
$$v(z_1, \dots, z_n) = \sum_{p_1, \dots, p_n\in\Z} v_{p_1, \dots, p_n} z_1^{p_1} \cdots z_n^{p_n}\,,$$
for some $(v_{p_1\dots, p_n})_{p_1, \dots, p_n \in\Z} \in V^{\Z^n}$. For every $i=1,\dots, n$, we define 
$$\res_{z_i} : V[[z_1,z_1^{-1}, \dots z_n, z_n^{-1}]] \to V[[z_1,z_1^{-1}, \dots, \widehat{z_i}, \widehat{z_i^{-1}} \dots, z_n, z_n^{-1}]]\,,$$
where a hat over a variable indicates omission of that variable, by setting
$$\res_{z_i} v(z_1, \dots, z_n) = \res_{z_i} \sum_{p_1, \dots, p_n\in\Z} v_{p_1, \dots, p_n} z_1^{p_1} \cdots z_n^{p_n} = \sum_{p_1, \dots, \hat{p}_i, \dots, p_n \in\Z} v_{p_1,\dots, p_{i-1},-1,p_{i+1},  \dots, p_n} z_1^{p_1} \cdots \widehat{z_i^{-1}} \cdots z_n^{p_n}\,.$$
For every $i=1,\dots, n$, we define $\partial_i:V[[z_1, z_1^{-1}, \dots, z_n,z_n^{-1}]] \to V[[z_1, z_1^{-1}, \dots, z_n,z_n^{-1}]]$ by setting
$$\partial_i v(z_1, \dots, z_n) = \sum_{p_1, \dots, p_n\in\Z} p_i v_{p_1, \dots, p_n} z_1^{p_1} \cdots z_i^{p_i-1} \cdots z_n^{p_n}\,.$$

If $A$ is a topological $\K$-algebra, then the multiplication in $A$ naturally extends to bilinear maps $$A[[z_1, z_1^{-1},\dots, z_m, z_m^{-1}]] \times  A[[z_{m+1}, z_{m+1}^{-1},\dots, z_{m+n}, z_{m+n}^{-1}]] \to A[[z_1,z_1^{-1}, \dots, z_{m+n}, z_{m+n}^{-1}]]$$ by setting
$$a(z_1, \dots, z_{m})b(z_{m+1}, \dots, z_{m+n}) = \sum_{p_1,\dots, p_{m+n}\in\Z} a_{p_1, \dots, p_{m}} b_{p_{m+1}, \dots, p_{m+n}} z_1^{p_1} \cdots z_{m+n}^{p_{m+n}}\,.$$

Let $a(z_1,\dots, z_n) \in A[[z_1, z_1^{-1},\dots, z_n,z_n^{-1}]]$ be an $A$-valued formal distribution in $n$ variables. Since $A$ is a topological $\K$-agebra, we can define the \emph{localization} $a_{|z_{n-1}=z_n}(z_1, \dots, z_{n-1}) \in A[[z_1, z_1^{-1}, \dots, z_{n-1}, z_{n-1}^{-1}]]$ of $a(z_1,\dots, z_n)$ at $z_{n-1}=z_n$, by setting
$$a_{|z_{n-1}=z_n} (z_1,\dots, z_{n-1})=\sum_{p_1,\dots, p_{n-1}\in\Z} \left (\sum_{p\in\Z} a_{p_1, \dots,p_{n-2}, p, p_{n-1}-p}\right )  z_1^{p_1} \dots z_{n-1}^{p_{n-1}}   \, ,$$
whenever
$$\sum_{p\in\Z} a_{p_1, \dots,p_{n-2}, p, p_{n-1}-p}$$
is convergent. If $A$ is complete as a topological algebra, it suffices that the above series be Cauchy.

\subsection{Laurent expansion and the Dirac formal distribution}
One way to obtain formal power series is to take the Laurent expansion of some holomorphic function $f:\C\to\C$. We shall usually write $f(z)_{|z|\ll 1}$ to denote the Laurent expansion around $0$. Similarly, we shall denote by $f(z)_{|z|\gg 1}$ the Laurent expansion around $\infty$.

Let
$$\delta(z) = \sum_{n\in\Z} z^n\,.$$
\begin{lem}
\label{lem:deltap}
For every $n\in\N^\times$, we have
$$\left (\frac{1}{1-z}\right )^n_{|z|\ll 1} -\left (\frac{1}{1-z}\right )^n_{|z|\gg 1} = \frac{\delta^{(n-1)}(z)}{(n-1)!}\,.$$
\end{lem}
\begin{proof}
It is straightforward to check that the result holds for $n=1$. Assuming it holds for some $n$, it follows, upon differentiation, that
$$n\left [\left (\frac{1}{1-z}\right )^{n+1}_{|z|\ll 1} -\left (\frac{1}{1-z}\right )^{n+1}_{|z|\gg 1} \right ]= \frac{\delta^{(n)}(z)}{(n-1)!}\,,$$
which completes te recursion.
\end{proof}

\begin{lem}
For any $n\in\N$ and any $A$-valued Laurent polynomial $f(z)\in A[z,z^{-1}]$, we have
$$f(z)\delta^{(n)}(z) = \sum_{p=0}^n (-1)^{n-p}{n\choose p} f^{(n-p)}(1) \delta^{(p)}(z)\,.$$
\end{lem}
\begin{proof}
The case $n=1$ is straightforward. Assuming the results holds for some $n\in\N$, we have, upon differentiation,
$$f'(z)\delta^{(n)}(z) + f(z) \delta^{(n+1)}(z) = \sum_{p=0}^n (-1)^{n-p}{n\choose p} f^{(n-p)}(1) \delta^{(p+1)}(z) \,,$$
which completes the recursion.
\end{proof}

\begin{exmp}
In particular, for any $A$-valued formal distribution $f(z_1, z_2)\in A[[z_1, z_1^{-1}, z_2, z_2^{-1}]]$ with a well-defined localization $f_{|z_1=z_2}(z_1)$ -- see previous subsection for a definition --, we have
$$f(z_1, z_2) \delta\left(\frac{z_1}{z_2}\right )  = f_{|z_1=z_2}(z_1) \delta\left(\frac{z_1}{z_2}\right )  \,,$$
\end{exmp}

Assuming that $\K$ is an algebraically closed field, we have
 \begin{lem}
 Let $P(z) \in \K[z]$ be a polynomial of degree $N$, with roots $\{\lambda_i : i\in \rran n\}$ and respective multiplicities $\left \{m_i:i\in\rran{n}\right \}$. If $a(z)\in\K[[z,z^{-1}]]$ is a $\K$-valued formal distribution, then
 $$P(z) a(z) = 0 \qquad \Leftrightarrow \qquad a(z) = \sum_{i=1}^n\sum_{p_i=0}^{m_i-1} \alpha_{i,p} \delta^{(p_i)} \left(\frac{z}{\lambda_i}\right ) \,,$$
 for some $\alpha_{i,p}\in \K$.
 \end{lem}
 \begin{proof}
 The if part is easily checked making use of the previous lemma. The only if part follows by an easy recursion, after writing that $P(z) =\prod_{i\in\rran n} (z-\lambda_i)^{m_i}$.
 \end{proof}
 \begin{lem}
  \label{lem:ratdelta}
 Let $P(z), Q(z) \in \K[z]$ be two coprime polynomials. Let $\{\lambda_i : i\in\rran n\}$ be the set of roots of $Q(z)$ and let $\{m_i : i\in\rran n\}$ be their respective multiplicities. Then, in $\K[[z,z^{-1}]]$,
 \be\label{eq:P/Qdelta}\left (\frac{P(z)}{Q(z)}\right )_{|z|\ll 1} - \left (\frac{P(z)}{Q(z)}\right )_{|z|\gg 1} =  \sum_{i=1}^n\sum_{p_i=0}^{m_i-1} \frac{(-1)^{p_i+1} \alpha_{i,p_i+1}}{(p_i)! \lambda_i^{p_i+1}} \delta^{(p_i)} \left(\frac{z}{\lambda_i}\right ) \,,\ee
 where, for every $i\in\rran{n}$ and every $p_i\in \rran{m_i}$, $\alpha_{i,p_i}$ is obtained from the partial fraction decomposition
 \be\label{eq:partfracdec}\frac{P(z)}{Q(z)} = A(z)+ \sum_{i=1}^n\sum_{p_i=1}^{m_i} \frac{\alpha_{i,p_i}}{(z-\lambda_i)^{p_i}} \,,\ee
 in which $A(z)\in\K[z]$ is a polynomial of degree $\deg(P) - \deg(Q)$.
 \end{lem}
 \begin{proof}
 Given the partial fraction decomposition (\ref{eq:partfracdec}), we can write
 \bea\left (\frac{P(z)}{Q(z)}\right )_{|z|\ll 1} - \left (\frac{P(z)}{Q(z)}\right )_{|z|\gg 1}  &=&  \sum_{i=1}^n\sum_{p_i=1}^{m_i} \alpha_{i,p_i}\left[\left( \frac{1}{(z-\lambda_i)^{p_i}}\right )_{|z|\ll 1} -\left( \frac{1}{(z-\lambda_i)^{p_i}}\right )_{|z|\gg 1}\right]\nn\\
&=& \sum_{i=1}^n\sum_{p_i=1}^{m_i} \frac{(-1)^{p_i} \alpha_{i,p_i}}{(p_i-1)! \lambda_i^{p_i}} \delta^{(p_i-1)}\left (\frac{z}{\lambda_i}\right )\nn
\eea
where we have used lemma \ref{lem:deltap} to derive the last equality. The claim obviously follows.
 \end{proof}

\subsection{The structure power series $G_{ij}^\pm(z)$} 
\label{sec:structfunc}
In this last subsection, we derive identities involving the structure power series $G_{ij}^\pm(z)$ by applying lemma \ref{lem:ratdelta}. Remember -- see remark \ref{rem:Gij} -- that in type $\dot{\mathfrak a}_1$, we have $G_{10}^\pm(z) = G_{11}^\mp(z)$.
\begin{prop}
The following hold true in $\F[[z, z^{-1}]]$.
\begin{enumerate}
\item[i.]  For every $p\in \Z-\{2\}$,
\be\label{eq:GG+GG-} \frac{G_{10}^+(zq^p) G_{11}^+(zq^{-p}) - G_{10}^-(z^{-1}q^{-p}) G_{11}^-(z^{-1}q^{p})}{q-q^{-1}} = \frac{[2]_q [p]_q}{[p-2]_q}\left [  \delta\left (zq^{2-p}\right )  -\delta\left (z q^{p-2}\right ) \right ] \, .\ee
In particular, when $p=1$, we have
\be \frac{G_{10}^+(zq) G_{11}^+(zq^{-1}) - G_{10}^-(z^{-1}q^{-1}) G_{11}^-(z^{-1}q)}{q-q^{-1}} = [2]_q \left [ \delta\left (zq^{-1} \right ) -  \delta\left (zq\right ) \right ] \, .\ee
 If $p=2$, we have instead
\be\frac{G_{10}^+(zq^2)G_{11}^+(zq^{-2}) - G_{10}^-(z^{-1}q^{-2}) G_{11}^-(z^{-1}q^2) }{(q-q^{-1})^2} = [2]_q^2 \left [\delta\left (z\right )-\delta'\left (z\right ) \right ]\,.\ee
\item[ii.] Similarly,
\be
\frac{G_{11}^+(zq^{-2})^2 - G_{11}^-(z^{-1}q^2)^2}{(q-q^{-1})^2} = \frac{2q^{-2}[2]_q}{q-q^{-1}} \delta\left (z\right ) + [2]_q^2 \delta'\left (z\right )\,.
\ee
\end{enumerate}
\end{prop}
\begin{proof}
In each case, it suffices to determine the partial fraction decomposition of the l.h.s and to apply lemma \ref{lem:ratdelta} to get the desired result.
\end{proof}

\newcommand{\etalchar}[1]{$^{#1}$}
\providecommand{\bysame}{\leavevmode\hbox to3em{\hrulefill}\thinspace}
\providecommand{\MR}{\relax\ifhmode\unskip\space\fi MR }
\providecommand{\MRhref}[2]{%
  \href{http://www.ams.org/mathscinet-getitem?mr=#1}{#2}
}
\providecommand{\href}[2]{#2}

%
%
\end{document}